%% file: uw-ethesis.tex
\documentclass[letterpaper,12pt,titlepage,oneside,final]{book}
 
\newcommand{\href}[1]{#1} 

\usepackage{ifthen}
\newboolean{PrintVersion}
\setboolean{PrintVersion}{false} 

\usepackage{amsmath,amssymb,amstext,amsthm} 
\usepackage[pdftex]{graphicx} 

\usepackage[pdftex,letterpaper=true,pagebackref=false]{hyperref} 
\hypersetup{
    plainpages=false,       
    pdfpagelabels=true,     
    bookmarks=true,         
    unicode=false,          
    pdftoolbar=true,        
    pdfmenubar=true,        
    pdffitwindow=false,     
    pdfstartview={FitH},    
    pdftitle={uWaterloo\ LaTeX\ Thesis\ Template},    
    pdfnewwindow=true,      
    colorlinks=true,        
    linkcolor=black,         
    citecolor=black,        
    filecolor=black,      
    urlcolor=black           
}
\ifthenelse{\boolean{PrintVersion}}{   
\hypersetup{	
    citecolor=black,%
    filecolor=black,%
    linkcolor=black,%
    urlcolor=black}
}{} 

\let\origdoublepage\cleardoublepage
\newcommand{\clearemptydoublepage}{%
  \clearpage{\pagestyle{empty}\origdoublepage}}
\let\cleardoublepage\clearemptydoublepage

\newtheorem{theorem}{Theorem}[section]
\newtheorem{proposition}[theorem]{Proposition}
\newtheorem{lemma}[theorem]{Lemma}
\newtheorem{corollary}[theorem]{Corollary}
\newtheorem{fact}[theorem]{Fact}

\theoremstyle{definition}
\newtheorem{definition}[theorem]{Definition}

\theoremstyle{remark}
\newtheorem{remark}[theorem]{Remark}
\newtheorem{example}[theorem]{Example}

\def\Ind{\setbox0=\hbox{$x$}\kern\wd0\hbox to 0pt{\hss$\mid$\hss} \lower.9\ht0\hbox to 0pt{\hss$\smile$\hss}\kern\wd0} 

\def\Notind{\setbox0=\hbox{$x$}\kern\wd0\hbox to 0pt{\mathchardef \nn=12854\hss$\nn$\kern1.4\wd0\hss}\hbox to 0pt{\hss$\mid$\hss}\lower.9\ht0 \hbox to 0pt{\hss$\smile$\hss}\kern\wd0} 

\def\ind{\mathop{\mathpalette\Ind{}}}

\newcommand{\compcent}[1]{\vcenter{\hbox{$#1\circ$}}} 

\newcommand{\comp}{\mathbin{\mathchoice 
{\compcent\scriptstyle}{\compcent\scriptstyle} 
{\compcent\scriptscriptstyle}{\compcent\scriptscriptstyle}}} 

\input xy
\xyoption{all}

\numberwithin{equation}{section}

\def \x {\bar x}

\def \d {\delta}

\def \dd {\partial}
\def \D {\Delta}
\def \t {\theta}
\def \T {\Theta}
\def \l {\langle}
\def \r {\rangle}
\def \L {\Lambda}
\def \I {\mathcal I}
\def \P {\mathcal P}
\def \V {\mathcal V}
\def \NN {\mathbb N}
\def \QQ {\mathbb Q}
\def \U {\mathbb U}
\def \C {\mathcal C}
\def \la {\langle}
\def \ra {\rangle}
\def \al {\alpha}
\def \DD {\mathcal D}
\def \ta {\tau_{\DD/\D}}
\def \ld {\ell_s}
\def \G {\mathcal G}
\def \H {\mathcal H}
\def \Ga {\Gamma}

\def \s {\sigma}

\def \ZZ {\mathbb Z}
\def \et {d_{D/\D}}
\def \w {\omega}
\def \M {\mathcal M}
\def \LL {\mathcal L}

\begin{document}

\input{uw-ethesis-frontpgs}


\chapter*{Introduction}
\addcontentsline{toc}{chapter}{Introduction}

Just as algebraic geometry can be viewed as the study of the set of solutions to systems of polynomial equations, \emph{differential algebraic geometry} studies the set of solutions to systems of differential algebraic equations. Our approach to the subject is in the spirit of Weil, with model theory providing the analogue of the universal domain: our \emph{differential algebraic varieties} are taken to live inside a large \emph{differentially closed field} of characteristic zero. By now, it is well known that the interaction between model theory and differential algebraic geometry goes well beyond this. The theory of differentially closed fields is arguably the most interesting example of an $\w$-stable theory so far. For example, the study of \emph{finite dimensional} differential algebraic varieties has led to a variety of interesting examples, and detailed descriptions, of notions from general stability (see for example \cite{Ma3} and \cite{PiT}). On the other hand, the machinery of $\omega$-stable theories has been succesfully implemented in diverse aspects of differential algebra, such as differential Galois theory and the theory of differential algebraic groups (see \cite{Po2}, \cite{Pi} and \cite{Pi4}).  The model theory of finite dimensional differential algebraic varieties is behind Hrushovki's celebrated proof of the Mordell-Lang conjecture for function fields in characteristic zero~\cite{Hru2}.

While much has been done in the finite dimensional setting, it was not until the last decade that great interest in describing \emph{infinite dimensional} differential algebraic varieties, from a model theoretic perspective, began \cite{Moo}. These descriptions are mostly of interest in the partial case, i.e., in differential fields with several commuting derivations, since in this setting one has a rich supply of infinite dimensional definable sets. One example is the Heat equation $\d_1x=\d_2^2 x$, which was studied in depth by Suer in \cite{So2}. This thesis can be viewed as a next step towards a better understanding of infinite dimensional differential algebraic geometry. 

The basic approach we have in mind is the following: just as one thinks of ordinary (i.e., one derivation) differential algebraic geometry as adding structure to algebraic geometry (namely that induced by the derivation), we view differential algebraic geometry in two commuting derivations as adding structure to ordinary differential algebraic geometry, and so on. Therefore, instead of treating pure fields as our base theory to which we then add some new structure, we treat differential fields as our base theory. This way of thinking of differential algebraic geometry is somewhat implicit in the work of Kolchin \cite{Ko2}; however, by formalizing this approach and building a theory systematically upon it, we are able to tackle several problems on the subject.

In the remainder of this introduction we will outline the contents of this thesis. More thorough introductions to these topics appear as the first sections of each chapter.

In the preliminary Chapter 1, we present the differential algebraic background needed for the rest of this thesis. We make special emphasis on the \emph{differential type} of a definable set, as this notion yields a way to classify differential algebraic varieties according to their \emph{transcendental dimension} (finite dimensional ones are precisely those of differential type zero).

\noindent{\bf Relative D-varieties and D-groups.} The first step towards understanding finite dimensional differential algebraic varieties is simply to observe that, since their (differential) function field has finite transcendence degree, they can be generically described by algebraic polynomial equations together with first order differential equations of a very particular form. Loosely speaking, one obtains that all the derivations can be completely described in terms of algebraic relations. This was formalized through the concept of \emph{algebraic D-variety} by Buium in \cite{Buium}, and later explored in model theoretic terms by Pillay in \cite{PiT} and~\cite{Pi7}. 

The situation is quite different for infinite dimensional differential algebraic varieties, their function fields have infinite transcendence degree and thus we need a way to measure how ``transcendental'' it really is. This is where the differential type comes into play, it gives us a way to differentiate between infinite dimensional differential algebraic varieties according to their (differential) transcendental properties. In loose terms, the differential type gives us the largest subset of the derivations that can be completely described in terms of the others. In particular, the differential type is a natural number between zero and the number of derivations.

We formalize the above ideas in Chapter~\ref{chapro} by introducing the category of \emph{relative D-varieties}. Briefly, a \emph{relative D-variety} is given by a partition $\DD\cup\D$ of the total set of derivations, say $\Pi$, and a pair $(V,s)$ where $V$ is a $\D$-algebraic variety and $s$ is a $\D$-algebraic section of a certain torsor of the $\D$-tangent bundle of $V$ called the \emph{relative prolongation of $V$} and denoted by $\ta V$ (see Definition~\ref{fipro}). The section $s$ needs to satisfy a certain \emph{integrability condition} (see Definition \ref{dvar}). Once in the right setting, we are able to prove a number of properties, including the following important characterization (Proposition \ref{finteo}):

\noindent {\bf Theorem A.} {\it Let $(K,\Pi)$ be a differential field in $m$ commuting derivations $\Pi$, and suppose $W$ is an irreducible $\Pi$-algebraic variety defined over $K$ of $\Pi$-type $\ell<m$. Then there exist}
\begin{itemize}
\item {\it $m$ linearly independent linear combinations $\{D_1,\dots,D_m\}$ of the derivations $\Pi$ over the constant field of $(K,\Pi)$,}
\item {\it a $\{D_1,\dots,D_\ell\}$-algebraic variety $V$ defined over $K$ of $\{D_1,\dots,D_\ell\}$-type $\ell$, and}
\item {\it $\{D_1,\dots,D_\ell\}$-polynomials $s_1,\dots,s_{m-\ell}$,}
\end{itemize}
{\it such that $W$ is $\Pi$-birationally equivalent to} $$\{a\in V: s_i(a)=D_{\ell+i}(a), \, i=1,\dots,m-\ell\}.$$

As we have stated the theorem above, relative D-varieties did not appear explicitly. However, if we let $\D=\{D_1,\dots,D_\ell\}$, $\DD=\{D_{\ell+1},\dots,D_{m}\}$ and $s=(\operatorname{Id},s_1,\dots,s_{m-\ell})$, then $(V,s)$ is a relative D-variety w.r.t. $\DD/\D$, and the displayed $\Pi$-algebraic variety in the conclusion is usually denoted by $(V,s)^\#$ and called the sharp points of $(V,s)$. The theorem says that every irreducible $\Pi$-algebraic variety is $\Pi$-birationally equivalent to the sharp points of a relative D-variety of the appropriate form. This is a precise generalization of the characterization of finite dimensional differential algebraic varieties in terms of algebraic D-varieties.

In Theorem A the birational equivalence between $W$ and $(V,s)^\#$ might indeed be cutting some proper subvarieties, and hence it might not be an isomorphism. However, in the presence of a group structure this can be refined. We consider the group objects in the category of relative D-varieties to obtain the notion of \emph{relative D-group}. Then, by applying a general version of Weil's group chunk theorem, we are able to show that in the case when $W$ is a connected differential algebraic group, then the birational equivalence becomes an isomorphism of groups between $W$ and the set of sharp points of a relative D-group (see Theorem \ref{co}). This result is used in Chapter \ref{chapgal} when we establish a connection between our two differential Galois theories.

\noindent {\bf Galois theory for partial differential fields.} As in the algebraic case, differential Galois theory is the study of differential field extensions with well behaved group of automorphisms. The subject has its roots in the work of Picard and Vessiot at the end of the 19th century when they studied the solutions of linear differential equations. After that, several formalizations and generalizations of their work were considered, but it was not until Kolchin's \emph{strongly normal extensions} entered the picture that differential Galois theory was systematically built as a whole new subject \cite{Ko}.

In Kolchin's work on strongly normal extensions the field of constants played an essential role. However, in \cite{Pi}, Pillay succesfully replaced the constants for an arbitrary definable set and introduced the notion of a \emph{generalized strongly normal extension}. There, he showed that these extensions have a \emph{good} Galois theory. Even though Pillay's setting was ordinary differential fields, his results can be extended more or less immediately to the partial case, as long as the definable sets considered are finite dimensional. One of the purposes of Chapter 3 is to extend his results to the possibly infinite dimensional setting. 

In algebraic Galois theory the Galois extensions are associated to polynomial equations. In the Picard-Vessiot linear differential Galois theory this remains true: the Picard-Vessiot extensions are the Galois extensions associated to homogeneous linear differential equations. In \cite{Pi2}, Pillay shows that, under certain assumptions on the base field, his generalized strongly normal extensions are the Galois extensions associated to certain differential equations on (possibly non-linear) algebraic groups. We develop a Galois theory associated to \emph{relative logarithmic differential equations} on relative D-groups (Section~\ref{galex}), associating to each such equation a generalized strongly normal extension. We show (Proposition \ref{gal} and Corollary \ref{corresp}):

\noindent {\bf Theorem B.} {\it The Galois group of the generalized strongly normal extension $L/K$ associated to a relative logarithmic differential equation, on a relative D-group $(G,s)$, is of the form $(H,t)^\#$ where $(H,t)$ is a relative D-subgroup of $(G,s)$ defined over $K$. Moreover, there is a Galois correspondence between the intermediate differential fields of $L/K$ and the relative D-subgroups of $(H,t)$ defined over $K$}.

The above theory simultaneously generalizes the parametrized Picard-Vessiot theory of Cassidy and Singer \cite{PM}, and Pillay's finite dimensional Galois theory on algebraic D-groups~\cite{Pi2}.

While Galois extensions corresponding to relative logarithmic differential equations are instances of generalized strongly normal extensions, it is not known in general (i.e., in the possibly infinite dimensional case) if the converse is true. However, we can show the following generalization of Remark 3.8 of \cite{Pi2} (see Theorem \ref{main} below):

\noindent {\bf Theorem C.} {\it Suppose $(K,\Pi)$ is a differential field with $m$ commuting derivations and $\D$ is a set of fewer than $m$ linearly independent linear combinations of the derivations in $\Pi$ over the constant field of $(K,\Pi)$. Suppose, moreover, that $L$ is a generalized strongly normal extension of $K$ whose $\Pi$-type is bounded by $\D$ (see Definition~\ref{boundty}) and such that $K$ is $\D$-closed. Then $L/K$ is the Galois extension associated to a logarithmic differential equation on a relative D-group.}

\noindent {\bf Geometric axioms for $DCF_{0,m}$.} Axiomatizations of the existentially closed models (i.e., models in which each consistent system of equations has a solution) of several theories of fields with operators have been formulated in terms of algebro geometric objects. For difference fields this was done by Chatzidakis and Hrushovski in \cite{CH}, for ordinary differential fields we have the characterization of Pierce and Pillay in \cite{PiPi}, and, generalizing both, for the theory of fields with free operators considered by Moosa and Scanlon we have \cite{MS}. These \emph{geometric axiomatizations} usually involve a scheme of axioms referring to certain dominant projections of irreducible algebraic varieties.

In \cite{Pierce2}, Pierce attempted to extend the geometric axioms for ordinary differentially closed fields (known as the Pierce-Pillay axioms) to the partial setting in a very natural way; however, he found a mistake in his argument. Essentially, the problem is that the commuting of the derivations is too strong of a condition to impose using simple algebraic relations. Nonetheless, in \cite{Pierce}, he does find a geometric axiomatization, but his new axioms do not formally specialize to the Pierce-Pillay axioms, and ultimately have a very different flavour.

In Chapter 4, using the approach discussed at the begining of the introduction, we treat differential fields with $\ell +1$ commuting derivations as adding one derivation to differential fields with $\ell$ derivations. Then we are able to characterize differentially closed fields in $\ell +1$ derivations in terms of differential algebro geometric objects in $\ell$ derivations in precisely the same spirit as the Pierce-Pillay axioms (see Theorem~\ref{charactheo}). 

\noindent {\bf Theorem D.} {\it $(K,\D\cup\{D\})$ is differentially closed if and only if}
\begin{enumerate}
\item [(i)] {\it $(K,\D)$ is differentially closed}
\item [(ii)] {\it Suppose $V$ and  $W$ are irreducible affine $\D$-algebraic varieties defined over $K$ such that $W\subseteq \tau_{D/\D} V$ and $W$ projects $\D$-dominantly onto $V$. If $O_V$ and $O_W$ are nonempty $\D$-open subsets of $V$ and $W$ respectively, defined over $K$, then there exists a $K$-point $a\in O_V$ such that $(a,Da) \in O_W$. }
\end{enumerate}

However, some complications arise in this setting. It is a well known open problem as to whether irreducibility in differential algebraic geometry is first order expressible (it is in algebraic geometry). A similar issue arises with the property of a projection being $\D$-dominant. Thus, it is not clear if our characterization above yields, as it is, a first order axiomatization. We bypass these issues, of differential irreducibility and dominance, by reformulating our characterization using the differential algebra machinery of \emph{characteristic sets of prime differential ideals}. Given a finite set $\L$ of $\D$-polynomials, we let $\V(\L)$ be the zeros of $\L$ and $\V^*(\L)$ be the zeros of $\L$ that are not zeros of the \emph{initials} or \emph{separants} of $\L$ (see Section~\ref{ondiff} for definitions). We are able to prove the following (Theorem \ref{characprime}):

\noindent {\bf Theorem E.} {\it $(K,\D\cup\{D\})$ is differentially closed if and only if}
\begin{enumerate}
\item [(i)] $(K,\D)$ {\it is differentially closed.}
\item [(ii)] {\it Suppose $\L$ and $\Ga$ are characteristic sets of prime $\D$-ideals of the $\D$-polynomial rings $K\{x\}_{\D}$ and $K\{x,y\}_{\D}$ respectively, such that $$\V^*(\Ga)\subseteq \V(f,\, \et f:\, f\in \L)$$ (see Section~\ref{dider} for the definition of the operator $\et$). Suppose $O$ is a nonempty $\D$-open subset of $\V^*(\L)$ defined over $K$ such that the projection of $\V^*(\Ga)$ to $\V(\L)$ contains $O$. Then there is a $K$-point $a \in \V^*(\L)$ such that $(a,D a) \in \V^*(\Ga)$}.
\end{enumerate}

By the results of Tressl from \cite{Tr}, one can express in a first order fashion when a finite set of differential polynomials is a characteristic set of a prime differential ideal. Thus, Theorem E yields a geometric axiomatization of partial differentially closed fields.

\noindent {\bf Partial differential fields with an automorphism.} For over fifteen years, it has been known  that the class of existentially closed ordinary differential fields with an automorphism has a first order axiomatization (see for example \cite{Bu}). However, the proof does not extend to the theory of partial differential fields with an automorphism. In Chapter 5 we settle this open problem.

In \cite{GR}, Guzy and Rivi\`ere observed that the existentially closed partial differential fields with an automorphism are characterized by a certain differential algebro geometric condition (see Fact 2.1), analogous to Theorem D above. However, due to the same definable questions there, it remains open as to whether their characterization is first order. Motivated by the ideas used in Chapter 4, we apply the machinery of characteristic sets of prime differential ideals in Chapter 5 to prove (c.f. Corollary \ref{mc}):

\noindent {\bf Theorem F.} {\it The class of existentially closed partial differential fields with an automorphism is axiomatizable.}

In the usual manner, various basic model theoretic properties of this model companion are derived. In addition, we extend the Pillay-Ziegler \cite{PZ} theory of differential and difference modules to this setting and prove an appropriate form of the \emph{canonical base property} for finite dimensional types. As a consequence we obtain the following Zilber dichotomy for finite dimensional types (Corollary~\ref{zildi}):

\noindent{\bf Theorem G.} {\it If $p$ is a finite dimensional type with $SU(p)=1$, then $p$ is either one-based or almost internal to the intersection of the fixed and constant fields (see Section~\ref{zidi} for the definitions).}

We expect that the finite dimensional assumption is not necessary, and that by  extending the methods of Moosa, Pillay, Scanlon \cite{Moo}, to this setting, we can remove this assumption. However, we have not checked this in detail and thus omit it from this thesis.

\chapter{Preliminaries}\label{chapre}

This chapter is intended to provide the reader with the necessary background in differential algebra and differential algebraic geometry. We include also a very cursory introduction to model theory, while the more advanced model theoretic notions are reviewed by specializing them to the theory of differentially closed fields. For the abstract definitions we suggest \cite{Ma2} and \cite{Pi8}.  We do take this opportunity to sometimes introduce new and convenient terminology, for example that of \emph{bounding} and \emph{witnessing} the differential type in Section~\ref{polype} below.

\section{Model theory}

This brief section is intended for the reader that is not at all familiar with model theory. While basic model theory is formally a prerequisite for this thesis, we hope that an ignorance on the subject will not be an insurmountable obstacle to reading this thesis, especially because we will be focusing on a few very particular theories and will not work in the full abstraction that model theory allows. Nevertheless, a basic understanding of the model theoretic approach to mathematics is necessary, and this is what we present now. For a detailed and contemporary introduction to model theory we suggest David Marker's textbook \cite{Ma2}. Our discussion here is informed by \cite{Moosa}.

A \emph{structure} $\M$ consists of a nonempty underlying set $M$, called the universe of $\M$, together with:
\begin{enumerate}
\item [(i)] a set $\{c_i: i\in I_C\}$ of distinguised elements of $M$, called \emph{constants}
\item [(ii)] a set $\{f_i:M^{n_i}\to M: i\in I_F\}$ of distinguished maps, called \emph{basic functions}
\item [(iii)] a set $\{R_i\subseteq M^{k_i}: i\in I_R\}$ of distinguished subsets of cartesian powers of $M$, called \emph{basic relations}
\end{enumerate}

For example any unitary ring $R$ is made into a structure by considering $(R,0,1,+,-,\times)$ where $0$ and $1$ are the constants, $+$ and $\times$ are the basic binary functions, $-$ is the basic unary function, and there are no basic relations.

Given a structure $\M$, the \emph{definable sets} in $\M$ are the subsets of various cartesian powers of $M$ that are obtained in a finitary manner from the constants, basic relations, and graphs of basic functions, by closing off under 
\begin{itemize}
\item intersections,
\item complements,
\item images under coordinate projections, and
\item fibers of coordinate projections.
\end{itemize}

The definable sets have a syntactic characterization that we do not explain in detail here, but it suffices to say that definable sets are \emph{defined by} first order formulas in a \emph{language} that has symbols for the basic constants, functions, and relations, as well as the standard logical symbols. These latter correspond to the above operations on definable sets in the natural way: intersections of definable sets corrrespond to \emph{conjunctions} of formulae, complements to \emph{negation}, images under projections to \emph{existential quantifications}, and taking fibers of projections correspond to \emph{specializations} or \emph{naming parameters}. Note that by combining these we get also \emph{disjunctions} and \emph{universal quantifications}.

In general, the class of definable sets can be very wild; for example, there is a rigorous sense in which all of mathematics can be definably encoded  (via G\"odel's numbering) in the ring $(\ZZ,0,1,+,-,\times)$. At the opposite extreme we have tame mathematics; for example, a theorem of Chevalley (and also Tarski) implies that in $(K,0,1,+,-,\times)$, where $K$ is an algebraically closed field, the definable sets are precisely the Zariski-constructible sets. In particular, the third operation (existential quantification) becomes superfluous: this goes by the term \emph{quantifier elimination} in model theory.

One central aspect of model theory is the study of definable sets in a given structure. Another aspect is the study of classes of structures satisfying a given set of \emph{first order axioms} in a fixed language. For example, in the language of unitary rings as above, we can express the axioms of algebraically closed fields. Each axiom is a first order finitary formula without free variables. One such collection of axioms consists of the axioms of fields together with a formula of the form
$$\forall a_{n-1}\, \cdots \, \forall a_1\, \forall a_0\, \exists x\;  x^n+a_{n-1}x^{n-1}+\cdots+a_1x+a_0=0,$$
for each $n>0$. This collection of axioms is usually denoted by $ACF$ and referred to as the \emph{theory of algebraically closed fields}. A structure in which these axioms hold is a \emph{model} of the theory. Of course, there are deep connections between the study of theories and their models, and the study of definable sets in a given structure.

We end this section by listing, more or less as a glossary, some basic model theoretic notation. Let $\M$ be a structure, $\LL$ be a language having the appropriate symbols, and  $T$ be an $\LL$-theory (i.e., a set of axioms in the language $\LL$).
\begin{itemize}
\item $\M\models T.$ This means that $\M$ is a model  of the theory $T$; that is, all the axioms in $T$ are true in $\M$.
\item $\M\models \phi(a).$ Here $a$ is a (finite) tuple from $M$, $\phi(x)$ is a first order $\LL$-formula in free variables $x$, and the notation means that $a$ realises $\phi$ in $\M$. In other words, if $X\subseteq M^{|x|}$ is the set defined by $\phi$ then $a\in X$.
\item $X(A)=\{a\in A^n: a\in X\}$. Here $A\subseteq M$ and $X$ is a definable subset of $M^n$. 
\item $S_n(A).$ Here $A\subseteq M$. This means the space of \emph{complete $n$-types over $A$}. A complete $n$-type $p$ over $A$ is a set of $\LL$-formulas over $A$ in $n$ variables such that each finite subset of $p$ is realised in $\M$ and $p$ is maximal with this property.
\item $a\models p.$ Here $a$ is a tuple from $M$ and $p$ is a complete type over a subset of $M$. This means that $a$ realises in $\M$ all the formulas from $p$. 
\item $p^\M=\{a\in M: a\models p\}$. Here $p$ is a complete type over a subset of $M$.
\item $tp(a/A)=\{\phi(x): \phi \text{ is an } \LL\text{-formula over } A \text{ and } \M\models \phi(a)\}.$ Here $a$ is a tuple from $M$ and $A\subseteq M$. This is the \emph{complete type of $a$ over $A$}. Clearly, $tp(a/A)\in S_n(A)$ where $n=|a|$.
\item $\operatorname{dcl}(A)=\{a\in M:\{a\}\text{ is an } $A$\text{-definable set} \}.$ This is the \emph{definable closure of $A$}. Note that $a\in \operatorname{dcl}(A)$ iff there is an $\LL$-formula over $A$ such that $a$ is the unique realisation in $\M$.
\item $\operatorname{acl}(A)=\{a\in \M: a \text{ is contained in a finite } A\text{-definable set}\}.$ This is the model theoretic \emph{algebraic closure of $A$}.
\item $A$ \emph{is interdefinable (interalgebraic) with B over $C$.} Here $A,B,C$ are subsets of $M$. This means that $A\cup C$ has the same definable (algebraic) closure as $B\cup C$.
\end{itemize}

\section{Differential algebra}\label{ondiff}

In this section we review the differential algebraic notions required for this thesis. For a complete treatment of differential algebra we refer the reader to \cite{Ko}.

Let $R$ be a ring (unital and commutative) and $S$ a ring extension. An additive map $\d:R\to S$ is called a \emph{derivation} if it satifies the Leibniz rule; i.e., $\d(ab)=\d(a)b+a\d(b)$. A ring $R$ equipped with a set of derivations $\D=\{\d_1,\dots,\d_m\}$, $\d_i:R\to R$ for $i=1,\dots,m$, such that the derivations commute with each other is called a \emph{$\D$-ring}. In the case when $R$ is a field we say \emph{$\D$-field}. The notions of \emph{$\D$-homomorphism}, \emph{$\D$-isomorphism} and \emph{$\D$-automorphism} are defined naturally.

Let $R$ be a $\D$-ring. The ring of \emph{$\D$-constants of} $R$ will be denoted by $$R^{\D}=\{a\in R:\, \d a=0\text{ for all } \d\in \D\}.$$ If $S$ is a $\D$-ring extension of $R$ and $A\subseteq S$, we denote by $R\{A\}_\D$ the $\D$-subring of $S$ generated by $A$ over $R$. In the case when $S$ is a field, we denote by $R\l A\r_\D$ the $\D$-subfield of $S$ generated by $A$ over $R$.

The set of \emph{derivative operators} is defined as the commutative monoid $$\T_\D= \{\d_m^{e_m}\cdots \d_1^{e_1}: \, e_i<\omega\}.$$ If $x=(x_1,\dots,x_n)$ is a tuple of (differential) indeterminates, we set $\T_\D x=\{\t x_i: \t\in \T_\D, i=1,\dots,n\}$ and call these the corresponding \emph{algebraic indeterminates}. The ring of \emph{$\D$-polynomials over} $R$ is defined as $R\{x\}_\D=R[\T_\D x]$. One can equip $R\{x\}_\D$ with the structure of a $\D$-ring by defining $$\d_j(\d_m^{e_m}\cdots \d_1^{e_1}x_i)=\d_m^{e_m}\cdots \d_j^{e_j+1}\cdots \d_1^{e_1}x_i$$ and then extending to all of $R\{x\}_\D$ using additivity and the Leibniz rule. In the case when $R$ is an integral domain, we define the field of \emph{$\D$-rational functions over} $R$, denoted by $R\l x\r_\D$, as the field of fractions of $R\{x\}_\D$, or equivalently $R\l x\r_\D=R(\T_\D x)$. By the quotient rule, $R\l x\r_\D$ admits a canonical $\D$-field structure.

When the set of derivations $\D$ is understood, we will often omit the subscripts in the differential notation. For example, $\T$ and $R\{x\}$ stand for $\T_\D$ and $R\{x\}_\D$, respectively.

Given a $\D$-ring $R$, an ideal $I$ of $R\{x\}$ is said to be a \emph{$\D$-ideal} (or differential ideal) if $\d f\in I$ for all $f\in I$ and $\d\in \D$. Given a set $B \subseteq R\{x\}$, the ideal and the $\D$-ideal generated by $B$ will be denoted by $(B)$ and $[B]$, respectively. 

\begin{fact}[Differential basis theorem, \cite{Ko}, Chap. III, \S 4]\label{difba}
Suppose $R$ is a $\QQ$-algebra that satisfies the ACC on radical $\D$-ideals. If $I$ is a radical $\D$-ideal of $R\{x\}$, then there are $f_1,\cdots, f_s\in R\{x\}_\D$ such that $I=\sqrt{[f_1,\dots,f_s]}$.
\end{fact}

A set $X\subseteq R^n$ is said to be \emph{$\D$-closed} if there is a set $B\subseteq R\{x\}$ such that $X=\{a\in R^n: f(a)=0 \text{ for all } f\in B\}$. The differential basis theorem shows that, if $R$ is a field, the $\D$-closed subsets of $R^n$ are the closed sets of a Noetherian topology, called the \emph{$\D$-topology} of $R^n$.

\noindent {\bf Characteristic sets.} In the remainder of this section we discuss the somewhat technical, but very useful, notion of characteristic sets. The motivation here is that despite the Noetherianity of the $\D$-topology, it is not necessarily the case that prime $\D$-ideals are finitely generated as $\D$-ideals. However, they are determined by a finite subset in a canonical way.

The \emph{canonical ranking} on the algebraic indeterminates $\T x$ is defined by 
\begin{displaymath}
\d_m^{e_m}\cdots \d_1^{e_1}x_i< \d_m^{r_m}\cdots \d_1^{r_1}x_j \iff \left(\sum e_k,i,e_m,\dots,e_1\right)<\left(\sum r_k,j,r_m,\dots,r_1\right)
\end{displaymath}
in the lexicographical order. Let $f\in R\{x\}\setminus R$. The \emph{leader} of $f$, $v_f$, is the highest ranking algebraic indeterminate that appears in $f$. The \emph{order} of $f$, $ord(f)$, is defined as $\sum e_k$ where $v_f=\d_m^{e_m}\cdots\d_1^{e_1}x_i$. The \emph{degree} of $f$, $d_f$, is the degree of $v_f$ in $f$. The \emph{rank} of $f$ is the pair $(v_f,d_f)$. If $g\in R\{x\}\setminus R$ we say that $g$ has \emph{lower rank} than $f$ if $rank(g)<rank(f)$ in the lexicograpical order. By convention, an element of $R$ has lower rank than all the elements of $R\{x\}\setminus R$.

Let $f\in R\{x\}\setminus R$. The \emph{separant} of $f$, $S_f$, is the formal partial derivative of $f$ with respect to $v_f$, that is $$S_f:=\frac{\partial f}{\partial v_f}\in R\{x\}.$$ The \emph{initial} of $f$, $I_f$, is the leading coefficient of $f$ when viewed as a polynomial in $v_f$, that is, if we write $$f=\sum_{i=0}^{d_f}g_i(v_f)^i$$ where $g_i\in R\{x\}$ and $v_{g_i}<v_f$ then $I_f=g_{d_f}$. Note that both $S_f$ and $I_f$ have lower rank than $f$. 

\begin{definition}
Let $f,g\in R\{x\}\setminus R$. We say $g$ is \emph{partially reduced} with respect to $f$ if no (proper) derivative of $v_f$ appears in $g$. If in addition the degree of $v_f$ in $g$ is less than $d_f$ we say that $g$ is \emph{reduced} with respect to $f$.
\end{definition}

A finite set $\L=\{f_1,\dots,f_s\}\subset R\{x\}\setminus R$ is said to be \emph{autoreduced} if for all $i\neq j$ we have that $f_i$ is reduced with respect to $f_j$. We will always write autoreduced sets in order of increasing rank, i.e., $rank(f_1)<\dots<rank(f_s)$. The canonical ranking on autoreduced sets is defined as follows: $\{g_1,\dots,g_r\}<\{f_1,\dots,f_s\}$ if and only if, either there is $i\leq r,s$ such that $rank(g_j)=rank(f_j)$ for $j<i$ and $rank(g_i)<rank(f_i)$, or $r>s$ and $rank(g_j)=rank(f_j)$ for $j\leq s$.

\begin{fact}[\cite{Ko}, Chap. I, \S 10]\label{re1} 
Every $\D$-ideal $I$ of $R\{x\}$ contains a lowest ranking autoreduced set. Any such set is called a characteristic set of $I$.
\end{fact} 

\begin{remark}\label{contin}
If $\L$ is a characteristic set of $I$ and $g\in R\{x\}\setminus R$ is reduced with respect to all $f\in\L$, then $g\notin I$. Indeed, if $g$ were in $I$ then $$\{f\in I: rank(f)<rank(g)\}\cup\{g\}$$ would be an autoreduced subset of $I$ of lower rank than $\L$, contradicting the minimality of $\L$.
\end{remark}

Suppose $\L=\{f_1,\dots,f_s\}$ is a subset of $R\{x\}\setminus R$. We let $\displaystyle H_{\L}:= \prod_{i=1}^s I_{f_i} \, S_{f_i}$. Also, for each $\t\in\T$, we let $(\L)_\t$ denote the ideal of $R\{x\}$ generated by the $f_i$'s and their derivatives whose leaders are strictly below $\t$. If $I$ is an ideal of $R\{x\}$ and $g\in R\{x\}$, we denote the \emph{saturated ideal of $I$ over} $g$ by
\begin{displaymath}
I:g^{\infty}=\{f\in R\{x\} \,: \, g^\ell f\in I \, \text{ for some } \ell\geq 0\}.
\end{displaymath}

\begin{definition}\label{coh}
Let $\L=\{f_1,\dots,f_s\}$ be a subset of $R\{x\}$. The set $\L$ is said to be {\it coherent} if the following two conditions are satisfied:
\begin{enumerate}
\item $\L$ is autoreduced.
\item For $i\neq j $, suppose there are $\t_i$ and $\t_j\in \T$ such that $\t_iv_{f_i}=\t_jv_{f_j}=v$ where $v$ is the least such in the ranking. Then 
\begin{displaymath}
S_{f_j}\t_if_i-S_{f_i}\t_jf_j\in (\L)_{\t}:H_{\L}^{\infty}.
\end{displaymath}
\end{enumerate}
\end{definition}

Even though differential ideals of $R\{x\}$ are not in general generated by their characteristic sets, when $R$ is a field of characteristic zero, prime differential ideals are determined by these. Moreover, we have a useful criterion to determine when a finite set of differential polynomials is the characteristic set of a prime differential ideal.

\begin{fact}[Rosenfeld's criterion, \cite{Ko}, Chap. IV, \S 9]\label{rosen}
Let $K$ be a $\D$-field of characteristic zero. If $\L$ is a characteristic set of a prime $\D$-ideal $\P$ of $K\{x\}$, then $$\P=[\L]:H_{\L}^{\infty},$$ $\L$ is coherent, and $(\L):H_{\L}^{\infty}$ is a prime ideal not containing a nonzero element reduced with respect to all $f\in\L$. Conversely, if $\L$ is a coherent subset of $K\{x\}$ such that $(\L):H_{\L}^{\infty}$ is prime and does not contain a nonzero element reduced with respect to all $f\in\L$, then $\L$ is a characteristic set of a prime $\D$-ideal of $K\{x\}$.
\end{fact}

\begin{example}
Suppose $K$ is a $\D$-field of characteristic zero and $a_1,\dots,a_m\in K$ satisfy $\d_ja_i=\d_ia_j$ for $i,j=1\dots,m$. We use Rosenfeld's criterion to show that the $\D$-ideal
$$[\d_ix-a_i x: i=1,\dots,m]\subseteq K\{x\}$$
is prime. Let $\L=\{\d_i x-a_ix:i=1,\dots,m\}$. Note that $H_\L=1$, and that the assumption on the $a_i$'s imply that $\L$ is coherent. Clearly $(\L)$ is a prime ideal of $K\{x\}$. If $g\in K\{x\}$ is reduced with respect to all the elements of $\L$, then $g$ is an (algebraic) polynomial. Hence, $(\L)$ does not contain a nonzero element reduced with respect to the elements of $\L$. By Rosenfeld's criterion, $\L$ is a characteristic set of the prime $\D$-ideal $$[\L]:H_\L^\infty=[\d_ix-a_i x:i=1,\dots,m].$$
\end{example}

\section{Differentially closed fields}

In this section we review the basic model theory of differential fields of characteristic zero with several commuting derivations. We will work in the first order language of partial differential rings with $m$ derivations $$\mathcal L_{m}=\{0,1,+,-,\times\}\cup\{\d_1,\dots,\d_m\},$$ where $\D=\{\d_1,\dots,\d_m\}$ are unary function symbols. The first order $\mathcal{L}_m$-theory $DF_{0,m}$ consists of the axioms of fields of characteristic zero together with axioms asserting that $\D$ is a set of $m$ commuting derivations. 

Note that a $\D$-field $(K,\D)$ is existentially closed (in the sense of model theory) if every system of $\D$-polynomial equations with a solution in a $\D$-field extension already has a solution in $K$. In \cite{Mc}, McGrail showed that the class of existentially closed $\Delta$-fields is axiomatizable; this is the theory $DCF_{0,m}$ of \emph{differentially closed fields with $m$ commuting derivations of characteristic zero}, also called simply \emph{$\D$-closed fields}. We refer the reader to \cite{Mc} for McGrail's original axioms, but let us give an alternative axiomatization given by Tressl in \cite{Tr}. He observed that $(K,\D)\models DCF_{0,m}$ if and only if given a characteristic set $\L=\{f_1,\dots,f_s\}$ of a prime differential ideal of $K\{x\}_{\D}$ then there exists a tuple $a$ from $K$ such that $$f_1(a)=0\land\cdots\land f_s(a)=0\land H_\L(a)\neq 0.$$ In order to conclude that this is indeed a first order axiomatization, Tressl proved the following property (which will play an essential role for us in Chapters~\ref{chapaxioms} and \ref{chapaut}):

\begin{fact}\label{defchar}
Suppose $(K,\D)$ is a $\D$-field of characteristic zero. The condition that ``$\L=\{f_1,\dots,f_s\}$ is a characteristic set of a prime differential ideal of $K\{x\}_\D$'' is a definable property (in the language $\mathcal{L}_m$) of the coefficients of $f_1,\dots,f_s$.
\end{fact}

This means that if $f_1(u,x),\dots, f_s(u,x)$ is a set of $\D$-polynomials over $\QQ$ in the variables $u$ and $x$, then the set $$\{a\in K^{|u|}: \{f_1(a,x),\dots,f_s(a.x)\} \text{ is a characteristic set of a prime $\D$-ideal of } K\{x\}_\D\}$$ is definable in $(K,\D)$.

Fact~\ref{defchar} is essentially an application of Rosenfeld's criterion (Fact~\ref{rosen} above) which reduces the problem to the classical problem of checking primality in polynomial rings in finitely many variables where uniform bounds are well known \cite{Van}.

In Chapter~\ref{chapaxioms} below we will give a new geometric axiomatization for $DCF_{0,m}$.

We now outline the model theoretic properties of $DCF_{0,m}$ that will be used throughout the thesis. For proofs and more details the reader is referred to \cite{Mc} and \cite{So2}. 

\begin{fact} \label{onty} \
\begin{enumerate}
\item [(i)] Every $\D$-field embeds into a $\D$-closed field.
\item [(ii)] If $(K,\D)$ is differentially closed, then the field of constants $K^\D$ is a pure algebraically closed field.
\item [(iii)] $DCF_{0,m}$ eliminates quantifiers; that is, every definable set is a finite union of sets of solutions to systems of $\D$-polynomial equations and inequations. In other words, the definable sets are precisely the $\D$-constructible sets.
\item [(iv)] $DCF_{0,m}$ eliminates imaginaries; that is, for every definable set $X$ with a definable equivalence relation $E$ there is a definable set $Y$ and a definable function $f:X\to Y$ such that $xEy$ if and only if $f(x)=f(y)$.
\item [(v)] Suppose $(K,\D)\models DCF_{0,m}$ and $A\subseteq K$. Then $\operatorname{dcl}(A)=\QQ\l A\r_\D$, the $\D$-field generated by $A$, and $\operatorname{acl}(A)=\QQ\l A\r_\D^{alg}$, the field theoretic algebraic closure of the $\D$-field generated by $A$.
\end{enumerate}
\end{fact}

In somewhat the same way as we can consider the algebraic closure of a field, there exist \emph{differential closures} of differential fields. A $\D$-field extension $\bar K$ of $K$ is a \emph{$\D$-closure of} $K$ if $(\bar K,\D)\models DCF_{0,m}$ and for every $(L,\D)\models DCF_{0,m}$ extending $K$ there is an $\D$-isomorphism from $\bar K$ into $L$ fixing $K$ pointwise. The following result follows from general model theoretic facts (see \cite{Pi6}).

\begin{fact} 
$\D$-closures always exist and are unique up to $\D$-isomorphism over $K$. Moreover, if $\bar K$ is a $\D$-closure of $K$ then for every $a\in \bar K$ the type $tp(a/K)$ is \emph{isolated}; that is, there is an $\mathcal{L}_m$-formula over $K$ such that for all tuples $b$ from $\U$ we have $$\U\models \phi(b) \iff tp(b/K)=tp(a/K).$$
\end{fact}

A model $(\U,\D)$ of $DCF_{0,m}$ is said to be \emph{saturated} if whenever $\Phi(x)$ is a (possibly infinite) system of $\D$-polynomial equations and inequations over a $\D$-subfield of cardinality less than $\U$, any finite subsystem of which has a solution in $\U$, then $\Phi(x)$ has a solution in $\U$. It follows from general model theoretic facts (see \cite{Pi8}) that for arbitrarily large cardinal $\kappa$ there is a saturated differentially closed field of cardinality $\kappa$.

\begin{fact}\label{poi}
Let $(\U,\D)$ be a saturated model of $DCF_{0,m}$.
\begin{enumerate}
\item [(i)] If $(K,\D)\models DCF_{0,m}$ and $|K|\leq |\U|$, then there is a $\D$-isomorphism from $K$ into $\U$.
\item [(ii)] Suppose $K$ is a $\D$-subfield of $\U$ with $|K|<|\U|$ and $f$ is a $\D$-isomorphism from $K$ into $\U$, then there is a $\D$-automorphism of $\U$ extending $f$.
\item [(iii)] Suppose $K$ is a $\D$-subfield of $\U$ with $|K|<|\U|$. If $a$ and $a'$ are tuples from $\U$, then $tp(a/K)=tp(a'/K)$ if and only if there is $f\in Aut_{\D}(\U/K)$ such that $f(a)=a'$. Here $Aut_\D(\U/K)$ is the set of $\D$-automorphisms of $\U$ fixing $K$ pointwise.
\item [(iv)] Let $A\subset \U$ with $|A|<|\U|$. Then, $$\operatorname{dcl}(A)=\{a\in \U: f(a)=a \text{ for all } f\in Aut_{\D}(\U/A)\}$$ and $$\operatorname{acl}(A)=\{a\in \U: \text{ the orbit of } a \text{ under } Aut_\D(\U/A) \text{ is finite}\}.$$ 
\item [(v)] Let $A\subset \U$ with $|A|<|\U|$. Then a definable set is definable over $A$ if and only if every automorphism of $\U$ fixing $A$ pointwise fixes the definable set setwise.
\end{enumerate}
\end{fact}

The previous facts explain why it is convenient to work in a saturated model. It serves as a universal domain in which to study differential fields. Thus, we will usually work in a fixed suficiently large saturated $(\U,\D)\models DCF_{0,m}$. Unless stated otherwise, all the parameter subsets of $\U$ considered will be assumed to be small, i.e., of cardinality less than $\U$.

\section{Independence and $U$-rank}

We work in fixed sufficiently large saturated differentially closed field $(\U,\D)$.

We have a robust notion of independence that is analogous to algebraic independence in algebraically closed fields. Let $A$, $B$, $C$ be (small) subsets of $\U$. We say that $A$ is \emph{independent of $B$ over $C$}, denoted by $A \ind_C B$, if and only if $\QQ\l A \cup C\r_{\D}$ is algebraically disjoint from $\QQ\l B \cup C\r_{\D}$ over $\QQ\l C\r_{\D}$. Abstractly this notion of independence comes from Shelah's nonforking in stable theories \cite{She}; that is, if $C\subseteq B$ and $a$ is a tuple, then $tp(a/B)$ \emph{does not fork over $C$}, or \emph{is a nonforking extension of} $tp(a/C)$, if and only if $a\ind_C B$. For our purposes, we may as well take the previous sentence as a definition of nonforking extension. We collect in the following fact the essential properties of independence.

\begin{fact}\label{prur}
Let $a$ be a tuple from $\U$ and $A,B,C$ subsets of $\U$.
\begin{enumerate}
\item (Invariance) If $\phi$ is an automorphism of $\U$ and $A\ind_C B$, then $\phi(A)\ind_{\phi(C)} \phi(B)$.
\item (Local character) There is a finite subset $B_0\subseteq B$ such that $A\ind_{B_0} B$.
\item (Extension) There is a tuple $b$ such that $tp(a/C)=tp(b/C)$ and $b\ind_C B$.
\item (Symmetry) $A\ind_C B$ if and only if $B\ind_ C A$
\item (Transitivity) Assume $C\subseteq B\subseteq D$. Then $A\ind_C D$ if and only if $A\ind_C B$ and $A\ind_B D$.
\item (Stationarity) If $C$ is an algebraically closed $\D$-field and $C\subseteq B$, then there is a unique nonforking extension $p\in S_n(B)$ of $tp(a/C)$.
\end{enumerate}
\end{fact}

The notion of nonforking extension induces a partial ordering in the set of types, and the foundation rank of this partial ordering is a very important model theoretic invariant.

\begin{definition}\label{Las}
Let $p\in S_n(K)$. Then $U(p)\geq 0$, and for an ordinal $\al$ we define $U(p)\geq \al$ inductively as follows:
\begin{enumerate}
\item [1)] For a succesor ordinal $\al=\beta +1$, $U(p)\geq \al$ if there is a forking extension $q$ of $p$ such that $U(q)\geq \beta$.
\item [2)] For a limit ordinal $\al$, $U(p)\geq \al$ if $U(p)\geq \beta$ for all $\beta<\al$.
\end{enumerate}
We define $U(p)=\al$ if $U(p)\geq\al$ and $U(p)\ngeq \al+1$. In case $p=tp(a/K)$ we sometimes write $U(a/K)$ for $U(p)$. 
\end{definition}

We now list some of the basic properties of $U$-rank.

\begin{fact}\label{yut}
Let $a$ and $b$ be tuples from $\U$.
\begin{enumerate}
\item [(i)] $U(a/K)$ is an ordinal.
\item [(ii)] $U(a/K)=0$ if and only if $a\in \operatorname{acl}(K)$. Consequently, $U(a/K)=1$ if and only if  $a\notin \operatorname{acl}(K)$ but every forking extension of $tp(a/K)$ is \emph{algebraic} (i.e., has only finitely many realisations).
\item [(iii)] If $L$ is a $\D$-field extension of $K$, then $a\ind_K L$ if and only if $U(a/K)=U(a/L)$.
\item [(iv)] Lascar inequalities: $$U(a/K\l b\r)+U(b/K)\leq U(a,b/K)\leq U(a/K\l b\r)\oplus U(b/K).$$ Here $\oplus$ stands for the Cantor sum of ordinals.
\end{enumerate}
\end{fact}

The $U$-rank of an arbitrary definable set $X\subseteq \U^n$ is defined by $$U(X)=\sup\{U(a/F): a\in X\}$$ where $F$ is any (small) $\D$-field over which $X$ is defined. Using Fact~\ref{prur}, it is not hard to show that this is well defined (i.e., it is independent of the choice of $F$). By the same token, if $f:X\to Y$ is a definable bijection between definable sets then $U(X)=U(Y)$.

\begin{example}
Since the field of constants $\U^\D$ is a pure algebraically closed field, every definable subset of $\U^\D$ is either finite or cofinite. That is, by definition, the field of constants forms a \emph{strongly minimal set}. It follows from Fact~\ref{yut}~(ii) that the constant field is of $U$-rank~$1$.
\end{example}

The $U$-rank of $\U$ is $\w^m$. See \cite{Ma4} for a detailed proof of this in the case when $m=1$. As we could not find a correct proof of the general case in the literature, we have included one in the appendix of this thesis. We also include in the appendix some $U$-rank computations of certain definable subgroups of the additive group.

\section{Affine differential algebraic geometry}\label{affi}

We fix a sufficiently large saturated model $(\U,\D)\models DCF_{0,m}$ and a base (small) $\D$-subfield $K$.

\begin{definition}
An \emph{affine $\D$-algebraic variety $V$} (or \emph{affine $\D$-variety} for short) \emph{defined over} $K$ is a $\D$-closed subset of $\U^n$ (for some $n$) defined over $K$. That is, for some set of $\D$-polynomials $B\subseteq K\{x\}_\D$, $$V=\{a\in \U^n: f(a)=0 \text{ for all } f\in B\}.$$
\end{definition}

\begin{remark} \
\begin{enumerate}
\item [(i)] Our approach here is in the spirit of classical rather than contemporary algebraic geometry: for us affine $\D$-varieties are a set of points rather than a scheme. The fact that we are working in a saturated model (so a universal domain in the spirit of Weil) means that this is not an unreasonable way to proceed. Scheme theoretic approaches to differential algebraic geometry have been certainly developed (see for example \cite{Kova}), but there are subtleties there that we wish to avoid.
\item [(ii)]  Every affine $\D$-variety is of course a definable set in $(\U,\D)$. On the other hand, quantifier elimination (Fact~\ref{onty}~(iii)) tells us that every definable set is a finite boolean combination of affine $\D$-varieties.
\end{enumerate}
\end{remark}

\begin{definition}\label{regmaps}
A map $f:V\to W$ between affine $\D$-varieties defined over $K$ is \emph{$\D$-regular over $K$} if there is a finite cover $\{O_i\}$ of $V$ by $\D$-open subsets defined over $K$ and a family $\{g_i\}$ of tuples of $\D$-rational functions over $K$ such that $f|_{O_i}=g_i|_{O_i}$.
\end{definition}

Given $A\subseteq K\{x\}$ and $V\subseteq \U^n$, we let $$\V(A):=\{a\in \U^n: f(a)=0 \text{ for all } f\in A\}$$ and $$\I(V/K):=\{f\in K\{x\}: f(a)=0 \text{ for all } a\in V\}.$$ 

The set $\V(A)$ is called the \emph{vanishing set of} $A$, and $\I(V/K)$ is called the \emph{defining $\D$-ideal of} $V$ over $K$. It is easy to check that $\I(V/K)$ is a radical differential ideal, and that $\V(\I(V/K))$ is the $\D$-closure of $V$ over $K$; that is, the smallest $\D$-closed set defined over $K$ containing $V$. If $a$ is a tuple from $\U$ we define the \emph{$\D$-locus of $a$ over $K$} as $\V(\I(a/K))$.

\begin{fact}[Differential Nullstellensatz, \cite{Ko}, Chap. IV, \S 2]\label{diffnull}
Let $A\subseteq K\{x\}$. Then $\I(\V(A)/K)=\sqrt{[A]}$.
\end{fact} 

An affine $\D$-variety $V$ defined over $K$ is said to be \emph{irreducible over $K$}, or \emph{$K$-irreducible}, if it is not the union of two proper $\D$-subvarieties defined over $K$; equivalently, the defining $\D$-ideal of $V$ over $K$ is prime. By the differential basis theorem (Fact~\ref{difba}), every affine $\D$-variety $V$ defined over $K$ has a unique irredundant decomposition into $K$-irreducible $\D$-varieties, which are called the \emph{$K$-irreducible components of} $V$.

\begin{fact}[\cite{Ma4}, Chap. II, \S 5]
If $V$ is $K^{alg}$-irreducible then $V$ is (absolutely) irreducible; that is, $V$ is $F$-irreducible for all $\D$-field extensions $F$ of $K$.
\end{fact}

Working in a saturated model gives us \emph{$\D$-generic points}:

\begin{definition}
Let $V$ be a $K$-irreducible affine $\D$-variety. A point $a\in V$ is called a \emph{$\D$-generic point of $V$ over $K$} (or simply generic when $\D$ and $K$ are understood) if $a$ is not contained in any proper $\D$-subvariety of $V$ defined over $K$.
\end{definition}

Note, for example, that any $a$ is a $\D$-generic point over $K$ of its own $\D$-locus over $K$. 

Let us show that $\D$-generic points always exist. Let $V$ be $K$-irreducible. Consider the set of $\mathcal{L}_m$-formulas over $K$ $$\Phi=\{x\in V\}\cup\{x\notin W: W \text{ is a proper }\D\text{-subvariety of } V \text{ defined over }K\}.$$ If this set were finitely inconsistent, then $V$ would not be $K$-irreducible. Hence, by saturation, $\Phi$ is realised in $\U$, and such a realisation is the desired $\D$-generic point. 

If $a$ and $a'$ are $\D$-generic points over $K$ of a $K$-irreducible affine $\D$-variety $V$, then, by quantifier elimination for $DCF_{0,m}$ (Fact~\ref{onty}), $tp(a/K)=tp(a'/K)$. Hence, we can define the \emph{$\D$-generic type of $V$ over $K$} to be the complete type over $K$ of any $\D$-generic point of $V$ over $K$. In fact, this yields a bijective correspondence between $K$-irreducible affine $\D$-varieties and complete types over $K$, given by $$V\mapsto \text{the $\D$-generic type of $V$ over $K$}.$$ The inverse of this correspondence is given by taking the $\D$-locus over $K$ of any realisation of the type. 

The differential analogue of the tangent bundle of an algebraic variety is given by the following construction of Kolchin (c.f. \cite{Ko}, Chap. VIII, \S 2).

\begin{definition}\label{koltan}
Let $V$ be an affine $\D$-variety defined over $K$. The \emph{$\D$-tangent bundle of} $V$, denoted by $T_\D V$, is the affine $\D$-variety defined by $$f(x)=0 \quad \text{and} \quad \sum_{\t\in \T_\D, i\leq n}\frac{\partial f}{\partial (\t x_i)}(x)\t u_i=0,$$ for all $f\in \I(V/K)$, together with the projection onto the first $n$ coordinates $\rho:T_\D V\to V$.
\end{definition}

An \emph{affine $\D$-algebraic group defined over $K$} is a group object in the category of affine $\D$-varieties defined over $K$. In other words, an affine $\D$-algebraic group over $K$ is a group whose underlying universe is an affine $\D$-algebraic variety and the group operation $p:G\times G\to G$ is a $\D$-regular map, both over $K$. 

The following property of affine $\D$-algebraic groups, which is not necessarily true for $\D$-varieties, will be important for us in the appendix.

\begin{fact}[\cite{Fre2}, \S 2]\label{gengroup}
If $G$ is a connected affine $\D$-algebraic group defined over $K$, then $U(G)=U(p)$ where $p$ is the $\D$-generic type of $G$ over $K$.
\end{fact}

We conclude this section with the following proposition describing the geometric meaning of characteristic sets.

\begin{proposition}\label{basic}
Suppose $V$ is a $K$-irreducible affine $\D$-variety. Let $\L$ be a characteristic set of the prime $\D$-ideal $\I(V/K)$. 
\begin{enumerate}
\item [(i)] Let $\V^*(\L)=\V(\L)\setminus\V(H_\L)$, where $H_\L$ is the product of the initials and separants of $\L$ (see Section~\ref{ondiff}). Then $\V^*(\L)$ in nonempty.
\item [(ii)] $\V^*(\L)=V\setminus \V(H_\L)$.
\item [(iii)] $V$ is a $K$-irreducible component of $\V(\L)$.
\end{enumerate}
\end{proposition}
\begin{proof} \

\noindent (i) If $\V(\L)\subseteq \V(H_\L)$ then, by the differential Nullstellensatz (Fact~\ref{diffnull}), $H_\L$ would be in $\sqrt{[\L]}\subseteq \I(V/K)$ contradicting Remark~\ref{contin}. Thus $\V(\L)\setminus\V(H_\L)\neq \emptyset$.

\noindent (ii) Let $a\in\V^*(\L)$. We need to show that $f(a)=0$ for all $f\in \I(V/K)$. Clearly if $f\in [\L]$ then $f(a)=0$. Let $f\in \I(V/K)$, since $\I(V/K)=[\L]:H_\L^\infty$, we can find $\ell$ such that $H_\L^\ell\, f\in [\L]$. Hence, $H_\L(a)f(a)=0$, but $H_\L(a)\neq 0$, and so $f(a)=0$. The other containment is clear. 

\noindent (iii) Let $W$ be an irreducible component of $\V(\L)$ containing $V$.  Since $\V^*(\L)=V\setminus \V(H_\L)$, we have that $W\setminus\V(H_\L)=V\setminus \V(H_\L)\subseteq V$. Hence $V$ contains a nonempty $\D$-open set of $W$, and so, by irreducibility of $W$, $V=W$.
\end{proof}

\section{Differential type and typical differential dimension}\label{polype}

We continue to work in a fixed sufficiently large saturated $(\U,\D)\models DCF_{0,m}$ and over a (small) base $\D$-subfield $K$ of $\U$. In this section we discuss the Kolchin polynomial and the associated notions of differential type and typical differential dimension. As these are central to our work in this thesis we will give more details than in previous sections.

Let $L$ be a $\D$-subfield of $\U$ containing $K$. A set $A\subseteq L$ is said to be \emph{$\D$-algebraically independent over $K$} if for every (finite) tuple $a$ from $A$ we have that $\I(a/K)=\{0\}$. A maximal subset of $L$ that is $\D$-algebraically independet over $K$ is called a \emph{$\D$-transcendence basis of $L$ over $K$}. An element $a\in L$ is said to be \emph{$\D$-transcendental over $K$} if $\{a\}$ is $\D$-algebraically independent over $K$, and \emph{$\D$-algebraic over $K$} otherwise.

\begin{fact}[\cite{Ko}, Chap. II, \S 9]
There exists a $\D$-transcendence basis of $L$ over $K$, and any two such basis have the same cardinality. 
\end{fact}

By virtue of the previous fact one defines the \emph{$\D$-transcendence degree of $L$ over $K$} as the cardinality of any $\D$-transcendence basis of $L$ over $K$.

By a \emph{numerical polynomial} we mean a polynomial (in one variable) $f\in \mathbb{R}[x]$ such that if $e=deg(f)$ then there are unique $d_0,\dots,d_e\in \ZZ$ with $f(x)=\sum_{i=0}^e d_i$$x+i\choose{i}$. Note that if $f$ is a numerical polynomial then $f(h)\in \ZZ$ for sufficiently large $h\in \NN$.

We denote by $\T_\D(h)$ the set of derivative operators of order less than or equal to $h\in \NN$. 

\begin{fact}[\cite{Ko}, Chap.2, \S 12]\label{poli}
Let $a=(a_1,\dots,a_n)$ be a tuple from $\U$. There exists a numerical polynomial $\omega_{a/K}$ with the following properties:
\begin{enumerate}
\item [(i)] For sufficiently large $h\in \NN$, $\omega_{a/K}(h)$ equals the transcendence degree of $K(\t a_i:\, i=1,\dots,n,\, \t\in \T_\D(h))$ over $K$.
\item [(ii)] deg $\omega_{a/K}\leq m$.
\item [(iii)] If we write $\omega_{a/K}=\sum_{i=0}^m d_i$$x+i\choose{i}$ where $d_i\in \ZZ$, then $d_m$ equals the $\D$-transcendence degree of $K\l a\r_\D$ over $K$.
\item [(iv)] If $b$ is another tuple with each coordinate in $K\l a\r$, then there is $h_0\in \NN$ such that for sufficiently large $h\in \NN$ we have $\omega_{b/K}(h)\leq \omega_{a/K}(h+h_0)$.
\end{enumerate}
\end{fact}

The polynomial $\omega_{a/K}$ is called the \emph{Kolchin polynomial of} $a$ \emph{over} $K$. Even though the Kolchin polynomial is not in general a $\D$-birational invariant, (iv) of Fact \ref{poli} shows that its degree is. We call this degree the $\D$\emph{-type of} $a$ \emph{over} $K$ and denote it by $\D$-type$(a/K)$. Similarly, if we write $\omega_{a/K}=\sum_{i=0}^m d_i$$x+i\choose{i}$ where $d_i\in \ZZ$, the coefficient $d_{\tau}$, where $\tau=\D$-type$(a/K)$, is a $\D$-birational invariant. We call $d_{\tau}$ the \emph{typical} $\D$\emph{-dimension} of $a$ \emph{over} $K$ and denote it by $\D$-dim$(a/K)$. We will adopt the convention that if $\omega_{a/K}=0$ then $\D$-type$(a/K)=0$ and $\D$-dim$(a/K)=0$.

Given a complete type $p=tp(a/K)$ its Kolchin polynomial $\omega_p$ is defined to be $\omega_{a/K}$. Thus it makes sense to talk about the $\D$-type and typical $\D$-dimension of complete types. 

Now we extend these concepts to definable sets. First recall that we can put a total ordering on numerical polynomials by eventual domination, i.e., $f\leq g$ if and only if $f(h)\leq g(h)$ for all sufficiently large $h\in \NN$.

\begin{definition}
Let $X$ be a definable set. The \emph{Kolchin polynomial of $X$} is defined by
\begin{displaymath}
\omega_X=\sup\left\{ \omega_{a/F}:\, a\in X\right\},
\end{displaymath}
where $F$ is any $\D$-field over which $X$ is defined (the fact that this does not depend on the choice of $F$ is a consequence, for example, of Theorem 4.3.10 of \cite{Mc}). In Lemma \ref{max} below we will see that $\omega_X=\omega_{a/F}$ for some $a\in X$. We define the $\D$-type of $X$ and its typical $\D$-dimension in the obvious way. Also, $X$ is said to be \emph{finite dimensional} if $\D$-type$(X)=0$.
\end{definition}

\begin{lemma}\label{max}
Let $X$ be a $K$-definable set and let $V_1, \dots, V_s$ be the $K$-irreducible components of the $\D$-closure of $X$ over $K$ (in the $\D$-topology). If $p_i$ is the $\D$-generic type of $V_i$ over $K$, then
\begin{displaymath}
\omega_X=\max\{\omega_{p_i}:\, i=1,\dots,s\}.
\end{displaymath}
\end{lemma}
\begin{proof}
Let $b\in X$, then $b$ is in some $V_i$. Let $a_i$ be a realisation of $p_i$. By $\D$-genericity, there is a $\D$-ring homomorphism $f_i:K\{a_i\}_{\D}\to K\{b\}_{\D}$ over $K$ such that $a_i\mapsto b$. Thus, for any $h\in \NN$ the transcendence degree of $K(\t a_{i,j}: j=1,\dots,n,\, \t\in \T_\D(h))$ over $K$, where $a_i=(a_{i,1},\dots,a_{i,n})$, is greater than or equal to the transcendence degree of $K(\t b_j: j=1,\dots,n',\, \t\in \T_\D(h))$ over $K$, where $b=(b_1,\dots,b_{n'})$. This implies that, for sufficiently large $h\in \NN$, $\omega_{a_i/K}(h)\geq\omega_{b/K}(h)$ and hence $\omega_{a_i/K}\geq\omega_{b/K}$.
\end{proof}

In the ordinary case of $\D=\{\d\}$, if an element $a$ is $\{\d\}$-algebraic over $K$ then $\d^{k+1}a\in K(a,\d a, \dots,\d^k a)$ for some $k$ (see for example Chapter 1.1 of \cite{Ma4}). In general, for arbitrary $\D$, it may occur that $a$ is $\D$-algebraic over $K$ but $\d_{m}^{k+1}a\notin K\l a,\d_m a,\dots,\d_m^{k}\r_{\D'}$ for every $k$, where $\D'=\D\setminus\{\d_m\}$. However, if we allow linearly independent transformations of $\D$ over $K^{\D}$, an analogous result holds, which we explain now.

Observe that any linear combination of the derivations in $\D$ over the constant field $\U^\D$ is again a derivation on $\U$. As we are working over $K$ we will be considering the $K^\D$-vector space $\operatorname{span}_{K^\D}\D$ of derivations of $\U$ obtained by taking linear combinations of the derivations in $\D$ over the constant field of $(K,\D)$. Note that if $\D'\subseteq \operatorname{span}_{K^\D}\D$, then any $\D$-field extension of $K$ is also a $\D'$-field. One consequence of $\U$ being differentially closed is that $\D$ is a basis for $\operatorname{span}_{K^\D}\D$. Moreover, by verifying the definition of existentially closed, it is not hard to see that if $\D'\subseteq \operatorname{span}_{K^\D}\D$ is a $K^\D$-linearly independent set, then $(\U,\D')$ is itself differentially closed and in fact saturated. Shifting from $\D$ to some other basis for $\operatorname{span}_{K^\D}\D$ is something that turns out to be quite useful. For example:

\begin{fact}[\cite{Ko}, Chap. 2, \S 11]\label{diffal}
Let $a=(a_1,\dots,a_n)$ be a tuple from $\U$ such that each $a_i$ is $\D$-algebraic over $K$. Then there is a natural number $k>0$ and a basis $\D'\cup\{D\}$ of the $K^{\D}$-vector space $\operatorname{span}_{K^\D}\D$ such that $K\l a\r_\D=K\l a, D,\dots, D^{k}a\r_{\D'}$.
\end{fact}

Fact~\ref{diffal} implies the following fact which will be of central importance to us.

\begin{fact}[\cite{Ko}, Chap. 2, \S 13]\label{Kolteo}
Let $a$ be a tuple from $\U$. Then there is a set $\D'$ of linearly independent elements of the $K^\D$-vector space $\operatorname{span}_{K^\D}\D$ with $|\D'|=\D$-type$(a/K)$, and a tuple $\al$ from $\U$, such that
\begin{displaymath}
K\la a\ra_{\D}=K\la\al\ra_{\D'}
\end{displaymath}
with $\D'$-type$(\al/K)=\D$-type$(a/K)$ and $\D'$-dim$(\al/K)=\D$-dim$(a/K)$.
\end{fact}

Fact~\ref{Kolteo} says that a tuple $a$ has $\D$-type less than or equal to $\tau\in {\mathbb N}$ if and only if there is a set $\D'$ of linearly independent elements of $\operatorname{span}_{K^\D}\D$ with $|\D'|=\tau$ such that $K\l a\r_\D$ is finitely $\D'$-generated over $K$. If in addition the $\D'$-transcendence degree of $K\l a\r_\D$ over $K$ is positive then $\D$-type$(a/K)=\tau$. This motivates the following terminology (which we take the liberty to introduce).

\begin{definition}[\it{Bounding and witnessing the $\D$-type}]\label{boundty}
Let $a$ be a (finite) tuple. A set $\D'$ of linearly independent elements of $\operatorname{span}_{K^\D}\D$ is said to \emph{bound the $\D$-type of $a$ over $K$} if the $\D$-field $K\l a\r_{\D}$ is finitely $\D'$-generated over $K$. Moreover, if $\D'$ bounds the $\D$-type of $a$ over $K$ and $|\D'|=\D$-type$(a/K)$ then we will say that $\D'$ \emph{witnesses the $\D$-type of $a$ over $K$}. 
\end{definition}

So Fact \ref{Kolteo} says that given any finite tuple $a$ we can always find a $\D'$ witnessing the $\D$-type of $a$ over $K$.

\begin{lemma}\label{onb}
Let $a$ be a tuple, $L$ a $\D$-field extension of $K$, and $\D'\subseteq \operatorname{span}_{K^\D}\D$ linearly independent. Suppose $a\ind_K L$. Then $\D'$ bounds $\D$-type$(a/L)$ if and only if it bounds $\D$-type$(a/K)$.
\end{lemma}
\begin{proof}
Assume $\D'$ bounds $\D$-type$(a/L)$. Then we can find $\alpha$ a tuple of the form $$(a,\t_1 a,\dots,\t_s a:\, \t_i\in \T_\D, s<\omega)$$ such that $L\l a\r_{\D}=L\l \alpha \r_{\D'}$; here if $a=(a_1,\dots,a_n)$ and $\t \in \T_\D$ then $\t a=(\t a_1,\dots,\t a_n)$. It suffices to prove the following: 

\vspace{.07in}

\noindent {\bf Claim.} $K\l a\r_{\D}=K\l \alpha,\d_{1}\alpha,\dots,\d_m\alpha\r_{\D'}$.

\vspace{.05in}

\noindent Let $w\in K\l \alpha,\d_{1}\alpha,\dots,\d_m\alpha\r_{\D'}$, we need to show that $\d w\in K\l \alpha,\d_{1}\alpha,\dots,\d_m\alpha\r_{\D'}$ for all $\d\in \D$. Since $w\in L\l \alpha \r_{\D}=L\l \alpha\r_{\D'}$, then $w=f(\alpha)$ for some $\D'$-rational function $f$ over $L$. Thus, there is a tuple $\beta$ of the form $(\al,\t_1'\al,\dots,\t_r'\al:\, \t_i'\in \T_{\D'}, r< \omega)$ such that $w=g(\beta)$ for some rational function $g$ over $L$. Let $\rho$ be a tuple whose entries form a maximal algebraically independent (over $K$) subset of the set consisting of the entries of $\beta$. Then $(w,\rho)$ is algebraically dependent over $L$. But since $a\ind_K L$, $K\l a\r_\D$ is algebraically disjoint from $L$ over $K$. Thus, since $(w,\rho)$ is from $K\l a\r_{\D}$, we get that $(w,\rho)$ is algebraically dependent over $K$. Thus, there is $h\in K[x,y]$ such that $h(w,\rho)=0$. Moreover, since $\rho$ is algebraically independent over $K$, $w$ appears non trivially in $h(w,\rho)$. Hence, $w$ is algebraic over $K(\rho)$ and so $\d w\in K(w,\rho,\d\rho)\subseteq K\l \alpha,\d_{1}\alpha,\dots,\d_m\alpha\r_{\D'}$.
\end{proof}

\begin{definition}\label{defboset}
Let $X$ be a $K$-definable set. A linearly independent subset $\D'\subseteq \operatorname{span}_{K^\D}\D$ is said to \emph{bound the $\D$-type of $X$} if $\D'$~bounds $\D$-type$(a/K)$ for all $a\in X$. We say $\D'$ \emph{witnesses the $\D$-type of $X$} if $\D'$ bounds the $\D$-type of $X$ and $|\D'|=\D$-type$(X)$.
\end{definition}

By Lemma \ref{onb}, this definition does not depend on the choice of $K$. In other words, $\D'$ bounds the $\D$-type$(X)$ if and only if there is a $\D$-field $F$, over which $X$ and $\D'$ are defined, such that $\D'$ bounds the $\D$-type$(a/F)$ for all $a\in X$. Indeed, suppose that $\D'$-bounds the $\D$-type of $X$, $F$ is a $\D$-field over which $X$ and $\D'$ are defined, and $a\in X$. One can assume, without loss of generality, that $F<K$. Let $b$ be a tuple from $\U$ such that $b\ind_F K$ and $tp(b/F)=tp(a/F)$. By assumption, $\D'$ bounds the $\D$-type$(b/K)$, and so, by Lemma~\ref{onb}, $\D'$ bounds the $\D$-type$(b/F)$. Hence, $\D'$ bounds $\D$-type$(a/F)$.

\begin{remark}\label{witge}
It is not clear (at least to the author) that every definable set has a witness to its $\D$-type. However, for definable groups one can always find such a witness. To see this let $G$ be a $K$-definable group. Note that, since the property ``$\D'$ witnesses the $\D$-type'' for definable sets is preserved under definable bijection, a witness to the $\D$-type of the connected component of $G$ will be a witness to the $\D$-type of $G$. Hence, we may assume that $G$ is connected. Let $p$ be the generic type of $G$ over $K$. By Fact \ref{Kolteo}, there is $\D'$, a linearly independent subset of $\operatorname{span}_{K^\D}\D$, such that for any $a\models p$ the $\D$-field $K\l a\r_{\D}$ is finitely $\D'$-generated over $K$. Let us check that this $\D'$ witnesses the $\D$-type of $G$. Let $g\in G$, then there are $a,b\models p$ such that $g=a\cdot b$. Indeed, choose $a\models p$ such that $a\ind_K g$ and let $b=a^{-1}\cdot g$ (see \S 7.2 of \cite{Ma2} for details). Thus $K\l g\r_{\D}\leq K\l a,b\r_{\D}$, but since the latter is finitely $\D'$-generated over $K$ the former is as well.
\end{remark}

We conclude this section with a remark on how the differential type interacts with the model theoretic notion of internality. Let $X$ be a $K$-definable set. A complete type $tp(a/K)$ is said to be \emph{internal to $X$} is there is $\D$-field extension $L$ of $K$ with $a\ind_K L$ and a tuple $c$ from $X$ such that $a\in \operatorname{dcl}(L,c)=L\l c\r_\D$.

\begin{lemma}\label{onint}
Let $X$ be a $K$-definable set and suppose that $tp(a/K)$ is $X$-internal. Then
\begin{enumerate}
\item $\D\text{-type}(a/K)\leq \D\text{-type}(X)$.
\item If $\D'$ bounds $\D$-type$(X)$ then $\D'$ also bounds $\D$-type$(a/K)$.
\end{enumerate}
\end{lemma}
\begin{proof} \

\noindent (1)  For any $d=(d_1,\dots,d_r)$, $\D$-type$\displaystyle(d/K)=\max_{i\leq r}\{\D$-type$(d_i/K)\}$ (see for example Lemma 3.1 of \cite{Moo}). Now let $L$ be a $\D$-field extension of $K$ such that $a\ind_K L$, and $c$ from $X$ such that $a\in L\l c\r_\D$. Hence,
\begin{displaymath}
\D\text{-type}(a/K)=\D\text{-type}(a/L)\leq \D\text{-type}(c/L)\leq \D\text{-type}(X),
\end{displaymath}
where the equality follows from Lemma \ref{onb} and the first inequality holds by Fact~\ref{poli}~(iv).

\noindent (2) Assume $\D'$ bounds $\D$-type$(X)$, then $L\l c\r_{\D}$ is finitely $\D'$-generated over $L$. Since $L\l a\r_{\D}\leq L \l c\r_{\D}$, then $L\l a\r_{\D}$ is also finitely $\D'$-generated over $L$. But $a\ind_K L$, so, by Lemma \ref{onb}, $\D'$ bounds the $\D$-type of $a$ over $K$.
\end{proof}

\section{Abstract differential algebraic varieties}

In Section~\ref{affi} we defined the notion of an affine differential algebraic variety. While these affine objects are quite concrete and provide the proper setting for the presentation of most of this thesis, in Sections \ref{relDgroup} and \ref{galex} we will need a more general notion of differential algebraic variety. In this section we discuss abstract differential algebraic varieties as presented in \cite{Pi4} (see Chap. 7 of \cite{Ma2} for the algebraic case). We then naturally extend the concepts of Sections \ref{affi} and \ref{polype} to this more general setting.

We continue to work in a sufficiently large saturated $(\U,\D)\models DCF_{0,m}$ and over a base $\D$-subfield $K$.

By a \emph{quasi-affine $\D$-variety} defined over $K$ we mean a $\D$-open subset of an affine $\D$-variety, where both are defined over $K$. We define $\D$-regular maps between quasi-affine $\D$-varieties in the natural way (i.e., as in Definition~\ref{regmaps}). The idea behind the construction of abstract $\D$-varieties is to build them up from affine ones (or even quasi-affine) in the same way that manifolds are built from open sets of $\mathbb{R}^n$ or $\mathbb{C}^n$. This is the analogue of Weil's definition of abstract algebraic varieties.

\begin{definition}\label{absvar}
An \emph{abstract $\D$-algebraic variety} (or \emph{$\D$-variety} for short) is a topological space $V$ equipped with a finite open cover $\{V_i\}_{i=1}^s$ and homeomorphisms $f_i:V_i\to U_i$, where $U_i$ is an affine $\D$-variety, such that
\begin{enumerate}
\item [(i)] $U_{ij}=f_i(V_i\cap V_j)$ is a $\D$-open subset of $U_i$, and
\item [(ii)] $f_{ij}=f_i\comp f_j^{-1}:U_{ji}\to U_{ij}$ is a $\D$-regular map of quasi-affine varieties.
\end{enumerate}
The pairs $(V_i,f_i)$ are called the charts of the $\D$-variety. Hence, a $\D$-variety $V$ is given by the information $(V,\{V_i\},\{f_i\})$. We say $V$ is defined over $K$ if the $U_i$'s and the $f_{ij}$'s are defined over $K$.
\end{definition}

It is easy to see that affine, and more generally quasi-affine, $\D$-varieties are abstract $\D$-varieties. 

\begin{definition}
Suppose $(V,\{V_i\},\{f_i\})$ and $(W,\{W_j\},\{g_j\})$ are $\D$-varieties defined over $K$. We say that a map $h:V\to W$ is $\D$-regular if $g_j\comp h\comp f_i^{-1}:f_i(V_i)\to g_j(W_j)$ is a $\D$-regular map of affine $\D$-varieties for each $i,j$. We say that $h$ is defined over $K$ if all the $g_j\comp h\comp f_i^{-1}$ are defined over $K$.
\end{definition}

Let $(V,\{V_i\},\{f_i\})$ be a $\D$-variety. By an \emph{affine open set} of $V$ we mean an open subset of $V$ of the form $f_i^{-1}(W)$ for some $i$ and $W$ a $\D$-open subset of $f_i(V_i)$. Any open subset of $V$ is a finite union of affine open subsets. This implies that the topology of $V$ is Noetherian. We can define irreducibility and $K$-irreducibility of $V$ in the natural way. It follows, from Noetherianity, that if $V$ is defined over $K$ then it has a unique irredundant decomposition into $K$-irreducible $\D$-varieties.

As the reader might expect, since abstract $\D$-varieties are modelled by affine ones, in most cases one can reduce constructions and proofs of statements about abstract $\D$-varieties to the affine case. More precisely, one chooses an affine cover and works with coordinates in affine space using a local chart. We will however, in some cases, carry out certain constructions and proofs to make the concepts clear.

We now extend some of the concepts that were defined for affine $\D$-varieties.

Let $(V,\{V_i\},\{f_i\})$ be a $K$-irreducible $\D$-variety. A point $a\in V$ is said to be a \emph{$\D$-generic point of $V$ over $K$} if it is not contained in any proper $\D$-subvariety defined over $K$. As in the affine case, $\D$-generic points always exist. Any $a$ of the form $f_i^{-1}(b)$, where $b$ is a $\D$-generic point over $K$ of the affine $K$-irreducible $\D$-variety $f_i(V_i)$, is a $\D$-generic point of $V$ over $K$. Moreover, for every $i$, any $\D$-generic point of $V$ will be of this form. Hence, if we identify any two $\D$-generic points over $K$ with their images under the same $f_i$, they will have the same complete type over $K$. We define the \emph{$\D$-generic type of $V$ over $K$} to be this complete type. Notice that the $\D$-generic type depends on the embedding $f_i$. However, if $p_i$ and $p_j$ are the $\D$-generic types of $V$ over $K$ with respect to $f_i$ and $f_j$, then there are $a_i\models p_i$ and $a_j\models p_j$ such that $a_i$ and $a_j$ are interdefinable over $K$ (to see this simply use the transition functions $f_{ij}$). 

Let $(V,\{V_i\},\{f_i\})$ be a $\D$-variety defined over $K$ and $a\in V$. We define the \emph{Kolchin polynomial} of $a$ over $K$ to be 
$\w_{a/K}=\max\{\w_{f_i(a)/K}: i=1,\dots,s\}$. The \emph{$\D$-type and typical $\D$-dimension} of $a$ over $K$ are defined in the obvious way. It follows, using the transition functions $f_{ij}$, that $\D$-type$(a/K)=\D$-type$(f_i(a)/K)$ and $\D$-dim$(a/K)=\D$-dim$(f_i(a)/K)$ for any $i$. That is, the notions of $\D$-type and typical $\D$-dimension of $a$ over $K$ do not depend on the embedding.

The \emph{Kolchin polynomial} of $V$ is defined as
$$\w_V=\max\{\w_{f_i(V_i)}:\, i=1,\dots,s\}.$$
We define the \emph{$\D$-type and typical $\D$-dimension} of $V$ in the obvious way. It follows that if $V$ is $K$-irreducible then $\D$-type$(V)=\D$-type$(p)$ and $\D$-dim$(V)=\D$-dim$(p)$ where $p$ is the $\D$-generic type of $V$ over $K$ with respect to any embedding.

Let $a\in V$. A set $\D'$ of linearly independent elements of $\operatorname{span}_{K^\D}\D$ is said to \emph{bound the $\D$-type of $a$ over $K$} if $\D'$ bounds the $\D$-type of $f_i(a)$ over $K$ for some $i$. Since the property ``$\D'$ bounds the $\D$-type'' is preserved under interdefinability over $K$, we get that $\D'$ bounds $\D$-type$(a/K)$ if and only if it bounds $\D$-type$(f_j(a)/K)$ for every $j$. We say that $\D'$ \emph{bounds the $\D$-type of $V$} if $\D'$ bounds the $\D$-type of each $f_i(V_i)$. The notion of \emph{witness} of the $\D$-type is defined in the obvious way (see Definitions~\ref{boundty} and \ref{defboset}).

Given an abstract $\D$-variety $(V,\{V_i\},\{f_i\})$, we can construct the $\D$-tangent bundle of $V$ as an abstract $\D$-variety. One patches together all the affine pieces $T_\D f_i(V_i)$ to form a $\D$-variety where the transitions are given by the differentials of the $f_{ij}$'s. For a more detailed description of this type of constructions, we refer the reader to the end of Section~\ref{defrel} below, where we perform the construction of the \emph{relative prolongation} of $V$.

\begin{definition}
An (abstract) $\D$-algebraic group is a group object in the category of $\D$-varieties. In other words, a $\D$-algebraic group is a group whose underlying universe $G$ is a $\D$-variety and the group operation $p:G\times G\to G$ is a $\D$-regular map. We say the $\D$-algebraic group is defined over $K$, if $G$ and $p$ are over $K$.
\end{definition} 

Let us finish this section with a couple of remarks about $\D$-algebraic groups.

\begin{remark}\label{witg} \
\begin{enumerate} 
\item [(i)] By Remark~\ref{witge}, every $\D$-algebraic group has a witness to its $\D$-type.
\item [(ii)] Every $\D$-algebraic group defined over $K$ is isomorphic to a $K$-definable group in $(\U,\D)$. Indeed, suppose $(G,\{V_i\},\{f_i\})$ is a $\D$-algebraic group over $K$. Let $H$ be the disjoint union of the $f_i(V_i)$ quotiented by the $K$-definable equivalence relation $E$ induced by the transitions $f_{ij}$. By elimination of imaginaries (Fact~\ref{onty}~(iv)) for $DCF_{0,m}$, $H$ is $K$-definable group which is isomorphic to $G$. Due to this fact, we think of $\D$-algebraic groups as being definable sets in $(\U,\D)$. Conversely, every $K$-definable group is definably isomorphic to a $\D$-algebraic group over $K$. The proof of this requires much more effort and we refer the reader to \cite{Pi4} for details.
\end{enumerate}
\end{remark}

\chapter{Relative prolongations, D-varieties and D-groups}\label{chapro}

In this chapter we introduce the notions of \emph{relative} prolongation, D-variety, and D-group. We present their basic properties and show that every differential algebraic variety which is not of maximal differential type is differentially birational to the sharp points of a relative D-variety. This material will be used in Chapters \ref{chapgal} and \ref{chapaxioms}. We begin, however, with a review of this theory in the ordinary case.

\section{A review of prolongations and D-varieties in $DCF_0$}\label{repro}

In this section we briefly review the notions/properties of prolongations and D-varieties over an ordinary differential field. For more details (and proofs) on the material reviewed in this section we suggest \cite{PiPi}, \cite{Ma3}, \cite{PiT}, \cite{Pi7} and \cite{KoP}.

Fix a sufficiently large saturated $(\U,\d)\models DCF_0$ and a small $\d$-subfield $K$. Recall that we identify all affine algebraic varieties with their set of $\U$-points and that we denote by $\U^\d$ and $K^\d$ the field of constants of $\U$ and $K$, respectively.

\begin{definition}\
Let $V\subseteq \U^n$ be an algebraic variety defined over $K$. The \emph{prolongation of} $V$, denoted by $\tau V$, is the subvariety of $\U^{2n}$ defined by
\begin{equation}\label{eqofprol}
f(x)=0 \quad \text{ and } \quad \sum_{i=1}^n\frac{\partial f}{\partial x_i}(x)u_i +f^\d(x)=0
\end{equation}
for all $f$ in $\I(V/K)$, where $x=(x_1,\dots,x_n)$, $u=(u_1,\dots,u_n)$. Here $f^\d$ is the polynomial obtained by applying $\d$ to the coefficients of $f$. The map $\pi:\tau V\to V$ denotes the projection onto the first $n$ coordinates and, for each $a\in V$, $\tau V_a$ denotes the fibre (with respect to $\pi$) of $\tau V$ above $a$.
\end{definition}

\begin{remark}\label{nothard}
It is not hard to verify that in the definition of $\tau V$ it suffices to consider the equations~\ref{eqofprol} as $f$ ranges over any set of polynomials generating $\I(V/K)$, rather than all of $\I(V/K)$.
\end{remark}

Note that if $a\in V$ then $(a,\d a)\in \tau V$. This allows us to define a map $\nabla=\nabla_\d^V:V\to \tau V$ given on points by $a\mapsto (a,\d a)$. It is clear that $\nabla$ is a $\d$-regular section of $\pi:\tau V\to V$; that is, $\pi\comp\nabla=\operatorname{Id}_V$. 

In the case that $V$ is defined over the constants $K^\d$, inspecting the equations~\ref{eqofprol} and using Remark~\ref{nothard} above, we see that $\tau V$ is simply the tangent bundle $TV$. In general, $\pi:\tau V\to V$ is a \emph{torsor} under $\rho:TV\to V$; that is, there is a regular map $TV \times_V\tau V\to \tau V$ inducing a regular action (by translation) of $TV_a$ on $\tau V_a$ for all $a\in V$. For this reason the prolongation is often thought of as a \emph{twisted} version of the tangent bundle.

The points of the prolongation also have the following differential algebraic interpretation.

\begin{fact}\label{propro}
Suppose $a\in V$ is a generic point of $V$ over $K$ (in the sense of algebraic geometry). Then the points $b\in \tau V_a$ correspond to derivations $\d':K(a)\to \U$ that extend $\d$ and $\d' a=b$.
\end{fact}

One consequence of Fact~\ref{propro}, using also that $(\U,\d)\models DCF_0$, is that if $V$ and $W\subseteq \tau V$ are irreducible affine algebraic varieties, such that $W$ projects dominantly onto $V$, then $\nabla(V)$ is Zariski-dense in $W$. In fact, this latter property characterizes the models of $DCF_0$; these are the Pierce-Pillay axioms which will be discussed in Section~\ref{piax} below.

Let $f:V\to W$ be a regular map between affine algebraic varieties, where $f=(f_1,\dots,f_s)$. Then one can define a regular map $\tau f:\tau V\to \tau W$ given by
\begin{displaymath}
\tau f(x,u)=\left(f_1(x),\dots,f_s(x), \sum_{i=1}^n\frac{\partial f_1}{\partial x_i}(x)u_i +f_1^\d(x),\dots,\sum_{i=1}^n\frac{\partial f_s}{\partial x_i}(x)u_i +f_s^\d(x)\right).
\end{displaymath} 
One can in fact show that $\tau$ is a (covariant) functor from the category of affine varieties to itself. 

Another way to think of $\tau V$ is to identify it with the $\U[\epsilon]/(\epsilon)^2$-points of $V\times_K R$, where $R$ denotes the ring of dual numbers $K[\epsilon]/(\epsilon)^2$ equipped with the $K$-algebra structure $e:K\to R$ given by $e(a)= a+\d (a) \epsilon$. One can formalize this approach by defining $\tau V$ as the Weil restriction of $V\times_K R$ from $K[\epsilon]/(\epsilon)^2$ to $K$. Note that the base change is with respect to the $K$-algebra structure $e:K\to R$ while the Weil restriction is with respect to the standard $K$-algebra structure of $K[\epsilon]/(\epsilon)^2$. This approach is taken and generalized by Moosa and Scanlon in the series of papers \cite{Moo2}, \cite{Moo3} and \cite{MS}.

More generally, one can define the prolongation of an abstract algebraic variety $V$ defined over $K$ by patching together the prolongations of an affine covering of $V$. This construction yields the prolongation of $V$ as an abstract algebraic variety $\tau V$ defined over $K$ with a canonical projection $\pi:\tau V\to V$. Then, all the results of this section also hold for abstract algebraic varieties defined over $K$.

When studying finite dimensional $\d$-algebraic varieties it is convenient to work in the following geometric category introduced by Buium \cite{Buium} (see also \cite{Pi7}).

\begin{definition}
An \emph{algebraic D-variety} is a pair $(V,s)$ where $V$ is an (abstract) algebraic variety and $s:V\to\tau V$ is a regular section of $\pi:\tau V\to V$. The D-variety $(V,s)$ is said to be defined over $K$ if $V$ and $s$ are defined over $K$. The \emph{set of sharp points of} $(V,s)$ is defined as $(V,s)^\#=\{a\in V: \, s(a)=\nabla a\}$. A \emph{morphism of D-varieties} $(V,s)$ and $(W,t)$ is a regular map $f:V\to W$ such that $\tau f\comp s=t\comp f$.
\end{definition}

\begin{fact} \label{onit} \
\begin{enumerate}
\item [(a)] The set of sharp points of a D-variety $(V,s)$ is a finite dimensional Zariski-dense subset of $V$. 
\item [(b)] If $W$ is a finite dimensional irreducible $\d$-variety defined over $K$, then there exists an algebraic D-variety $(V,s)$ defined over $K$ such that $W$ is $\d$-birationally equivalent (over $K$) to $(V,s)^\#$.
\end{enumerate}
\end{fact}

Note that Fact~\ref{onit} essentially reduces the study of finite dimensional $\d$-varieties to the study of first order differential equations.

\begin{example}
In the case when a finite dimensional $\d$-variety $W\subseteq \U$ is given by a single $\d$-polynomial equation of the form $\d^n x=f(x,\d x,\dots,\d^{n-1} x)$ where $f\in K[t_0,\dots,t_{n-1}]$, it is easy to find a D-variety $(V,s)$ such that $W$ is $\d$-birationally equivalent to $(V,s)^\#$. Let $V=\U^n$ and $s(x_0,\dots,x_{n-1})=(x_0,\dots,x_{n-1}, x_1,\dots,x_{n-1},f(x_0,\dots,x_{n-1}))$. Then $(V,s)$ is a D-variety, and the map from $W$ to 
$$(V,s)^\#=\{(a_0,\dots,a_{n-1})\in V:  \d^n a_0=f(a_0,\dots,a_{n-1}) \text{ and } \d^i a_0=a_i, i=1,\dots,n-1\}$$
given by $x\mapsto (x,\d x,\dots,\d^{n-1}x)$ is a $\d$-birational equivalence (in fact a $\d$-isomorphism in this case).
\end{example}

If $G$ is an algebraic group then $\tau G$ has naturally the structure of an algebraic group: if $p:G\times G\to G$ is the group operation, then the group operation on $\tau G$ is given by $\tau p: \tau G\times \tau G\to \tau G $ (this uses the fact that $\tau$ commutes with products). An \emph{algebraic D-group} is a D-variety $(G,s)$ where $G$ is an algebraic group and $s:G\to \tau G$ is a group homorphism. In the presence of a group structure, part~(b) of Fact~\ref{onit} can be refined to show that every finite dimensional definable group in $DCF_0$ is definably isomorphic (rather than just $\d$-birationally equivalent) to $(G,s)^\#$ for some algebraic D-group $(G,s)$.

The rest of this chapter is devoted to extend these notions/properties to the setting of several commuting derivations. Our approach will be to relativize the prolongation with respect to a given partition of the set of derivations. Hence, the role of algebraic varieties in the above discussion will be played by differential algebraic varieties involving a subset of the set of derivations, and the role of $\d$ will be played by the remaining derivations. In particular, an analogue of Fact~\ref{onit} will thus give us a geometric category in which to study possibly infinite dimensional differential algebraic varieties. The differential algebraic machinery (i.e., differential derivations) that we will require to accomplish this will be presented in the next section.

\section{Differential derivations}\label{dider}

If $K<F$ is an extension of fields and $\d$ is a derivation on $K$, then one can always extend $\d$ to a derivation on $F$. The situation is quite different if we start with a field $K$ equipped with several commuting derivations $\D$, a $\D$-field extension $F$, and a distinguished additional derivation $D$ on $K$ commuting with $\D$.  If we want to extend $D$ to a derivation on $F$ such that the extension still commutes with $\D$, then we need to find solutions to a certain system of $\D$-equations which might not have a solution in $F$. However, as we will see, this system will always have a solution in some $\D$-field extension of $F$. This latter fact was proven by Kolchin in Chapter 0.4 of \cite{Ko2}. In this section we give a brief review of this result.

Let $(R,\D)$ be differential ring and $(S,\D)$ a $\D$-ring extension of $R$. 

\begin{definition}
Suppose $D:R\to S$ is a derivation. Then $D$ is called a \emph{$\D$-derivation} if $D$ commutes with $\D$ on $R$. 
\end{definition}

Fix a $\D$-derivation $D:R\to S$. Let $f\in R\{x\}_\D$, where $x=(x_1,\dots,x_n)$. By $f^D$ we mean the $\D$-polynomial in $S\{x\}_\D$ obtained by applying $D$ to the coefficients of $f$. 

\begin{remark}\label{super}
The map $f\mapsto f^D$ is a $\D$-derivation from $R\{x\}_\D$ to $S\{x\}_\D$. Indeed, clearly the map is additive and commutes with $\D$, so we only need to check that it satisfies the Leibniz rule on monomials. Recalling that $\T_\D$ is the set of derivative operators (see page 10), we need to consider $\displaystyle f(x)=c\prod_{\t\in \T_\D, i\leq n}(\t x_i)^{\alpha_{\t,i}}$ and $\displaystyle g(x)=d\prod_{\t\in \T_\D, i\leq n}(\t x_i)^{\beta_{\t,i}}$ where $c$, $d\in R$ and $\alpha_{\t,i}=\beta_{\t,i}=0$ except finitely many times. Then
\begin{eqnarray*}
(fg)^D(x)
&=& D(cd)\prod_{\t\in \T_\D, i\leq n}(\t x)^{\alpha_{\t,i}+\beta_{\t,i}} \\
&=& (D(c)d+cD(d))\prod_{\t\in \T_\D, i\leq n}(\t x)^{\alpha_{\t,i}+\beta_{\t,i}}\\
&=& f^D(x)g(x) +f(x)g^D(x).
\end{eqnarray*}
\end{remark}

\begin{lemma}\label{exten1}
Suppose $a=(a_1,\dots,a_n)$ is a tuple from $S$ and $D':R\{a\}_\D\to S$ is a $\D$-derivation extending $D$. Then, for all $f\in R\{x\}_\D$,  we have
\begin{displaymath}
D' f(a)=\sum_{\t\in \T_\D, i\leq n}\frac{\partial f}{\partial(\t x_i)}(a)\t D' a_i+f^D(a).
\end{displaymath}
\end{lemma}
\begin{proof}
By additivity it suffices to prove this for monomials $\displaystyle f(x)=c\prod_{\t\in \T_\D, i\leq n}(\t x_i)^{\alpha_{\t,i}}$ where $c\in R$ and $\alpha_{\t,i}=0$ except finitely many times,
\begin{eqnarray*}
D' f(a) 
&=& \sum_{\eta\in \T_\D,k\leq n}\; \prod_{\t\neq \eta, i\neq k} c (\t a_i)^{\alpha_{\t,i}} \; \alpha_{\eta,k} (\eta a_k)^{\alpha_{\eta,k}-1} D' \eta a_k+ D(c)\prod_{\t\in \T_\D, i\leq n}(\t a_i)^{\alpha_{\t,i}} \\
&=& \sum_{\t\in T_\D, i\leq n}\frac{\partial f}{\partial(\t x_i)}(a)\t D' a_i+f^D(a) \\
\end{eqnarray*}
where in the last equality we used that $D'\t a_i=\t D'a_i$ for all $\t\in \T_\D$ and $i\leq n$.
\end{proof}

Due to the last property, it is natural to study the following operator.

\begin{definition}\label{et}
Let $f\in R\{x\}_\D$. We define the $\D$-polynomial $\et f\in S\{x,u\}_\D$ by
\begin{displaymath}
\et f(x, u):= \sum_{\t\in \T_\D, i\leq n}\frac{\partial f}{\partial(\t x_i)}(x)\t u_i+f^D(x).
\end{displaymath}
where $u=(u_1,\dots,u_n)$. For a tuple $a$ from $S$, we will write $\et f_a(u)$ instead of $\et f(a,u)$.
\end{definition}

Thus, Lemma~\ref{exten1} says that $D' f(a)=\et f(a,D'a)$. Note that, as only finitely many algebraic indeterminates $\t x_i$ appear in $f$, the above displayed sum is a finite sum and so indeed $\et f\in S\{x,u\}_\D$. 

The map $f\mapsto \et f$ from $R\{x\}_\D\to S\{x,u\}_\D$ is clearly additive. Using Remark~\ref{super}, it is not hard to verify that 
\begin{equation}\label{etder}
\et (fg)(x,u)=\et f(x,u) g(x)+f(x)\et g(x,u).
\end{equation}
Hence, if we identify $R\{x\}_\D$ as a $\D$-subring of $S\{x,u\}_\D$, we get that $\et :R\{x\}_\D\to S\{x,u\}_\D$ is a derivation extending $D$. Moreover, as $\et \t x_i=\t u_i=\t \et x_i$ for all $\t\in \T_\D$ and $i\leq n$, $\et$ is in fact a $\D$-derivation. 

Note that, under the assumptions of Lemma~\ref{exten1}, for all $f\in R\{x\}_\D$ we have $Df(a)=\et f_a(D'a)$, and so, if $f\in\I(a/R)_\D$, we get $\et f_a(D'a)=0$. Thus, any $\D$-derivation $D'$ from $R\{a\}_\D$ to $S$ extending $D$ yields a tuple $D'a=(D'a_1,\dots,D' a_n)$ of $S$ at which $\et f_a$ vanishes for all $f\in\I(a/R)_\D$. The following fact is the converse of this implication and gives the criterion for when a $\D$-derivation can be extended to a finitely generated $\D$-ring extension.

\begin{fact}\label{exx}
Let $a$ be a tuple from $S$. Suppose there is a tuple $b$ from $S$ such that for every $f\in \I(a/R)_\D$,
\begin{displaymath}
\et f_a(b)=0.
\end{displaymath}
Then there is a unique $\D$-derivation $D':R\{a\}_\D\to S$ extending $D$ such that $D'a=b$.
\end{fact}

Thus if we want to extend $D$ to a $\D$-derivation from $R\{a\}_\D$ to $S$, we need to find a solution to the system of $\D$-equations $\{\et f_a(u)=0: f\in\I(a/R)_\D\}$. In the case when $S$ is a field (of characteristic zero), this system does have a solution in some $\D$-field extension of $S$. Indeed, the $\D$-ideal $$\left[\et f_{a}(u): f\in\I(a/R)_\D\right]$$ of $S\{u\}_\D$ is prime. From this and Fact~\ref{exx} one obtains:

\begin{fact}\label{exx2}
Suppose $(K,\D)$ is a differentially closed field extending $(R,\D)$ and $D:R\to K$ a $\D$-derivation. Then there is a $\D$-derivation $D':K\to K$ extending $D$.
\end{fact}

\section{The relative prolongation of a differential algebraic variety}\label{defrel}

For the rest of this chapter we fix a sufficiently large saturated $(\U,\Pi)\models DCF_{0,m}$ and a small $\Pi$-subfield $K$. We assume a (disjoint) partition $\Pi=\DD\cup\D$ with $\DD=\{D_1,\dots,D_r\}\neq\emptyset$. Note that $\DD$ is a set of commuting $\D$-derivations.

\begin{remark}\label{extlin}
Suppose $\Pi'$ is a basis of the space of derivations on $\U$ obtained by taking linear combinations of the elements of $\Pi$ over the constant field $K^\Pi$. We denote this $K^\Pi$-vector space by $\operatorname{span}_{K^\Pi}\Pi$. Then $(\U,\Pi')$ is again a saturated model of $DCF_{0,m}$. Indeed, $(\U,\Pi')$ is interdefinable with $(\U,\Pi)$ over $K^\Pi$. Thus, if $\DD'\cup \D'$ is a partition of $\Pi'$, then one can apply all the definitions and results of this chapter to that partition as well.
\end{remark}

Let $D\in \DD$. Recall, from the previous section, that $D:K\to K$ induces a $\D$-derivation $\et :K\{x\}_\D\to K\{x,u\}_\D$ (see Definition~\ref{et}).  We can extend $\et$ to all $\D$-rational functions over $K$ using the quotient rule; that is, if $f=\frac{g}{h}\in K\l x\r_\D$ then $$\et f(x,u)=\frac{\et g(x,u) \, h(x)-g(x)\,\et h(x,u)}{(h(x))^2}.$$ Lemma~\ref{exten1} extends to all $\D$-rational functions over $K$, and so, for any tuple $a$ from $\U$ and $f\in K\l x\r_\D$, we have $$Df(a)=\et f(a,Da).$$ If $f=(f_1,\dots,f_s)$ is a sequence of $\D$-rational functions over $K$ by $\et f$ we mean $(\et f_1,\dots,\et f_s)$. If $a$ is an $n$-tuple, we write $\et f_{a}(u)$ instead of $\et f(a,u)$.

\begin{remark}
For the first part of this section we restrict ourselves to the category of affine $\D$-varieties, since the presentation in the affine setting is more intuitive and concrete. Then, in the second part, we extend the constructions and concepts to the larger category of abstract $\D$-varieties.
\end{remark}

\begin{definition}\label{fipro}
Let $V$ be an affine $\D$-variety defined over $K$. The \emph{relative prolongation of $V$ (w.r.t $D/\D$)}, denoted by $\tau_{D/\D} V$, is the affine $\D$-algebraic variety defined by
\begin{displaymath}
f(x)=0 \; \text{ and } \; \et f(x,u)=0,
\end{displaymath}
for all $f$ in $\I(V/K)_{\D}$. The map $\pi:\tau_{D/\D} V\to V$ denotes the projection onto the first $n$ coordinates and, if $a$ is a point in $V$, we denote the fibre of $\tau_{D/\D}V$ above $a$ by $\tau_{D/\D}V_a$.
\end{definition}

\begin{remark}
Note that we allow the possibility of $\D=\emptyset$. In this case $V$ is an affine algebraic variety and the relative prolongation $\tau_{D/\emptyset}V$ is equal to the prolongation $\tau V$ of Section~\ref{repro} with respect to the ordinary differential field structure $(K,D)$.
\end{remark}

It turns out that in the definition of $\tau_{D/\D} V$ we can restrict the polynomials $f$ to a set that differentially generates $\I(V/K)_\D$ (it is worth pointing out that we might not be able to find a generating set whose elements have bounded order \cite{Ko3}).

\begin{lemma}\label{togen}
Let $A\subset K\{x\}_\D$ be such that $[A]=\I(V/K)$. Then $\tau_{D/\D} V$ is defined by $f(x)=0$ and $\et f(x,u)=0$ for all $f\in A$.
\end{lemma}
\begin{proof}
Let $a\in V$ and $b$ such that $\et f_a(b)=0$ for all $f\in A$, we need to show that $\et g_a(b)=0$ for all $g\in \I(V/K)_\D$. Since $\et$ is a $\D$-derivation, for each $\t\in \T_\D$, $h\in K\{x\}_\D$ and $f\in A$, we have that $$\et(h\t f)_a(b)=\et h_a(b)\t f(a)+h(a)\t \et f_a(b) .$$ By assumption $\et f_a(b)=0$ and since $f\in A\subseteq\I(V/K)_\D$ we get $\t f(a)=0$. Hence, $\et (h\t f)_a(b)=0$, and so it follows that $\et g_a(b)=0$ for all $g\in [A]=\I(V/K)_\D$.
\end{proof}

The previous lemma implies that the relative prolongation is independent of the $\D$-field over which $V$ is defined.

\begin{corollary}\label{change}
Let $F$ be a $\D$-subfield of $K$. If $V$ is also defined over $F$, then $\tau_{D/\D} V$ is defined by $f(x)=0$ and $\et f(x,u)=0$ for all $f\in\I(V/F)_\D$ (i.e., the defining $\D$-ideal of $V$ over $F$).
\end{corollary}
\begin{proof}
As $V$ is defined over $F$ we have that $\I(V/K)_\D=\I(V/F)_\D\cdot K\{x\}_\D$. In particular, $\I(V/K)_\D=[\I(V/F)_\D]$. Now apply Lemma~\ref{togen} with $A=\I(V/F)_\D$.
\end{proof}

\begin{lemma}\label{overc}
If $V$ is defined over the $D$-constants $K^{D}$, then $\tau_{D/\D} V=T_{\D} V$. Here $T_\D V$ is Kolchin's $\D$-tangent bundle (see Definition~\ref{koltan}).
\end{lemma}
\begin{proof}
First note that if $f\in K^D\{x\}_{\D}$ then $$\et f(x,u)=\sum_{\t\in \T_{\D},\, i\leq n} \frac{\partial f}{\partial(\t x_i)}(x)\t u_i.$$ Hence, by Corollary~\ref{change}, both $\tau_{D/\D} V$ and $T_{\D} V$ are defined by $$\left\{f(x)=0 \,\text{ and }\, \sum_{\t\in \T_{\D},\, i\leq n} \frac{\partial f}{\partial(\t x_i)}(x)\t u_i=0: \; f\in \I(V/K^D)_{\D}\right\}.$$
\end{proof}

In general, when $V$ is not necessarily defined over $K^D$, we have that each fibre $T_\D V_a$ acts regularly on $\tau_{D/\D}V_a$ by translation. In fact,

\begin{lemma}\label{torito}
$\pi:\tau_{D/\D}V\to V$ is a torsor under $\rho:T_\D V\to V$ where the fibrewise action is translation.
\end{lemma}
\begin{proof}
Let $a\in V$. For each $(a,b)\in T_\D V_a$ and $(a,c), (a,d)\in \tau_{D/\D}V_a$ we show that $(a,b+c)\in \tau_{D/\D}V_a$ and $(a,c-d)\in T_\D V_a$. Let $b=(b_1,\dots,b_n)$, $c=(c_1,\dots,c_n)$, $d=(d_1,\dots,d_n)$, and $f\in \I(V/K)_\D$, then
\begin{eqnarray*}
\et f(a,b+c)
&=& \sum_{\t\in \T_{\D},\, i\leq n} \frac{\partial f}{\partial(\t x_i)}(a)\t (b_i+c_i)+f^{D}(a) \\
&=& \sum_{\t\in \T_{\D},\, i\leq n} \frac{\partial f}{\partial(\t x_i)}(a)\t b_i+ \sum_{\t\in \T_{\D},\, i\leq n} \frac{\partial f}{\partial(\t x_i)}(a)\t c_i+f^{D}(a) \\
&=& 0 + \et f(a,c)=0.
\end{eqnarray*}
Thus, $(a,b+c)\in \tau_{D/\D}V_a$. On the other hand,
\begin{eqnarray*}
\sum_{\t\in \T_{\D},\, i\leq n} \frac{\partial f}{\partial(\t x_i)}(a)\t (c_i-d_i)
&=& \sum_{\t\in \T_{\D},\, i\leq n} \frac{\partial f}{\partial(\t x_i)}(a)\t c_i- \sum_{\t\in \T_{\D},\, i\leq n} \frac{\partial f}{\partial(\t x_i)}(a)\t d_i+f^{D}(a) \\
&=& \left(\et f(a,c) -f^D(a)\right) - \left(\et f(a,d)-f^D(a)\right) \\
&=& -f^D(a)+f^D(a)=0.
\end{eqnarray*}
Thus, $(a,c-d)\in T_\D V_a$. These facts imply that the map $T_\D V _a\times\tau_{D/\D}V_a\to \tau_{D/\D}V_a$ given by $((a,b),(a,c))\mapsto (a,b+c)$ is a well defined regular action. Hence, the regular map $T_\D V\times_V \tau_{D/\D} V\to \tau_{D/\D} V$ given on points by $((x,u),(x,v))\mapsto (x,u+v)$ induces a regular action of $T_\D V$ on $\tau_{D/\D} V$ over ~$V$.
\end{proof}

The following lemma shows that the fibres of the relative prolongation of any $K$-irreducible component of $V$ are generically the same as the fibres of the relative prolongation of $V$.

\begin{lemma}\label{irred}
Let $\{V_1,\dots,V_s\}$ be the $K$-irreducible components of $V$. If $a \in V_i\backslash\bigcup_{j\neq i}V_j$, then $\tau_{D/\D} V_{a}=\tau_{D/\D} (V_i)_{a}$.
\end{lemma}
\begin{proof}
Clearly $\tau_{D/\D} (V_i)_{a}\subseteq\tau_{D/\D} V_{a}$. Let $b\in\tau_{D/\D} V_{a}$ and $f\in \I(V_i/K)_{\D}$. Since $a$ is not in $V_j$, for $j\neq i$, we can pick a $g_j\in \I(V_j/K)_\D$ such that $g_j(a)\neq 0$. Then, if $g=\prod_j g_j$, we get $fg\in \I(V/K)_\D$ and so
\begin{displaymath}
0= \et (fg)_{a}(b) = \et f_{a}(b)g(a)+f(a)\et g_{a}(b) = \et f_{a}(b)g(a)
\end{displaymath}
where the third equality holds because $a\in V_i$. Since $g(a)\neq 0$, we have that $\et f_{a}(b)=0$, and so $b\in\tau_{D/\D} (V_i)_{a}$.
\end{proof}

\begin{remark}
The relative prolongation can be described in terms of local $\D$-derivations along the lines of Kolchin's definition of the $\D$-tangent bundle (see \cite{Ko2}, Chap. 8, \S 2). We give a brief explanation. Let $V$ be an affine $\D$-variety defined over $K$ with coordinate $\D$-ring $$K\{ V \}_{\D}:=K\{x\}_{\D}/\I\,,$$
where $x=(x_1,\dots,x_n)$. Let $a\in V(F)$, where $F$ is a $\D$-extension of $K$. The \emph{local} $\D$\emph{-ring of} $V$ \emph{at} $a$ \emph{over} $K$, $\mathcal{O}_a(V/K)_{\D}$, is just the localization of $K\{V\}_{\D}$ at the prime $\D$-ideal $\P:=\{f\in K\{V\}_{\D}:f(a)=0\}$. Note that $\mathcal{O}_a(V/K)_{\D}$ is a $\D$-ring extension of $K$. A \emph{local} $\D$\emph{-derivation at} $a$ is an additive map $\xi:\mathcal{O}_a(V/K)_\D\to F$ commuting with $\D$ such that $\xi(f g)=\xi(f)g(a)+f(a)\xi(f)$ for all $f,g\in \mathcal{O}_a(V/K)_\D$. If $\xi$ is a local $\D$-derivation at $a$ extending $D$, then $(\xi(\bar x_1),\dots,\xi(\bar x_n))$ is a solution in $F^n$ to the system $\{\et f_a(u)=0: f\in \I(V/K)_{\D}\}$, where $\bar x_i$ is the image in $\mathcal{O}_a(V/K)_\D$ of $x_i+\I$. Conversely, every solution $b$ in $F^n$ to the system $\{\et f_a(u)=0: f\in \I(V/K)_{\D}\}$ gives rise to a local $\D$-derivation at $a$ extending $D$ defined by
\begin{displaymath}
\xi_b (f)=\et f_a(b),
\end{displaymath}
and it satisfies $(\xi_b(\bar x_1),\dots,\xi_b(\bar x_n))=b$. Thus, the set of local $\D$-derivations at $a$ extending $D$ can be identified with the set of solutions in $F$ of the system $\{\et f_a(u)=0: f\in \I(V/K)_{\D}\}$. Hence, the $F$-points of $\tau_{D/\D}V_a$ can be identified with the set of all local $\D$-derivations at $a$ extending $D$. 
\end{remark}

Now, let us consider the whole set $\DD=\{D_1,\dots,D_r\}$ at once, and for each affine $\D$-variety $V$ over $K$ define

\begin{definition}
The \emph{relative prolongation of $V$ (w.r.t. $\DD/\D$)} is the fibred product:
\begin{displaymath}
\tau_{\DD/\D} V:=\tau_{D_1/\D}V\times_V \cdots\times_V \tau_{D_r/\D}V.
\end{displaymath}
If $V$ and $W$ are affine $\D$-varieties and $f:V\to W$ is a $\D$-regular map (all defined over $K$), then we have an induced $\D$-regular map, $\tau_{\DD/\D} f:\tau_{\DD/\D} V\to \tau_{\DD/\D}W$ defined over $K$, given on points by
\begin{equation}\label{expta}
(x,u_1,\dots,u_s)\mapsto (f(x),d_{D_1/\D}f(x,u_1),\dots,d_{D_r/\D}f(x,u_r)).
\end{equation}
\end{definition}

\begin{proposition}\label{fun}
$\tau_{\DD/\D}$ is a functor from the category of affine $\D$-algebraic varieties (defined over $K$) to itself. Moreover, $\tau_{\DD/\D}$ commutes with products; that is, $\tau_{\DD/\D}(V\times W)$ is naturally isomorphic to $\tau_{\DD/\D}V\times \tau_{\DD/\D}W$ via the map $$(x,y,u_1,\dots,u_r,v_1\dots,v_r)\mapsto(x,u_1,\dots,u_r,y,v_1,\dots,v_r).$$
\end{proposition}
\begin{proof}
For ease of notation we assume $\DD=\{D\}$. We need to show that for each pair of $\D$-regular maps $f:V\to W$ and $g:W\to U$ we have that $\tau_{D/\D}(g\comp f)=\tau_{D/\D}g\,\comp\,\tau_{D/\D}f$. By the definition of $\tau_{D/\D}$ it suffices to show that $\et(g\comp f)(a,b)=\et g(f(a),\et f(a,b))$ for all $(a,b)\in\tau_{D/\D} V$. Write $f=(f_1,\dots,f_s)$. For each $\sigma\in \T_\D$ and $j\leq s$ we have that
\begin{displaymath}
\sum_{\t\in \T_\D,\, i\leq n}\frac{\partial \sigma f_j}{\partial \t x_i}(a)\t b_i=\sum_{\t\in \T_\D,\, i\leq n} \sigma\left(\frac{\partial f_j}{\partial \t x_i}(a)\,\t b_i\right),
\end{displaymath}
and hence 
\begin{displaymath}
\sum_{\t\in \T_\D,i\leq n}\frac{\partial (g\comp f)}{\partial \t x_i}(a)\t b_i=\sum_{\s\in \T_\D, j\leq s}\frac{\partial g}{\partial \s y_j}(f(a))\s \left(\sum_{\t\in\T_\D,i\leq n} \frac{\partial f_j}{\partial \t x_i}(a)\t b_i\right).
\end{displaymath}
From this we get
\begin{eqnarray*}
(g\comp f)^D(a)
&=& D(g(f(a)))-\sum_{\t\in \T_\D,i\leq n}\frac{\partial (g\comp f)}{\partial \t x_i}(a)\t D a_i \\
&=& \sum_{\s \in\T_\D,j\leq s}\frac{\partial g}{\partial \s y_j}(f(a))\s (Df_j(a))+g^D(f(a))-\sum_{\t\in \T_\D,i\leq n}\frac{\partial (g\comp f)}{\partial \t x_i}(a)\t a_i \\
&=& \sum_{\sigma\in \T_\D,\, j\leq s}\frac{\partial g}{\partial \sigma y_j}(f(a))\sigma(f^D_j(a))+g^D(f(a))=\et g(f(a),f^D(a)).
\end{eqnarray*}

Finally, using these equalities we have that
\begin{eqnarray*}
\et(g\comp f)(a,b)
&=& \sum_{\t\in \T_\D,i\leq n}\frac{\partial (g\comp f)}{\partial \t x_i}(a)\t b_i+(g\comp f)^D(a) \\
&=& \sum_{\sigma\in \T_\D,\, j\leq s}\frac{\partial g}{\partial \sigma y_j}(f(a))\sigma\left(\sum_{\t\in\T_\D,\,i\leq n}\frac{\partial f_j}{\partial \t x_i}(a)\t b_i +f^D_j(a)\right)+g^D(f(a)) \\
\vspace{2in}
&=& \et g(f,\et f)(a,b).
\end{eqnarray*}

The moreover clause follows from the fact that $\I(V\times W/K)_\D$ equals the ideal generated by $\I(V/K)_\D\subseteq K\{x\}_\D$ and $\I(W/K)_\D\subseteq K\{y\}_\D$ in $K\{x,y\}_\D$.
\end{proof}

Let $V$ be an affine $\D$-variety defined over $K$. Recall that for each tuple $a$ from $\U$, $D\in \DD$ and $f\in K\{x\}_\D$ we have that $Df(a)=\et f(a,Da)$, and so for each $a\in V$ and $g\in \I(V/K)_\D$ we have $$\et g(a,Da)=Dg(a)=0.$$ Hence, $(a,D_1 a,\dots,D_r a)\in \tau_{\DD/\D}V$ for all $a\in V$. This motivates the following definition.

\begin{definition}
For each affine $\D$-variety $V$ defined over $K$, we define the map $\nabla=\nabla_\DD^V$ from $V$ to $\tau_{\DD/\D}V$ by $$\nabla a=(a,D_1 a,\dots,D_r a).$$
Clearly $\nabla$ is a $\Pi$-regular section of $\pi:\ta V\to V$, i.e., a $\Pi$-regular map such that $\pi\comp\nabla=\operatorname{Id}_V$. 
\end{definition}

\begin{remark}\label{commu}
If $f:V\to W$ is a $\D$-regular map between affine $\D$-algebraic varieties (all defined over $K$), then $\nabla_\DD^W\comp f=\tau_{\DD/\D}f\comp\nabla_\DD^V$. In other words, the following diagram commutes
$$\xymatrix{
\ta V \ar[rr]^{\ta f}&&\ta W\\
V \ar[u]^{\nabla_\DD^V}\ar[rr]^{f}&&W\ar[u]_{\nabla_\DD^W}
}$$
Indeed, since for every tuple $b$ and $D\in \DD$ we have that $Df(b)=\et f(b,Db)$, for every $a\in V$  $$\nabla_\DD^W(f(a))=(a,\et f(\nabla_\DD^V a))=\tau_{\DD/\D}f( \nabla_\DD^V a).$$
\end{remark}

\noindent {\bf Abstract relative prolongations.} While the definitions and results of this section were stated for affine $\D$-varieties (for the sake of concreteness), all of them make sense and hold for abstract $\D$-varieties. In the remainder of this section we discuss how to make this precise by presenting in detail the construction of the \emph{abstract relative prolongation} and the induced map $\nabla$. 

First, for a quasi-affine $\D$-variety $U$ defined over $K$, set $\ta U$ to be the preimage of $U$ under $\pi:\ta V\to V$ where $V$ is the $\D$-closure (in the $\D$-topology) of $U$ over $K$. If $f:U\to U'$ is $\D$-regular map between quasi-affine $\D$-varieties (defined over $K$), we define $\ta f:\ta U\to \ta U'$ by specifying it on points exactly as we did for the affine case (see (\ref{expta}) above).

Let $(V,\{V_i\},\{f_i\})$ be an (abstract) $\D$-variety defined over $K$ (see Definition~\ref{absvar}), where $\{V_i\}_{i=1}^s$ is an open cover and $(V_i,f_i)$ are the local charts. Let $U_i=f_i(V_i)$ be the affine $\D$-varieties defined over $K$ that form the affine cover of $V$, $U_{ij}=f_i(V_i\cap V_j)$, and $f_{ij}=f_i\comp f_j^{-1}:U_{ji}\to U_{ij}$. 

\begin{definition}
The \emph{relative prolongation of $V$ (w.r.t. $\DD/\D$)}, denoted by $\ta V$, is constructed as follows. Let $X$ be the disjoint union of the $\ta(U_i)$'s and define on $X$ an equivalence relation $E$ where $aEb$ if and only if $\ta f_{ij}(a)=b$ for some $i$ and $j$. Then $\ta V$ is defined as $X/E$ with the induced topology. We define the projection $\pi:\ta V \to V$ by $\pi(\al)=f_i^{-1}(\pi_i(a))$ where $a\in \ta U_i $ is a representative of the $E$-class $\al$ and $\pi_i:\ta U_i \to U_i$. 
\end{definition}

Let us check that the map $\pi:\ta V \to V$ is well defined. Suppose $aEb$, then $\ta f_{ij}(a)=b$, and applying the projection $\pi_i$ we get $f_i\comp f_j^{-1}(\pi_j(a))=\pi_i(b)$. Thus, $f_j^{-1}(\pi_j(a))=f_i^{-1}(\pi_i(b))$, as desired.

A similar construction allows us to define the $\D$-tangent bundle $\rho:T_\D V\to V$ as an abstract $\D$-variety. From the affine case (Lemma~\ref{torito}), we can deduce that in the abstract category we still have that $\pi:\ta V\to V$ is a torsor under $\rho:T_\D V\to V$.

We can also define the $\Pi$-regular section $\nabla=\nabla_\DD^V$ of $\pi:\ta V\to V$ by $$\nabla(a)=\nabla_\DD^{U_i}(f_i(a))/E, \; \text{ for } a\in V_i.$$ This is well defined. Indeed, suppose $a\in V_j$.  Then 
$$\ta f_{ji}(\nabla_\DD^{U_i} f_i(a))=\nabla_\DD^{U_j}f_{ji}(f_i(a))=\nabla_\DD^{U_j}f_j(a),$$
where the first equality follows from Remark~\ref{commu}. This shows that $\nabla_\DD^{U_i} f_i(a)\,E\,\nabla_\DD^{U_j}f_j(a)$, as desired.

Given a $\D$-regular map $f:V\to W$ between abstract $\D$-varieties $(V,\{V_i\},\{f_i\})$ and $(W,\{W_i\},\{g_i\})$. We define the $\D$-regular map $\ta f:\ta V\to \ta W$ by $$\ta f(\al)=\ta(g_j\comp f\comp f_i^{-1})(a)/E_W, \; \text{ for all } \al\in \ta V,$$ where $a\in \ta U_i$ is a representative of the $E_V$-class of $\al$ and $f\comp f_i^{-1}(a)\in W_j$. 

Now that we have the appropriate constructions, we can deduce, from the affine case, the results of this section for abstract $\D$-varieties. 

From now on, unless stated otherwise, we will work in the more general category of abstract $\D$-varieties, which we simply call $\D$-varieties.

\

\section{Relative D-varieties}\label{relDvar}

In this section we present the theory of relative D-varieties. Our main goal is to introduce a convenient category for studying the geometry of those differential algebraic varieties that are not of maximal differential type (see Proposition~\ref{finteo}).

We continue to work in our universal domain $(\U,\Pi)$ with a partition $\Pi=\DD\cup\D$ where $\DD=\{D_1,\dots,D_r\}$, and over a base $\Pi$-field $K<\U$. 

Let us start with the definition of relative D-variety. 

\begin{definition}\label{dvar}
A \emph{relative D-variety (w.r.t. $\DD/\D$) defined over $K$} is a pair $(V,s)$ where $V$ is a $\D$-variety and $s$ is a $\D$-regular section of $\pi: \tau_{\DD/\D}V\to V$ (both defined over $K$). We require the following \emph{integrability condition} on $s$: for each $a\in V$,
\begin{equation}\label{rel1}
d_{D_i/\D}s_j(a, s_i(a))= d_{D_j/\D}s_i(a, s_j(a)) \quad \text{ for } i,j=1,\dots r,
\end{equation}
where $(\operatorname{Id},s_1,\dots,s_r):U\to \ta U$ are coordinates of $s$ in a local chart $U$ containing $a$. By a \emph{relative D-subvariety of} $(V,s)$ we mean a relative D-variety $(W,s_W)$ such that $W$ is a $\D$-subvariety of $V$ and $s_W=s|_W$. We say $(V,s)$ is an affine relative D-variety if $V$ is an affine $\D$-variety.
\end{definition}

Note that if $\D=\emptyset$ we recover the definition of algebraic D-variety, see Section~\ref{repro}.

A morphism between relative D-varieties $(V,s)$ and $(W,t)$ is a $\D$-regular map $f:V\to W$ satisfying $\tau_{\DD/\D}f\comp s=t\comp f$; that is, the following diagram commutes
$$\xymatrix{
\ta V \ar[rr]^{\ta f}&&\ta W\\
V \ar[u]^{s}\ar[rr]^{f}&&W\ar[u]_{t}
}$$

To every relative D-variety $(V,s)$ defined over $K$ we can associate a $\Pi$-variety (also defined over $K$) given by
\begin{displaymath}
(V,s)^{\sharp}=\{\, a\in V : \, s(a)=\nabla(a)\}.
\end{displaymath}
A point in $(V,s)^{\sharp}$ is called a \emph{sharp point} of $(V,s)$.

\begin{example}\label{lastex}
In this example we illustrate the fact that in the absence of the integrability condition~\ref{rel1} there may not be any sharp point. Consider the case of two derivations $(\U,\Pi=\{D_1,D_2\})$. Let $V=\U$, $s_1(x)=x$ and $s_2(x)=x+1$. Then 
$$s=(\operatorname{Id},s_1,s_2):\U\to \tau_{\Pi/\emptyset}V=\U^3$$ 
is a section of $\tau_{\Pi/\emptyset}V$. Note that $s$ does not satisfy the integrability condition. Indeed,
$$d_{D_1/\emptyset}s_2(x,s_1(x))=x \text{ and } d_{D_2/\emptyset}s_1(x,s_2(x))=x+1.$$
We now show that there is no point $a\in V$ such that $s(a)=\nabla(a)$. Towards a contradiction suppose $a\in V$ is such that $s(a)=\nabla(a)$. On the one hand we have
$$\d_2\d_1 a=\d_2(a)=a+1.$$
On the other hand,
$$\d_1\d_2 a=\d_1(a+1)=\d_1 a=a.$$
Since $\d_2\d_1 a=\d_1\d_2 a$, we get $1=0$, and so we have a contradiction.
\end{example}

\begin{proposition}\label{lema1} Let $(V,s)$ be a relative D-variety defined over $K$.
\begin{enumerate}  
\item $(V,s)^{\sharp}$ is $\D$-dense in $V$.
\item Assume $V$ is $K$-irreducible (as a $\D$-variety). There exists a $\D$-generic point $a$ of $V$ over $K$ such that $a\in (V,s)^{\sharp}$. Moreover, any such $a$ is a $\Pi$-generic point of $(V,s)^\#$ over $K$. In particular, $(V,s)^\#$ will be $K$-irreducible.
\item Assume $V$ is $K$-irreducible (as a $\D$-variety). Let $\omega_V$ be the Kolchin polynomial of $V$ as a $\D$-algebraic variety and let $\omega_{(V,s)^\#}$ be the Kolchin polynomial of $(V,s)^\#$ as a $\Pi$-algebraic variety. Let $\mu$ be the smallest positive integer such that $s_i\in K(\t x:\t\in \T_\D(\mu))$, for all $i=1,\dots, r$, whenever we write $s$ in coordinates $s=(\operatorname{Id},s_1,\dots,s_r),$ in a local chart. Then for sufficiently large $h\in\NN$
\begin{displaymath}
\omega_{(V,s)^{\sharp}}(h)\leq \omega_{V}(\mu  h). 
\end{displaymath}
In particular, $\D$-type$(V)=\Pi$-type$(V,s)^{\sharp}$. 
\end{enumerate}
\end{proposition}

\begin{proof} 
It is sufficient to consider the affine case.

\noindent (1) Let $F<\U$ be a $\Pi$-closed field extension of $K$ and $O$ be any (non-empty) $\D$-open subset of $V$ defined over $F$. We need to find a tuple $c$ from $F$ such that $c\in (V,s)^{\sharp}\cap O$. Let $W$ be an irreducible component of $V$ such that $O\cap W\neq \emptyset$ and let $a$ be a $\D$-generic point of $W$ over $F$. Clearly $a\in O$. 

Now let $f\in \I(a/F)_{\D}=\I(W/F)_{\D}$. Since $a$ is a generic point of $W$ over $F$, we can find $g$ such that $g(a)\neq 0$ and $fg\in \I(V/F)_{\D}$. Thus, for $i=1,\dots,r$,
\begin{eqnarray*}
0
&=&d_{D_i/\D} (fg)_{a}(s_i(a)) \\
&=&d_{D_i/\D} f_{a}(s_i(a))g(a)+f(a)d_{D_i/\D} g_{a}(s_i(a)) \\
&=&d_{D_i/\D} f_{a}(s_i(a))g(a)
\end{eqnarray*}
since $f(a)=0$. But since $g(a)\neq 0$, we get $d_{D_i/\D}f_{a}(s_i(a))=0$ for $i=1,\dots,r$. Fact~\ref{exx} now implies that there are $\D$-derivations $D_i':F\{a\}_{\D}\to F\{a,s_i(a)\}_{\D}=F\{a\}_{\D}$ extending $D_i:F\to F$ such that $D_i'(a)=s_i(a)$, for $i=1,\dots,r$.

Let us check that $D_i'$ and $D_j'$ commute on $F\{a\}_{\D}$; that is, for $f\in F\{x\}_{\D}$, $[D_i',D_j']f(a)=0$. A rather lengthy computation shows that 
\begin{displaymath}
[D_i',D_j']f(a)=d_{[D_i,D_j]/\D}f(a,[D_i',D_j']a).
\end{displaymath}
Because $D_i$ and $D_j$ commute on $F$ we get $f^{[D_i,D_j]}=0$, so we only need to check that $[D_i',D_j']a=0$. Here we use the integrability condition (\ref{rel1}). We have
\begin{eqnarray*}
[D_i',D_j']a
&=& D_i' s_j(a)-D_j' s_i(a)\\
&=&d_{D_i/\D}s_{j}(a, D_i'a)-d_{D_j/\D}s_{i}(a,D_j'a)\\
&=& d_{D_i/\D}s_{j}(a, s_i(a))-d_{D_j/\D}s_{i}(a, s_j(a))=0.
\end{eqnarray*}
Thus, the $\D$-derivations $\DD'=\{D_{1}',\dots,D_r'\}$ commute on $F\{a\}_{\D}$. So, $F\{a\}_{\D}$ together with $\DD'$, is a differential ring extension of $(F,\Pi)$ and it has a tuple, namely $a$, living in $(V,s)^{\sharp}\cap O$. Hence, there is $c$ in $F$ such that $c\in (V,s)^{\sharp}\cap O$ as desired.

\noindent (2) Consider the set of formulas
\begin{displaymath}
\Phi=\{x\in (V,s)^{\sharp}\}\cup\{x\notin W: W  \; \textrm{is a proper} \; \D\textrm{-algebraic subvariety of } V \textrm{ over } K \}
\end{displaymath}
if this set were inconsistent, by compactness and $K$-irreducibility of $V$, we would have that $(V,s)^{\sharp}$ is contained in a proper $\D$-algebraic subvariety of $V$, but this is impossible by part (1). Hence $\Phi$ is consistent and a realisation is the desired point. The rest is clear.

\noindent (3) Let $a$ be a $\D$-generic point of $V$ over $K$ such that $a\in (V,s)^{\sharp}$ (this is possible by (2)). Then, since $\nabla(a)=s(a)$, for any $D\in \DD$ we can write $Da$ as a $\D$-rational function from $K(\t x: \t\in\T_\D(\mu))$ evaluated at $a$. Thus, we get that for each $h\in\NN$,
\begin{displaymath}
K(\t a: \, \t \in \T_{\Pi}(h))\subseteq K(\t a: \, \t\in \T_{\D}(\mu h)).
\end{displaymath} 
\end{proof}

Next we show that morphisms between relative D-varieties are precisely those regular $\D$-maps preserving sharp points.

\begin{lemma}
Let $(V,s)$ and $(W,t)$ be relative D-varieties defined over $K$. A regular $\D$-map $f:V\to W$ is a morphism of relative D-varieties if and only if $\displaystyle f\left((V,s)^\#\right)\subseteq (W,t)^\#$.
\end{lemma} 
\begin{proof}
Suppose $f$ is a morphism of relative D-varieties and let $a\in (V,s)^\#$. We need to show that $f(a)\in (W,t)^\#$, i.e., $t(f(a))=\nabla_\DD^W(f(a))$. By Remark~\ref{commu}, we have $$\nabla_\DD^W\comp f(a)=\ta f\comp \nabla_\DD^V (a)=\ta f\comp s(a)=t\comp f (a),$$ as desired. Conversely, assume $f\left((V,s)^\#\right)\subseteq (W,t)^\#$. For all $a\in (V,s)^\#$ we have $$\ta f(s(a))=\ta f(\nabla_\DD^V(a))=\nabla_\DD^W(f(a))=t(f(a)).$$
Hence, $\ta f\comp s$ and $t\comp f$ agree on all of $(V,s)^\#$, and so, by (1) of Proposition~\ref{lema1}, they agree on a $\D$-dense subset of $V$. Thus, $\ta f\comp s=t\comp f$.
\end{proof}

The previous lemma implies that we can define a (covariant) functor $\#$ from the category of relative D-varieties defined over $K$ to the category of $\Pi$-varieties defined over $K$, which is defined on objects by $(V,s)^\#$ and on morphisms by $f|_{(V,s)^\#}$.

\begin{proposition}\label{biv}
Let $(V,s)$ be an relative D-variety defined over $K$. The $\sharp$-functor establishes a bijective correspondence between relative D-subvarieties of $(V,s)$ defined over $K$ and $\Pi$-subvarieties of $(V,s)^\#$ defined over $K$. The inverse is given by taking the $\D$-closure (in the $\D$-topology of $V$) over $K$.
\end{proposition}
\begin{proof}
If $(W,s_W)$ is a relative D-subvariety of $(V,s)$ defined over $K$ then, by (1) of Proposition~\ref{lema1}, the $\D$-closure of $(W,s_W)^{\#}$ is just $W$. On the other hand, if $X$ is a $\Pi$-algebraic subvariety of $(V,s)^\#$ defined over $K$ and $W$ is its $\D$-closure over $K$, then $(W,s_W:=s|_W)$ is clearly a $\DD/\D$-subvariety of $(V,s)$. We need to show $X=(W,s_W)^\#$. Towards a contradiction suppose $X\neq (W,s_W)^\#$. Using the equations of the sharp points $s(x)=\nabla_{\DD}x$, we obtain a $\D$-algebraic variety $U$ defined over $K$ such that $X\subseteq U$ and $(W,s_W)^\# \not\subseteq U$. This would imply that $(W,s_W)^\#\not\subseteq W$, but this is impossible.
\end{proof}

So far in this section and the previous one we have been using a fixed partition $\DD\cup\D$ of $\Pi$. However, as we mentioned in Remark~\ref{extlin} at the beginning of Section~\ref{defrel}, all the definitions and results make sense and hold if we replace such a partition for any two sets $\DD$ and $\D$ of linearly independent elements of the $K$-vector space $\operatorname{span}_{K^\Pi}\Pi$ such that their union $\DD\cup\D$ forms a basis.

We now show that, up to $\Pi$-birational equivalence, every irreducible $\Pi$-variety of $\Pi$-type less than $m$ is the sharp points of a relative D-variety (with an appropriate choice of $\DD$ and $\D$ such that $\DD\cup\D$ is a basis of $\operatorname{span}_{K^\Pi}\Pi$). This is the analogue in several derivations of the well known characterization of finite rank ordinary differential algebraic varieties (see \S 4 of \cite{PiT}).

\begin{proposition}\label{finteo}
Let $W$ be a $K$-irreducible $\Pi$-variety defined over $K$ with $\Pi$-type$(W)=\ell<m$ and $\Pi$-dim$(W)=d$. Then there is a basis $\DD\cup \D$ of the $K^\Pi$-vector space $\operatorname{span}_{K^\Pi}\Pi$ with $|\D|=\ell$, and a relative D-variety $(V,s)$ w.r.t. $\DD/\D$ defined over $K$ with $\D$-type$(V)=\ell$ and $\D$-dim$(V)=d$, such that $W$ is $\Pi$-birationally equivalent (over $K$) to $(V,s)^{\sharp}$.
\end{proposition}

\begin{proof}
It suffices to consider the affine case.
Let $a$ be a $\Pi$-generic point of $W$ over $K$. Since the $\Pi$-type of $W$ is $\ell$, by Fact \ref{Kolteo}, there exists a set $\D$ of $\ell$ linearly independent elements of $\operatorname{span}_{K^\Pi}\Pi$ such that $K\la a\ra_{\Pi}=K\la\al\ra_{\D}$ for some tuple $\al$ of $\U$ with $\D$-tp$(\al/K)=\ell$ and $\D$-dim$(\al/K)=d$. Let $\DD=\{D_1,\dots,D_r\}$ be any set of derivations such that $\DD\cup\D$ is a basis of $\operatorname{span}_{K^\Pi}\Pi$.

For each $i=1,\dots,r$, we have
\begin{displaymath}
D_i \alpha=\frac{f_i(\al)}{g_i(\al)}
\end{displaymath}
where $\frac{f_i}{g_i}$ is a sequence of $\D$-rational functions over $K$. Let
\begin{displaymath}
b:= \left(\al,\frac{1}{g_{1}(\al)},\dots,\frac{1}{g_r(\al)}\right).
\end{displaymath}
Note that $b$ and $\al$ are $\D$-interdefinable over $K$, hence $\D$-type$(b/K)=\D$-type$(\al/K)$ and $\D$-dim$(b/K)=\D$-dim$(\al/K)$. Also, $a$ and $b$ are $\Pi$-interdefinable over $K$ and hence it suffices to show that $b$ is a $\Pi$-generic point of the set of sharp points of a relative D-variety w.r.t. $\DD/\D$ defined over $K$.

Let $V$ be the $\D$-locus of $b$ over $K$ (then $\D$-type$(V)=r$ and $\D$-dim$(V)=d$). A standard trick gives us a sequence $s=(\operatorname{Id},s_1,\dots,s_r)$ of $\D$-polynomials over $K$ such that $s(b)=\nabla_{\DD}b\in \tau_{\DD/\D}V$. Since $b$ is a $\D$-generic point of $V$ over $K$, we get that $(V,s)$ is a relative D-variety defined over $K$. Also, $b\in(V,s)^{\sharp}$ then, by Proposition \ref{lema1} (3), $b$ is a $\Pi$-generic point of $(V,s)^{\sharp}$ over $K$ as desired.
\end{proof}

The proof of Proposition \ref{finteo} actually proves:

\begin{corollary}\label{ong}
Let $W$ be a $K$-irreducible $\Pi$-variety defined over $K$ and $\D$ a set of fewer than $m$ linearly independent elements of $\operatorname{span}_{K^\Pi}\Pi$ that bounds $\Pi$-type$(W)$ (see Definition~\ref{defboset}). If we extend $\D$ to a basis $\DD\cup\D$ of $\operatorname{span}_{K^\Pi}\Pi$, then $W$ is $\Pi$-birationally equivalent (over $K$) to $(V,s)^\#$ for some relative D-variety $(V,s)$ w.r.t. $\DD/\D$ defined over $K$.
\end{corollary}

\

\section{Relative D-groups}\label{relDgroup}

In this section we develop the basic theory of relative D-groups which we will use in Section \ref{galex} to understand the generalized strongly normal extensions of Section \ref{genstrong} as Galois extensions of certain differential equations. 

We continue to work in our universal domain $(\U,\Pi)$ with a partition $\Pi=\DD\cup\D$, and over a base $\Pi$-field $K<\U$. We first study the additional properties that $\ta G$ has in the case when $G$ is a $\D$-algebraic group.

Let $\D'\subseteq \D$ and suposse $V$ is a $\D'$-variety defined over $K$. It is not known (at least to the author) if $\tau_{\DD/\D}V=\tau_{\DD/\D'}V$. However, this equality holds in the case of $\D'$-algebraic groups:

\begin{proposition}\label{prog}
Let $\D'\subseteq \D$ and let $G$ be a $\D'$-algebraic group defined over $K$. Then $\tau_{\DD/\D} G=\tau_{\DD/\D'}G$.
\end{proposition}
\begin{proof}
It suffices to consider the case when $G$ is affine. For ease of notation we assume $\DD=\{D\}$. By Lemma~\ref{irred} we may assume that $G$ is a connected $\D'$-algebraic group. Let $\L$ be a characteristic set of the prime $\D'$-ideal $\I(G/K)_{\D'}$. By Chapter 7 of \cite{Fre}, $G$ is also a connected $\D$-algebraic group and $\L$ is a characteristic set of $\I(G/K)_\D$. Let $a$ be a $\D$-generic point of $G$ over $K$. We claim that $\tau_{D/\D}G_a=\tau_{D/\D'}G_a$. Let $b\in \tau_{D/\D'}G_a$, then $\et f_a(b)=d_{D/\D'}f_a(b)=0$ for all $f\in \L$. It is easy to see now that $\et f_a(b)=0$ for all $f\in [\L]_\D$, where $[\L]_\D$ denotes the $\D$-ideal generated by $\L$ in $K\{x\}_\D$. Now let $g\in\I(G/K)_\D$, since $\L$ is a characteristic set of $\I(G/K)_\D$ we can find $\ell$ such that $H_\L^\ell\, g\in [\L]_\D$ , where $H_\L$ is the product of initials and separants of the elements of $\L$ (see Chap. IV, \S 9 of \cite{Ko}). Thus we have
\begin{eqnarray*}
0&=& \et \big(H_\L^\ell g\big)_a(b) \ \ \ \ \ \ \ \ \ \text{ as $H_\Lambda^\ell g\in[\L]_\D$}\\
&=& H_\L^\ell(a)\et g_a(b)+g(a)\et (H_\L^\ell)_a(b)\\
&=& H_\L^\ell(a)\et g_a(b).
\end{eqnarray*}
Since $a$ is a $\D$-generic point of $G$ over $K$, $H_\L(a)\neq 0$, and so $\et g_a(b)=0$. Hence, $\tau_{D/\D'}G_a\subseteq\tau_{D/\D}G_a$. The other containment is clear.

Now we show that $\tau_{D/\D}G_g=\tau_{D/\D'}G_g$ for all $g\in G$. Let $g\in G$, we can find $h\in G$ such that $g=ha$. Let $\lambda^h:G\to G$ denote left multiplication by $h$, then $\tau_{D/\D}\lambda^h(\tau_{D/\D} G_a)=\tau_{D/\D}G_g$ and $\tau_{D/\D'}\lambda^h(\tau_{D/\D'} G_a)=\tau_{D/\D'}G_g$. But $\tau_{D/\D}\lambda^h=\tau_{D/\D'}\lambda^h$ as $G$ is a $\D'$-algebraic group, and so, since $\tau_{D/\D}G_a=\tau_{D/\D'}G_a$, we get $\tau_{D/\D}G_g=\tau_{D/\D'}G_g$.
\end{proof}

Let $G$ be a $\D$-algebraic group defined over $K$. Then $\tau_{\DD/\D}G$ has naturally the structure of a $\D$-algebraic group defined over $K$; more precisely, if $p:G\times G\to G$ is the group operation on $G$, then $\tau_{\DD/\D}p:\tau_{\DD/\D}G\times \tau_{\DD/\D}G\to \ta G$ is a group operation on $\ta G$. Here we are identifying $\ta(G\times G)$ with $\ta G\times\ta G$ (via the natural isomorphism given in Proposition~\ref{fun}).

\begin{remark}
By Remark \ref{commu} we have $$\ta p\comp \nabla_\DD^{G\times G}=\nabla_\DD^G\comp p.$$ Hence, the section $\nabla_\DD^G:G\to\ta G$ is a group homomorphism.
\end{remark}

Let us give some useful explicit formulas for the group law of $\ta G$. For ease of notation suppose $G$ is affine. For each $f\in K\{x\}_\D$ let $$d_\D f_x u:=\sum_{\t\in\T_\D, i\leq n}\frac{\partial f}{\partial \t x_i}(x)\,\t u_i,$$ and for a tuple $f=(f_1,\dots,f_s)$ let $d_\D f_x u$ be $(d_\D(f_1)_x u,\dots,d_\D(f_s)_x u)$. 
Suppose $\lambda^g$ and $\rho^g$ denote left and right multiplication by $g\in G$, respectively, and $p$ is the group operation on $G$, then for $(g,u_1,\dots,u_r)$ and $(h,v_1,\dots,v_r)$ in $\ta G$ we have $$(g,u_1,\dots,u_r)\cdot(h,v_1,\dots,v_r)=(g\cdot h,\,d_\D (\lambda^g)_h v_i+d_\D(\rho^h)_g u_i+p^{D_i}(g,h):i\leq r).$$ 
The inverse is given by 
\begin{displaymath}
(g,u_1,\dots,u_r)^{-1}=(g^{-1},d_\D(\lambda^{g^{-1}}\comp\rho^{g^{-1}})_g(D_i g-u_i)+D_i (g^{-1}):i\leq r).
\end{displaymath}
It is clear, from these formulas, that $\ta G_e$ is a normal $\D$-subgroup of $\ta G$, where $e$ is the identity of $G$.

\

Now we show that at least for $\D$-algebraic groups the functor $\ta$ preserves irreducibility.

\begin{proposition}
If $G$ is a connected $\D$-algebraic group then $\ta G$ is also connected.
\end{proposition}
\begin{proof}
Let $H$ be the connected component of $\ta G$. First note that for any $g\in G$, since $\ta G_g$ is irreducible (as it is a translate of $T_\D G_g$) and the irreducible components of $G$ are disjoint, if $H\cap \ta G_g\neq \emptyset$ then $\ta G_g\subseteq H$. Towards a contradiction suppose that $H\neq \ta G$. Then for some $g\in G$ the intersection $H\cap\ta G_g$ is empty. Hence, $\nabla^{-1}(H)$ and $\nabla^{-1}(\ta G\setminus H)$ are proper $\Pi$-algebraic subvarieties of $G$ whose union is all of $G$. This implies that $G$ is reducible as a $\Pi$-algebraic variety, but, by Theorem~7.2.3 of \cite{Fre}, $G$ will also be reducible as a $\D$-variety. This contradicts our hypothesis.
\end{proof}

The following relativization of algebraic D-groups is obtained by considering the group objects in the category of relative D-varieties.

\begin{definition}
A \emph{relative D-group (w.r.t. $\DD/\D$) defined over $K$} is a relative D-variety $(G,s)$ such that $G$ is a $\D$-algebraic group and $s:G\to \ta G$ is a group homomorphism (all defined over $K$). A \emph{relative D-subgroup} of $(G,s)$ is a relative D-subvariety $(H,s_H)$ of $(G,s)$ such that $H$ is a subgroup of $G$.
\end{definition}

\begin{remark}  \
\begin{enumerate}
\item If $(G,s)$ is a relative D-group then $(G,s)^\#$ is a $\Pi$-subgroup of $G$. Indeed, if $g$ and $h\in (G,s)^\#$ then $$s(g\cdot h^{-1})=s(g)\cdot s(h)^{-1}=(\nabla g)\cdot (\nabla h)^{-1}=\nabla (g\cdot h^{-1}).$$ Hence, $g\cdot h^{-1}\in (G,s)^\#$.
\item Suppose $G$ is a $\D$-algebraic group and $(G,s)$ is a relative D-variety. If $(G,s)^\#$ is a subgroup of $G$ then $(G,s)$ is a relative D-group. Indeed, since $\nabla_\DD^G$ is a group homomorphism, the restriction of $s$ to $(G,s)^\#$ is a group homomorphism, and so, since $(G,s)^\#$ is $\D$-dense in $G$ and $s$ is a regular $\D$-map, $s$ is a group homomorphism on all of $G$.
\end{enumerate}
\end{remark}

The correspondence of Proposition \ref{biv} specializes to a natural correspondence between relative D-subgroups of a relative D-group and $\Pi$-subgroups of its set of sharp points:

\begin{lemma}\label{corre}
Suppose $(G,s)$ is a relative D-group defined over $K$ such that $(G,s)^\#(\bar K)=(G,s)^\#(K)$, for some $\Pi$-closure $\bar K$ of $K$. Then the $\sharp$-functor establishes a bijective correspondence between relative D-subgroups of $(G,s)$ defined over $K$ and $\Pi$-algebraic subgroups of $(G,s)^\#$ defined over $K$. The inverse is given by taking the $\D$-closure (in the $\D$-topology of $G$) over $K$.
\end{lemma}
\begin{proof}
This is an immediate consequence of Proposition~\ref{biv} and the discussion above, except for the fact that the $\D$-closure over $K$ of a $\Pi$-algebraic subgroup of $(G,s)^\#$ is a subgroup of $G$. Let us prove this. Let $H$ be a $\Pi$-algebraic subgroup of $(G,s)^\#$ and $\bar H$ its closure over $K$. Let $$X:=\{a\in \bar H:\, b\cdot a\in \bar H \text{ and }b\cdot a^{-1}\in \bar H \text{ for all } b\in \bar H\}\subseteq \bar H ,$$ then $X$ is a $\D$-algebraic subgroup of $G$ defined over $K$. We claim that $\bar H=X$. By the definition of $\bar H$, it suffices to show that $H(\bar K)\subseteq X(\bar K)$. Let $a\in H(\bar K)=H(K)$, and consider $Y_a=\{x\in \bar H:\, x\cdot a\in \bar H \text{ and }x\cdot a^{-1}\in \bar H\}$. Then $H\subset Y_a$, but $Y_a$ is $\D$-closed and defined over $K$, so $\bar H\subset Y_a$. Thus $a\in X$, as desired.
\end{proof}

We now show that every definable group $G$ in $(\U,\Pi)$ of differential type less than $m$ is, after possibly replacing $\Pi$ by some independent linear combination, definably isomorphic to a relative D-group w.r.t. $\DD/\D$ with $|\D|=\Pi$-type$(G)$. By Remark~\ref{witg} we can find a set $\D$ of linearly independent elements of $\operatorname{span}_{K^\Pi}\Pi$ which witnesses the $\Pi$-type of $G$, and hence it suffices to prove the following.

\begin{theorem}\label{co} 
Let $G$ be a connected $K$-definable group. Suppose $\D$ is a set of fewer than $m$ linearly independent elements of $\operatorname{span}_{K^\Pi}\Pi$ that bounds the $\Pi$-type of $G$. If we extend $\D$ to a basis $\DD\cup\D$ of $\operatorname{span}_{K^\Pi}\Pi$, then $G$ is $K$-definably isomorphic to $(H,s)^{\sharp}$ for some relative D-group $(H,s)$ w.r.t. $\DD/\D$ defined over $K$.
\end{theorem}
\begin{proof}
Recall that any $K$-definable group is $K$-definably isomorphic to a $\Pi$-algebraic group defined over $K$ (to the author's knowledge the proof of this fact, for the partial case, does not appear anywhere; however, it is well known that the proof for the ordinary case \cite{Pi4} extends with little modification). Thus, we assume that $G$ is a $\Pi$-algebraic group defined over $K$. By Corollary \ref{ong}, there is a $\Pi$-rational map $\al$ and a relative D-variety $(V,s)$ w.r.t. $\DD/\D$, both defined over $K$, such that $\al$ yields a $\Pi$-birational equivalence between $G$ and $(V,s)^\#$. Let $p$ be the $\Pi$-generic type of $G$ over $K$ and $q$ be the $\D$-generic type of $V$ over $K$. 

We first show that there is a generically defined $\D$-group structure on $q$. Note that $\al$ maps realisations of $p$ to elements of $(V,s)^\#$ realising $q$. Let $g$ and $h$ be $\Pi$-independent realisations of $p$ then $$K\l g,h\r_\Pi=K\l \al(g),\al(h)\r_\Pi=K\l\al(g),\al(h)\r_\D,$$ and so $\al(g\cdot h)\in K\l \al(g),\al(h)\r_\D$. Thus, we can find a $\D$-rational map $\rho$ defined over $K$ such that $\al(g\cdot h)=\rho(\al(g),\al(h))$. Hence, for any $g,h$ $\Pi$-independent realisations of $p$ $$\al(g\cdot h)=\rho(\al(g),\al(h)),$$ and so $\rho(\al(g),\al(h))$ realises $q$. Now, since $\al(g)$ and $\al(h)$ are $\D$-independent realisations of $q$, we get that for any $x,y$ $\D$-independent realisations of $q$, $\rho(x,y)$ realises $q$. Moreover, if $g,h,l$ are $\Pi$-independent realisations of $p$ then one easily checks that $$\rho(\al(g),\rho(\al(h),\al(l)))=\rho(\rho(\al(g),\al(h)),\al(l)),$$ but $\al(g)$, $\al(h)$ and $\al(l)$ are $\D$-independent realisations of $q$, so for any $x,y,z$ $\D$-independent realisations of $q$ we get $\rho(x,\rho(y,z))=\rho(\rho(x,y),z)$. 

We thus have a stationary type $q$ (in the language of $\D$-rings) and a $\D$-rational map $\rho$ satisfying the conditions of Hrushovski's theorem on groups given generically (see \cite{Hru} or \cite{Po}), and so there is a connected $K$-definable group $H$ and a $K$-definable injection $\beta$ (both in the language of $\D$-rings) such that $\beta$ maps the realisations of $q$ onto the realisations of the generic type of $H$ over $K$ and $\beta(\rho(x,y))=\beta(x)\cdot\beta(y)$ for all $x,y$ $\D$-independent realisations of $q$. We assume, without loss of generality, that $H$ is a $\D$-algebraic group defined over $K$. We have a (partial) definable map $\gamma:=\beta\comp \al:G\to H$ such that if $g,h$ are $\Pi$-independent realisations of $p$ then $\gamma(g\cdot h)=\gamma(g)\cdot \gamma(h)$. It follows that there is a $K$-definable group embedding $\bar \gamma:G\to H$ extending $\gamma$. Indeed, let $$U:=\{x\in G:\gamma(x y)=\gamma(x)\gamma(y)\text{ and } \gamma(yx)=\gamma(y)\gamma(x) \text{ for all }y\models p \text{ with } x\ind_K y\},$$ then $U$ is a $K$-definable subset of $G$ (by definability of types in $DCF_{0,m}$). If $g\models p$ then $g\in U$, and so every element of $G$ is a product of elements of $U$ (see \S 7.2 of \cite{Ma2}). Let $g\in G$ and let $u,v\in U$ be such that $g=u\cdot v$, then $\bar \gamma$ is defined by $$\bar \gamma(g)=\gamma(u)\cdot \gamma(v).$$ 
It is well known that this construction yields a group embedding (see for example \S 3 of Marker's survey \cite{Ma}). We also have a (partial) definable map $t:=\ta \beta\comp s\comp \beta^{-1}:H\to \ta H$ such that for every $g\models p$, as $\al(g)\in (V,s)^\#$, we have $t(\gamma (g))=\nabla_\DD^H(\gamma(g))$. Thus, for $g,h$ $\Pi$-independent realisations of $p$ $$t(\gamma(g)\cdot \gamma(h))=\nabla_\DD^H(\gamma(g)\cdot \gamma(h))=(\nabla_\DD^H\gamma(g))\cdot(\nabla_\DD^H\gamma(h))=t(\gamma(g))\cdot t(\gamma(h)).$$ Thus, for $x,y$ $\D$-independent realisations of the generic type of $H$ over $K$ we have $t(x\cdot y)=t(x)\cdot t(y)$, and so we can extend $t$ to a $K$-definable group homomorphism $\bar t:H\to \ta H$. Therefore, $(H,\bar t)$  is a relative D-group w.r.t. $\DD/\D$. Clearly, if $g$ is a $\Pi$-generic of $G$ over $K$, then $\gamma(g)$ is a $\Pi$-generic (over $K$) of both $\bar{\gamma}(G)$ and $(H,\bar t)^\#$. Hence, $\bar{\gamma}(G)=(H,\bar t)^\#$, as desired.
\end{proof}

\chapter{Generalized Galois theory for partial differential fields}\label{chapgal}\let\thefootnote\relax\footnotetext{The results in this chapter form the basis for a paper entitled ``Relative D-groups and differential Galois theory in several derivations'' that is to appear in \emph{Transactions of the AMS.}}

In this chapter we develop the fundamental results of generalized strongly normal extensions for partial differential fields in the possibly infinite dimensional setting. The connection to logarithmic differential equations on relative D-groups is also established. This theory generalizes simultanoeusly the parametrized Picard-Vessiot theory of Cassidy and Singer \cite{PM} and the finite dimensional differential Galois theory of Pillay \cite{Pi2}. We finish with two examples of non-linear Galois groups associated to logarithmic differential equations.

\

\section{A review of Galois theory for ordinary differential fields}\label{revgal}

In this section we give a brief review of the several differential Galois theories for ordinary differential fields. We also mention some of the extensions of these theories for partial differential fields. We work in a sufficiently saturated $(\U,\d)\models DCF_0$ and a base $\d$-subfield $K$ of $\U$. Recall that we denote by $\U^\d$ and $K^\d$ the field of constants of $\U$ and $K$, respectively.

Recall that, in usual Galois theory, the Galois extension associated to a polynomial is precisely the splitting field of the polynomial in question, and the Galois group is the group of automorphisms of the extension that leave the base field unchanged. It is then shown that the structure of the Galois group reflects the algebraic properties of the roots of the polynomial via the correspondence between its subgroups and the intermediate fields of the extension and the base field. One can proceed in an analogous fashion when dealing with homogeneous linear ordinary differential equations over $K$; that is, equations of the form $$a_n\d^n x+a_{n-1}\d^{n-1}x+\cdots +a_1\d x+a_0 x=0, \quad a_i \in K.$$ One can easily rewrite each such equation in vector form 
\begin{displaymath}
\d y=M y
\end{displaymath}
where $y$ is an $n\times 1$ column of indeterminates and $M$ is an $n\times n$ matrix over $K$. 

\begin{definition}
A \emph{Picard-Vessiot extension of} $K$ for the homogeneuos linear differential equation $\d y=M y$ is a $\d$-field extension $L$ of $K$ such that 
\begin{enumerate}
\item $L^\d=K^\d$ is algebraically closed.
\item $L$ is generated by the entries of a matrix $A$ in GL$_{n}(L)$ such that $\d A=M A$.
\end{enumerate}
The matrix $A$ is called a \emph{fundamental system of solutions} of the system $\d y=M y$. We let $Aut_\d(L/K)$ denote the group of $\d$-automorphisms of $L$ over $K$.
\end{definition}

Let us mention that most authors do not require, in condition (1) of the definition of Picard-Vessiot extension, for $K^\d$ to be algebraically closed, only that no new constants are added as we pass from $K$ to $L$. We restrict ourselves to this case for the sake of a smoother exposition of the fundamental results of the classical theory.

Let us summarize the basic properties of Picard-Vessiot extensions:
\begin{enumerate}
\item [(i)] Picard-Vessiot extensions always exist and are unique, up to $\d$-isomorphism over $K$.
\item [(ii)] Let $L$ be a Picard-Vessiot extension. Then, $Aut_\d(L\l \U^\d\r_\d/K\l \U^\d\r_\d)$ is isomorphic to the $\U^\d$-points of a linear algebraic group $\G$ defined over $K^\d$.
\item [(iii)] There is a natural Galois correspondence between the algebraic subgroups of $\G(\U^\d)$ defined over $K^\d$ and the intermediate $\d$-fields of $K$ and $L$.
\end{enumerate}

For proofs and more properties of Picard-Vessiot extensions see \cite{VS} or \cite{Ko4}. 

One more fundamental property of a Picard-Vessiot extension $L$ is that for any $\d$-isomorphism $\s$ from $L$ into $\U$ over $K$, it is always the case that $\s(L)\subset  L\l \U^\d\r_\d$. Kolchin noticed that this implicit property (i.e., it does not refer to a differential equation) of Picard-Vessiot extensions is what yields a good Galois theory. Based on this observation, he was able to generalize the theory to a much larger class of extensions, including some that deal with extensions that arise from certain non-linear differential equations.

\begin{definition}
A finitely generated $\d$-field extension $L$ of $K$ is said to be \emph{strongly normal} if:
\begin{enumerate}
\item $L^\d=K^\d$ is algebraically closed
\item For any $\d$-isomorphism $\s$ from $L$ into $\U$ over $K$ we have $\s(L)\subset L\l \U^\d\r_\d$
\end{enumerate}
\end{definition}

In \cite{Ko}, Kolchin shows, amongst other things, that the group of automorphisms $$Aut_\d(L\l \U^\d\r_\d/K\l \U^\d\r_\d)$$ is isomorphic to the $\U^\d$-points of an algebraic (not necessarily linear) group defined over $K^\d$, and that there is a Galois correspondence analogous to (iii) above. It is worth mentioning that, in \cite{Po2}, Poizat observed that condition (2) in the definition of strongly normal can be rephrased in terms of the model theoretic notion of internality, this lead him to direct proofs of Kolchin's results using the machinery of $\w$-stable theories. Indeed, this is how model theorists became interested in differential Galois theory.

In the spirit of Poizat's observation, Pillay took differential Galois theory one step further. He succesfully replaced the role of the constants $\U^\d$ in the definition of strongly normal, by an arbitrary $K$-definable set $X$. He defined \emph{$X$-strongly  normal extensions} as those finitely generated $\d$-field extensions $L$ of $K$ such that $X(K)=X(\bar L)$, for some $\d$-closure $\bar L$ of $L$, and for any $\d$-isomorphism $\s$ from $L$ into $\U$ over $K$ we have $\s(L)\subset L\l X \r_\d$. Then, he proved, using model theoretic techniques, that the automorphism group $$Aut_\d(L\l X\r_\d/K\l X\r_\d)$$ is isomorphic to a finite dimensional $\d$-algebraic group defined over $K$ (not necessarily in the constants), and that there is a Galois correspondence \cite{Pi}. 

In \cite{Pi2}, Pillay takes a different approach and reformulates his Galois theory in terms of logarithmic differential equations. His idea is to replace the role of GL$_n$ in the definition of Picard-Vessiot extensions by an arbitrary algebraic D-group $(G,s)$. In this situation one gets a surjective crossed homomorphism $\ell_s:G\to TG_e$, where $TG_e$ is the Lie algebra of $G$ (i.e., the tangent space at the identity $e$). The analogue of a linear equation is now an equation of the form $\ell_s x=\al$, where $\al$ is a $K$-point in the Lie algebra of $G$. Under some additional technical conditions, Pillay defines the notion of a differential Galois extension for the equation $\ell_s x=\al$, proves its existence and uniqueness, and observes that they are $(G,s)^\#$-strongly normal extensions. In fact, he shows, under the additional assumption that $K$ is algebraically closed, that every generalized strongly normal extension is of this form.

We will extend Pillay's theory to the partial differential, infinite dimensional, setting. Let us therefore mention some of the extensions to the partial setting that have already been considered in the literature. Suppose now that we are in a sufficiently saturated $(\U, \Pi)\models DCF_{0,m}$ and we have a partition $\DD\cup\D$ of $\Pi$ with $\DD=\{D_1,\dots,D_r\}$. In \cite{PM}, Cassidy and Singer extend the Picard-Vessiot theory to a parametrized Picard-Vessiot theory. For equations of the form $$D_1 y=M_1 y, \dots, D_r x=M_r x,$$ where $M_i$ are $n\times n$ matrices over a $\Pi$-subfield $K$ satisfying the \emph{integrability condition} $$D_iM_j-D_jM_i=[M_i,M_j],$$ they define PPV-extensions as a $\Pi$-field extension $L$ of $K$ such that $L^\DD=K^\DD$ is $\D$-closed and $L$ is generated as a $\Pi$-field by the entries of a matrix $A$ in GL$_n(L)$ such that $D_i A=M_i A$ for $i=1,\dots,r$. They proved that the results from the Picard-Vessiot theory have natural extensions to this parametrized setting.

In the same spirit of PPV-extensions, Landesman parametrizes Kolchin's strongly normal extensions. He defines a $\D$-strongly normal extension $L$ as a finitely generated $\Pi$-field extension of $K$ such that $L^{\DD}=K^\DD$ is $\D$-closed and for any $\Pi$-isomorphism $\s$ from $L$ into $\U$ over $K$ we have $\s(L)\subset L\l \U^\DD\r_\Pi$. Then, he shows that the results for strongly normal extensions have natural extensions to this parametrized setting \cite{La}.

The goal of this chapter is to extend the differential Galois theories of Pillay (generalized strongly normal and logarithmic differential) to the partial setting. In fact both of his theories have more or less immediate extensions to the partial case as long as the definable sets in question are finite dimensional. Thus, our work is to extend Pillay's theories from finite dimensional to possibly infinite dimensional definable sets. As we will see, the PPV extensions of Cassidy and Singer and the $\D$-strongly normal extensions of Landesman are special cases of what we present here.

\

\section{Generalized strongly normal extensions}\label{genstrong}

In this section we extend Pillay's strongly normal extensions \cite{Pi} to the possibly infinite dimensional partial differential setting. In fact most of the arguments here are extensions of Pillay's arguments from \cite{Pi}, but for the sake of review, as well as completeness, we will give the details.

Fix a sufficiently large saturated $(\U,\Pi)\models DCF_{0,m}$ and a base $\Pi$-subfield $K$ of $\U$. All definable sets will be identified with their $\U$-points. Unless otherwise specified, the notation and terminology of this section will be with respect to the derivations $\Pi$; for example, $K\l x\r$ means $K\l x\r_{\Pi}$, generated means $\Pi$-generated, isomorphism means $\Pi$-isomorphism, etc.

\begin{definition}\label{def}
Let $X$ be a $K$-definable set. A finitely generated $\Pi$-field extension $L$ of $K$ is said to be $X$-\emph{strongly normal} if:
\begin{enumerate}
\item $X(K)=X(\bar L)$ for some (equivalently any) differential closure $\bar L$ of $L$.
\item For any isomorphism $\sigma$ from $L$ into $\U$ over $K$, $\sigma(L)\subseteq L\l X\r$.
\end{enumerate}
\end{definition}

A $\Pi$-field extension $L$ of $K$ is called a \emph{generalized strongly normal extension} of $K$ if it is an $X$-strongly normal extension for some $K$-definable $X$. 

\begin{remark} \label{rem} \
\begin{enumerate} 
\item [(i)] Suppose $\DD\cup\D$ is a partition of $\Pi$ with $\DD\neq \emptyset$. If $X=\U^{\DD}$, then the $X$-strongly normal extensions are exactly the ``$\D$-strongly normal extensions" of Landesman \cite{La}. Indeed, one need only observe that in this case~(1) is equivalent to $K^\DD=L^{\DD}$ is $\D$-closed (see for example Lemma 9.3 of \cite{PM}). 
\item [(ii)] If $L=K\l a\r$ is an $X$-strongly normal extension and $b$ realizes $tp(a/K)$, then $L'=K\l b\r$ is also an $X$-strongly normal extension. Moreover, if $b\in \bar L$ then $L=L'$. Indeed, if $\sigma$ is an automorphism of $\U$ over $K$ such that $\sigma(a)=b\in \bar L$, then, by (2), $L'\subseteq L\l X\r$. But as $L'\subseteq \bar L\models DCF_{0,m}$, we have that $L'\subseteq L\l X(\bar L)\r=L\l X(K)\r=L$. A similar argument shows $L\subseteq L'$.
\item [(iii)] When $X$ is finite, $X$-strongly normal extensions are just the usual Galois (i.e., finite and normal) extensions of $K$. In the other extreme, if $\Pi$-type$(X)=m$ then $X$-strongly normal extensions are trivial. Indeed, suppose that $\Pi$-type$(X)=m$ and $X(K)=X(\bar L)$. Then, since $X$ having $\Pi$-type equal to $m$ implies that $X$ projects $\Pi$-dominantly onto some coordinate, we get a $\Pi$-dense $K$-definable subset of $\U$ whose $\bar L$-points are contained in $K$. This implies that $K=\bar L$, and hence $X$-strongly normal extensions are trivial. Therefore, this notion is mainly of interest when $X$ is infinite and of $\Pi$-type strictly less than $m$.
\end{enumerate}
\end{remark}

For the rest of this section we fix a $K$-definable set $X$.

\begin{proposition}\label{inter}
For any $L$ finitely generated $\Pi$-field extension of $K$, condition~(2) of Definition~\ref{def}  is equivalent to $L$ being generated by a fundamental system of solutions of an $X$-internal type over $K$. 
\end{proposition}
\begin{proof}
If a type $p\in S_n(K)$ is $X$-internal, then there is a sequence $(a_1,\dots,a_\ell)$ of realisations of $p$, called a fundamental system of solutions, such that for every $b\models p$ there is a tuple $c$ from $X$ such that $b\in\operatorname{dcl}(K,a_1,\dots,a_\ell,c)$ (see \S7.4 of \cite{Pi8}). 

Suppose $L=K\l a\r$ satisfies condition~(2). We claim that $tp(a/K)$ is $X$-internal and $a$ is itself a fundamental system of solutions. If $b\models tp(a/K)$ there is an automorphism $\sigma$ of $\U$ over $K$ such that $b=\sigma(a)\in L\l X\r$. Hence, $b\in \operatorname{dcl}(K,a,c)$ for some tuple $c$ from $X$. Conversely, suppose $L=K\l a\r$ and $a$ is a fundamental system of solutions of an $X$-internal type over $K$. Note that $tp(a/K)$ is also $X$-internal with fundamental system of solutions $a$ itself. So if $\sigma$ is an isomorphism from $L$ into $\U$ over $K$, then $\sigma(a)\models tp(a/K)$ and hence there is a tuple $c$ from $X$ such that $\sigma(a)\in \operatorname{dcl}(K,a,c)=L\l c\r$. That is $\sigma(L)\subseteq L\l X\r$, as desired.
\end{proof}

If $L=K\l a\r$ is an $X$-strongly normal extension then, by Lemma \ref{onint} and Proposition \ref{inter},
\begin{displaymath}
\Pi\text{-type}(a/K)\leq \Pi\text{-type}(X).
\end{displaymath}
Moreover, if $\D\subseteq \operatorname{span}_{K^\Pi}\Pi$ bounds the $\Pi$-type of $X$ then $\D$ also bounds the $\Pi$-type of $a$ over $K$ (see \ref{defboset} for the definition of ``bounds the $\Pi$-type''). Specializing to the case when $\DD\cup\D$ is a partition of $\Pi$ and $X=\U^{\DD}$, we recover Theorem 1.24 of \cite{La} that every $\U^\DD$-strongly normal extension of $K$ is finitely $\D$-generated over $K$.

We now start presenting the main properties of $X$-strongly normal extensions.

\begin{lemma}\label{cont}
Suppose $L$ is an $X$-strongly normal extension of $K$. Then $L$ is contained in some differential closure of $K$.
\end{lemma}
\begin{proof}
Let $L=K\l a\r$ and $p=tp(a/K)$. Then for any realisation $b$ of $p$, we have $b\in K\l a,X\r$. By compactness, we get a $K$-definable function $f$ and $\phi\in p$ such that for any $b_1,b_2$ realisations of $\phi$ there is a tuple $c$ of $X$ such that $b_1=f(b_2,c)$. Let $\bar L$ be a differential closure of $L$ and $\bar K$ a differential closure $K$ contained in $\bar L$. Let $b\in \bar K$ satisfy $\phi$. Then there is a tuple $c$ from $X(\bar L)=X(K)$ such that $a=f(b,c)$. Thus, $a\in \bar K$ and $L<\bar K$.
\end{proof}

The previous lemma implies that condition~(1) of Definition~\ref{def} can be replaced by 
\begin{enumerate}
\item [(1')] \quad   $X(K)=X(\bar K) \text{ and } L\subseteq \bar K \text{ for some differential closure } \bar K \text{ of } K.$
\end{enumerate}

The following result is fundamental for the analysis of the group of automorphisms of generalized strongly normal extensions. 

\begin{lemma}\label{exiso}
Let $L$ be an $X$-strongly normal extension of $K$. Let $\sigma$ an isomorphism from $L$ into $\U$ over $K$. Then $\sigma$ extends to a unique automorphism of $L\l X\r$ over $K\l X\r$.
\end{lemma}
\begin{proof}
Let $L=K\l a\r$ and $\phi$ a formula isolating $tp(a/K)$ (the existence of such a formula is given by Lemma~\ref{cont}). We show that $\phi$ also isolates $tp(a/K\l X\r)$. If not, there is $\psi(x,y)$ (over $K$) and $c$ a tuple from $X$ such that $\psi(x,c)\in tp(a/K\l X\r)$ and also a $b$ realising $\phi(x)$ but not $\psi(x,c)$. Then we can find such $b$ and $c$ in $\bar L$, and so $c\in X(\bar L)=X(K)$ contradicting the fact that $\phi$ isolates $tp(a/K)$. This implies that $tp(a/K\l X\r)=tp(\sigma(a)/K\l X\r)$, and so there is a unique isomorphism from $L\l X\r$ onto $\s(L)\l X\r$ over $K\l X\r$ extending $\s$. However, it follows from condition~(2) of Definition~\ref{def} that $L\l X\r=\s(L)\l X\r$.
\end{proof}

Now, let $L=K\l a\r$ be an $X$-strongly normal extension of $K$. Following Pillay (and Kolchin), we define Gal$_X(L/K):=Aut_{\Pi}(L\l X\r/K\l X\r)$; that is, Gal$_X(L/K)$ is the group of automorphisms of $L\l X\r$ over $K\l X\r$. Note that $tp(a/K)^\U$ is in $\operatorname{dcl}(L\cup X)$, and thus Gal$_X(L/K)$ acts naturally on $tp(a/K)^\U$. It follows, from Lemma~\ref{exiso}, that this action is regular. We also let gal$(L/K):=Aut_{\Pi}(L/K)$. By Lemma~\ref{exiso}, every element of gal$(L/K)$ extends uniquely to an element of Gal$_X(L/K)$. Via this map, we identify gal$(L/K)$ with a subgroup of Gal$_X(L/K)$.

\begin{lemma}
Gal$_X(L/K)$ and its action on $tp(a/K)^\U$ do not depend (up to isomorphism) on the choice of $X$ nor on the choice of generator $a$ of $L$ over $K$. We therefore simply write Gal$(L/K)$ for Gal$_X(L/K)$.
\end{lemma}
\begin{proof}
If $L$ is also an $X'$-strongly normal extension of $K$ then, by Lemma~\ref{exiso}, there is a (unique) group isomorphism $$\psi:\, \text{Gal}_X(L/K)\to \,\text{Gal}_{X'}(L/K)$$ such that $\sigma|_{L\l X\r\cap L\l X'\r}=\psi(\sigma)|_{L\l X\r\cap L\l X'\r}$. Indeed, if $\sigma\in$ Gal$_X(L/K)$, then the restriction of $\sigma$ to $L$ gives rise to an isomorphism from $L$ into $\U$ over $K$ and so $\psi(\sigma)$ is defined as the unique extension of $\sigma|_L$ to an automorphism of $L\l X'\r$ over $K\l X'\r$. It is clear that this yields the desired group isomorphism.  Furthermore, if $a'$ is another generator of $L$ over $K$, then $$tp(a/K)^\U\cup tp(a'/K)^\U\,\text{ is in } \operatorname{dcl}(L\cup X\cup X')$$
and so any $K$-definable bijection between $tp(a/K)^\U$ and $tp(a'/K)^\U$, together with $\psi$, gives rise to an isomorphism between the natural actions of Gal$_X(L/K)$ on $tp(a/K)^\U$ and Gal$_{X'}(L/K)$ on $tp(a'/K)^\U$. 
\end{proof}

The following theorem asserts the existence of the \emph{definable} Galois group of a generalized strongly normal extension and summarises some of its basic properties.

\begin{theorem}\label{group}
Suppose $L=K\l a\r$ is an $X$-strongly normal extension of $K$.
\begin{enumerate}
\item There is a $K$-definable group $\G$ in $\operatorname{dcl}(K \cup X)$ with an $L$-definable regular action on $tp(a/K)^\U$ such that $\G$ together with its action on $tp(a/K)^\U$ is (abstractly) isomorphic to Gal$(L/K)$ together with its natural action on $tp(a/K)^\U$. We call $\G$ the \emph{Galois group of $L$ over $K$.}
\item $\Pi$-type$(\G)=\Pi$-type$(a/K)$ and $\Pi$-dim$(\G)=\Pi$-dim$(a/K)$. Moreover, $\D\subseteq \operatorname{span}_{K^\Pi}\Pi$ bounds the $\Pi$-type of $\G$ if and only if $\D$ bounds the $\Pi$-type of $a$ over $K$. 
\item $L$ is a $\G$-strongly normal extension.
\item The action of $\G$ on $tp(a/K)^\U$ is $K$-definable if and only if $\G$ is abelian.
\item If $\mu:$Gal$(L/K)\to \G$ is the isomorphism from (1), then $\mu($gal$(L/K))=\G(\bar K)=\G(K)$.
\end{enumerate}	
\end{theorem}
\begin{proof} \

\noindent (1) The construction is exactly as in \cite{Pi}, but we recall it briefly. Let $Z=tp(a/K)^{\U}$, by Lemma~\ref{cont}, $tp(a/K)$ is isolated and so $Z$ is $K$-definable. If $b\in Z$ then $b$ is a tuple from $L\l X\r$,  so by compactness we can find a $K$-definable function $f_0(x,y)$ such that for any $b\in Z$ there is a tuple $c$ from $X$ with $b=f_0(a,c)$. Let $Y_0=\{c \in X: f_0(a,c)\in Z\}$, then $Y_0$ is a $K$-definable set of tuples from $X$. Consider the equivalence relation on $Y_0$ given by $E(c_1,c_2)$ if and only if $f_0(a,c_1)=f_0(a,c_2)$, then $E$ is $K$-definable. By elimination of imaginaries we can find a $K$-definable set $Y$ in $\operatorname{dcl}(K\cup X)$ which we can identify with $Y_0/E$. Now define $f:Z\times Y\to Z$ by $f(b,d)=f_0(b,c)$ where $c$ is such that $d=c/E$. Now note that for any $b_1,b_2\in Z$ there is a unique $d\in Y$ such that $b_2=f(b_1,d)$, and so we can write $d=h(b_1,b_2)$ for some $K$-definable function $h$.

Define $\mu:$ Gal$(L/K)\to Y$ by $\mu(\sigma)=h(a,\sigma a)$. Then $\mu$ is a bijection. Let $\G$ denote the group with underlying set $Y$ and with the group structure induced by $\mu$. Then $\G$ is a $K$-definable group and is in $\operatorname{dcl}(K\cup X)$. Consider the action of $\G$ on $Z$ induced (via $\mu$) from the action of Gal$(L/K)$ on $Z$, i.e., for each $g\in \G$ and $b\in Z$ let $g.b:=\mu^{-1}(g)(b)$. This action is indeed $L$-definable since
\begin{equation}\label{action1}
\mu^{-1}(g)(b)=f(\mu^{-1}(g)(a),h(a,b))=f(f(a,g),h(a,b)).
\end{equation}

\noindent (2) Since $\G$ acts regularly and definably on $tp(a/K)^\U$, the map $g\mapsto g.a$ is a definable bijection between these two $K$-definable sets. Part (2) follows since $\Pi$-type, $\Pi$-dim and the property ``$\D$ bounds the $\Pi$-type" for definable sets are all preserved under extension of the set of parameters, as well as under definable bijection.

\noindent (3) Since $\G$ is in $\operatorname{dcl}(K\cup X)$ and $X(\bar K)=X(K)$, we have that $\G(\bar K)=\G(K)$. If $\sigma$ is an isomorphism from $L=K\l a\r$ into $\U$ over $K$, then, by Lemma~\ref{exiso}, we can extend it uniquely to an element of Gal$_X(L/K)$. Then, by (1), we get that $\sigma(a)=\mu (\sigma).a$ is in $\operatorname{dcl}( L\cup\G)$ and so $\sigma(L)\subseteq L\l \G\r$.

\noindent (4) If the action is $K$-definable, for any $g_1,g_2\in \G$ we get
\begin{displaymath}
(g_1g_2).a=g_1.(\mu^{-1}(g_2)(a))=\mu^{-1}(g_2)(g_1.a)=(g_2g_1).a
\end{displaymath}
so $g_1g_2=g_2g_1$. Conversely, suppose $\G$ is abelian and let $\sigma\in$ Gal$(L/K)$ be such that $\sigma(a)=b$, then
\begin{displaymath}
\mu^{-1}(g)(b)=\mu^{-1}(g)(\sigma(a))=\sigma(\mu^{-1}(g)(a))=\sigma(f(a,g))=f(b,g).
\end{displaymath}

\noindent (5) Note that if $g\in \G(\bar K)$ then $g=p(c)$ where $p\in K\l x\r$ and $c$ is a tuple from $X(\bar K)=X(K)$, and so $g \in \G(K)$. Thus $\G(\bar K)=\G(K)=\G(L)$. Now if $\sigma\in$ gal$(L/K)$ then $\sigma(a)$ is a tuple from $L$, and $\mu(\sigma)=h(a,\sigma(a))\in \G(L)$. Conversely, let $g\in \G(L)$ with $g=\mu(\sigma)$ and $\sigma\in$ Gal$(L/K)$. Then $\sigma(a)=f(a,g)$ is in $L$ and thus the restriction of $\sigma$ to $L$ determines an automorphism. Hence, $\sigma\in$ gal$(L/K)$.
\end{proof}

\begin{remark}\label{above}
Using the notation of the proof of Theorem \ref{group}~(1), one could consider the map $Z\times \G\to Z$ given by $(b,g)\mapsto f(b,g)=:g*b$ and ask if this is a well-defined action. In general it is not, since it is not necessarily the case that $g_1*(g_2*b)=(g_1g_2)*b$; however, it is the case that $g_1*(g_2*b)=(g_2g_1)*b$. Hence, $f$ induces a regular $K$-definable action of $\G_{\text{op}}$ on $Z$, where $\G_{\text{op}}$ is the group with universe $\G$ and product $g\cdot h:=h\,g$. Moreover, $\G$ will be abelian if and only if $g*b$ defines an action of $\G$ on $Z$, and in this case $g*b$ will coincide with the action of Theorem \ref{group}~(1).
\end{remark}

\begin{corollary}
A differential field extension $L$ of $K$ is a generalized strongly normal extension if and only if there are a $K$-definable group $\H$ satisfying $\H(\bar L)=\H(K)$ and a $K$-definable principal homogeneous $\H$-space $Z$ such that $L$ is generated by an element of $Z$.
\end{corollary}
\begin{proof}
Suppose $L=K\l a\r$ is a generalized strongly normal extension. Let $Z=tp(a/K)^{\U}$ and $\H=\G_{\text{op}}$ where $\G$ is the $K$-definable group from Theorem~\ref{group}. In Remark \ref{above} we saw that the $f$ from the proof of Theorem \ref{group}~(1) makes $Z$ into a $K$-definable principal homogeneous $\H$-space. The converse is clear, since $L$ will be an $\H$-strongly normal extension of $K$.
\end{proof}

Now we explain why the construction of the definable group $\G$ in the proof of Theorem~\ref{group} is independent of $X,a$ as well as the choice of $Y,f$.

\begin{lemma}\label{uni}
Let $L=K\l a\r$ be an $X$-strongly normal extension and $\mu:$Gal$(L/K)\to\G$ as in Theorem \ref{group}. If $\G'$ is a $K$-definable group satisfying the following:
\begin{enumerate}
\item $\G'(\bar K)=\G'(K)$
\item For some $a'$ such that $L=K\l a'\r$, there is an $L$-definable action of $\G'$ on $Z':=tp(a'/K)^{\U}$.
\item There is a group isomorphism $\mu':$Gal$(L/K)\to \G'$ preserving the action on $Z'$.
\end{enumerate}
Then $\mu'\comp\mu^{-1}:\G\to \G'$ is the unique $K$-definable group isomorphism such that: any $K$-definable bijection between $Z:=tp(a/K)^{\U}$ and $Z'$, together with $\mu'\comp\mu^{-1}$, gives an isomorphism between the action of $\G$ on $Z$ and the action of $\G'$ on $Z'$.
\end{lemma}
\begin{proof}
There is a $K$-definable function $q(x,y,z)$ such that $q(a',y,z):Z'\times \G'\to Z'$ defines the action of $\G'$ on $Z'$. Let $f'(y,z):=q(y,y,z)$. For any $b_1,b_2\in Z'$ there is a unique $g'\in \G'$ such that $f'(b_1,g')=b_2$ and so we can write $g'=h'(b_1,b_2)$ for some $K$-definable function $h'$. We have that for any $\sigma\in$ Gal$(L/K)$
\begin{displaymath}
\sigma(a')=q(a',a',\mu'(\sigma))=f'(a',\mu'(\sigma)).
\end{displaymath}
Thus, $\mu'(\sigma)=h'(a',\sigma a')$. Let $\rho:K\l a \r\to K\l a'\r$ be a $K$-definable bijection such that $\rho(a)=a'$. For every $g\in \G$ we get
\begin{displaymath}
\mu'\comp\mu^{-1}(g)=h'(a',\mu^{-1}(g)(a'))=h'(\rho(a), \rho(f(a,g))),
\end{displaymath}
where $f$ is the $K$-definable function from the proof of Theorem \ref{group}. For every $b\in Z$, we can find $\sigma$ an automorphism of $\U$ over $K$ such that $\sigma(a)=b$ and $\sigma$ fixes both $\G$ and $\G'$ pointwise (since $L$ is a $\G$ as well as a $\G'$-strongly normal extension). Then
\begin{displaymath}
h'(\rho(b), \rho(f(b,g)))=\sigma(h'(\rho(a), \rho(f(a,g))))=h'(\rho(a), \rho(f(a,g))).
\end{displaymath}
Hence, $\mu'\comp\mu^{-1}$ is $K$-definable. The rest is clear.
\end{proof}

Even though the construction of $\G$ in Theorem \ref{group} depends on the choice of $(X,a,Y,f)$, if we choose different data $(X',a',Y',f')$, for the same extension $L/K$, and construct the corresponding $\G'$ and $\mu'$, Lemma \ref{uni} shows that $\G$ and $\G'$ will be $K$-definably isomorphic via $\mu'\comp \mu^{-1}$. Moreover, any $K$-definable bijection between $tp(a/K)^\U$ and $tp(a'/K)^\U$, together with $\mu'\comp\mu^{-1}$, gives rise to an isomorphism between the actions of $\G$ on $tp(a/K)^\U$ and $\G'$ on $tp(a'/K)^\U$. On the other hand, if $\G'$ satisfies the conditions of Lemma \ref{uni} then it comes from such a construction (since in this case $L$ will be a $\G'$-strongly normal extension and there is a natural choice of data which will give rise to $\G'$). Thus, to any generalized strongly normal extension we can associate a unique (up to $K$-definable isomorphism) $K$-definable group $\G$ equipped with an $L$-definable regular action on $tp(a/K)^\U$ for each generator $a$ of $L$ over $K$. We call $\G$ \emph{the Galois group} of $L$ over $K$.

\begin{example}\label{Gg} 
Suppose we have a partition $\DD\cup\D$ of $\Pi$, $X=\U^\DD$, $L$ is an $X$-strongly normal extension of $K$ and $\G$ is its Galois group. Then $\G$ is in $\operatorname{dcl}(K\cup \U^\DD)$. Hence, $\G$ is a $\D$-algebraic group over $K^\DD$ in the $\DD$-constants. Thus, (1) and (2) of Theorem \ref{group} recover Theorem 1.24 of \cite{La}.
\end{example}

We have the following Galois correspondence.
\begin{theorem}\label{normal}
Let $L=K\l a\r$ be an $X$-strongly normal extension of $K$ with Galois group $\G$, and let $\mu:$Gal$(L/K)\to \G$ be the isomorphism from Theorem \ref{group}. 
\begin{enumerate}
\item If $K\leq F \leq L$ is an intermediate $\Pi$-field, then $L$ is an $X$-strongly normal extension of $F$. Moreover,
\begin{displaymath}
\G_F:=\mu(\textrm{Gal}(L/F))
\end{displaymath}
is a $K$-definable subgroup of $\G$ and is the Galois group of $L$ over $F$. The map $F\mapsto\G_F$ establishes a 1-1 correspondence between the intermediate differential fields and $K$-definable subgroups of $\G$. 
\item $F$ is an $X$-strongly normal extension of $K$ if and only if $\G_F$ is a normal subgroup of $\G$, in which case $\G/\G_F$ is the Galois group of $F$ over $K$.
\end{enumerate}
\end{theorem}
\begin{proof} \

\noindent (1) Clearly $L$ is an $X$-strongly normal extension of $F$. Also, $F$ is finitely generated over $K$, in fact $F=K\l b\r$ where $b$ generates the minimal $\Pi$-field of definition of the set of realisations of $tp(a/F)$ (recall this type is isolated). Let $b=p(a)$ for $p\in K\l x\r$. Note that if $g\in \G$ then $g\in \G_F$ if and only if $\mu^{-1}(g)(c)=c$ for all $c\in F$ which is equivalent to $\mu^{-1}(g)(p(a))=p(a)$. Thus, $g\in \G_F$ if and only if $g\in \G$ and $p(f(a,g))=p(a)$, which is a $K$-definable condition since $tp(a/K)$ is isolated. Thus $\G_F$ is $K$-definable. It is clear, from the definition, that $\G_F$ is the Galois group of $L$ over $F$.

Now we prove the Galois correspondence. Let $F_1$ and $F_2$ be distinct intermediate differential fields, we show that if $b\in F_2\backslash F_1$ there is $\sigma\in$ Gal$(L/F_1)$ such that $\sigma(b)\neq b$. As $F_1=\operatorname{dcl}(F_1)$, there is an automorphism $\s'$ of $\U$ over $F_1$ such that $\s'(b)\neq b$. The restriction of $\s'$ to $L\l X\r$ yields the desired element of Gal$(L/F_1)$. Thus, $\G_{F_1}\neq \G_{F_2}$ and so the map $F\mapsto \G_F$ is injective.  

For surjectivity let $H$ be a $K$-definable subgroup of $\G$. Consider $Y=\{\mu^{-1}(h)(a)\in L\l X\r: h\in H\}$ and let $b$ be a tuple that generates the minimal $\Pi$-field of definition of $Y$. Then $b\in L$ and let $F=K\l b\r$. We show $H=\G_F$. Let $h\in H$ then clearly $\mu^{-1}(h)(Y)=Y$, so $\mu^{-1}(h)(b)=b$ and then $h\in \G_F$. Conversely, if $g\in \G_F$ then, since $a\in Y$, $\mu^{-1}(g)(a)\in Y$. Thus, $\mu^{-1}(g)=\mu^{-1}(h)$ for some $h\in H$ and so $g=h\in H$. This establishes the Galois correspondence.

\noindent (2) Suppose $F$ is an $X$-strongly normal extension of $K$, we need to show Gal$(L/F)$ is normal in Gal$(L/K)$. This is immediate since for all $\sigma\in$ Gal$(L/K)$ we get that $\sigma(F)\subset F\l X\r$. Now suppose $\G_F$ is normal in $\G$, in order to show that $F$ is an $X$-strongly normal extension of $K$ we only need to check that for every isomorphism $\phi$ from $F$ into $\U$ over $K$, $\phi(F)\subset F\l X\r$. For  contradiction suppose there is $b\in F$ such that $\phi(b)\notin F\l X\r$. We can extend $\phi$ to $\psi\in$ Gal$(L/K)$ and thus by part (1) we can find $\sigma\in$ Gal$(L/F)$ such that $\sigma\comp\psi(b)\neq \psi(b)$. Hence, $\psi^{-1}\comp\sigma\comp\psi(b)\neq b$ and this contradicts normality. Finally, the group isomorphism $\eta:\G/\G_F\to\text{Gal}(F/K)$ given by $\eta(g\,\G_F)= \mu^{-1}(g)|_{F\l X\r}$ shows that $\G/\G_F$ is the Galois group of $F$ over $K$.
\end{proof}

\begin{corollary}\label{conn}
Let $L$ be a generalized strongly normal extension of $K$ with Galois group $\G$. Then $\G$ is connected if and only if $K$ is relatively algebraically closed in $L$.
\end{corollary}
\begin{proof}
Suppose $b\in L$ is algebraic over $K$ but $b\notin K$. Let $F\leq L$ be the (algebraic) normal closure of $K(b)$, then $F$ is an $X$-strongly normal extension of $K$. Thus, Gal$(F/K)$ is a nontrivial finite group and from (2) of Theorem \ref{normal} we get that Gal$(L/F)$ is a proper normal definable subgroup of Gal$(L/K)$ of finite index. This implies $\G$ is not connected. Conversely, if $\G$ is not connected, we can find a proper normal definable subgroup of finite index. This subgroup will be of the form $\G_F$ for some $X$-strongly normal extension $F$ of $K$ contained in $L$. But Gal$(F/K)$ will be finite and nontrivial, so $F$ will be a proper algebraic extension of $K$ in $L$.
\end{proof}

Finally, as in the ordinary case, we have a positive answer to the ``baby" inverse Galois problem:

\begin{proposition}\label{bab}
Let $G$ be a connected definable group with $\Pi$-type$(G)<m$. Then there is a $\Pi$-field $K$, over which $G$ is defined, and a $G$-strongly normal extension $L$ of $K$ such that $G$ is the Galois group of $L$ over $K$.
\end{proposition}
\begin{proof}
Let $K_0$ be a $\Pi$-closed field over which $G$ is defined. By the embeddability property of differential algebraic groups into algebraic groups, there is a connected algebraic group $H$ defined over $K_0$ and a $K_0$-definable group embedding of $G$ into $H$  (to the author's knowledge the proof of this fact, for the partial case, does not appear anywhere; however, it is well known that the proof for the ordinary case \cite{Pi4} extends with little modification). We therefore assume that $G$ is a subgroup of the algebraic group $H$.

Let $H/G$ be set of left cosets of $G$ in $H$ and $\nu:H\to H/G$ be the canonical projection. By elimination of imaginaries for $DCF_{0,m}$, we may assume $H/G$ and $\nu$ are a definable set and a definable function, respectively, both over $K_0$. Let $a$ be a $\Pi$-generic point of $H$ over $K_0$. Then, since $H$ is an algebraic group, $U(a/K_0)=\omega^m\cdot d$ where $d$ is the (algebraic-geometric) dimension of $H$. Let $\al=\nu(a)$, $K=K_0\l \al\r_\Pi$ and $L=K\l a\r_\Pi$. Note that $L=K_0\l a\r_\Pi$.

We first check that $G(\bar L)=G(K_0)$. Since $K_0$ is $\Pi$-closed, if $b\in \bar L$ and $b\ind_{K_0} L$ then $b\in K_0$ (this follows from definability of types in $DCF_{0,m}$). Thus it suffices to show that $g\ind_{K_0} a$ for all $g\in G$. Let $g\in G$, then $U(g/K_0)<\omega ^m$ and by the Lascar inequalities $$\omega^m \cdot d \leq U(a,g/K_0)\leq U(a/K_0,g)\oplus U(g/K_0).$$ Hence, $U(a/K_0,g)=\omega^m\cdot d$ and so $g\ind_{K_0}a$. Now suppose $\s$ is an isomorphism from $L$ into $\U$ over $K$. Then $\nu(a)=\nu(\s a)$, and thus $a^{-1}\cdot \s a\in G$. Hence, $\s a\in L\l G\r$, which implies that $\s(L)\subset L\l G\r$. This shows that $L$ is a $G$-strongly normal extension.

Now we check that the Galois group is $G$. It follows from the proof of Theorem~\ref{group}~(i) that the Galois group is given by
$$\{g\in G: a\cdot g=\sigma(a) \text{ for some }\sigma\in \,\text{Gal}(L/K)\},$$ and thus it suffices to show that for each $g\in G$ we have that $tp(a\cdot g/K)=tp(a/K)$. Let $g\in G$, then we have that $\nu(a\cdot g)=\nu(a)=\al$. Since $\{x\in H:\, \nu x=\al\}$ is in definable bijection with $G$ and the latter has Lascar rank less than $\omega^m$, then $U(a/K)< \omega^m$ and $U(a\cdot g/K)<\omega^m$. Using Lascar inequalities again we get that $U(\al/K_0)=\omega^m\cdot d$ and also that $U(a\cdot g/K_0)=\omega^m\cdot d$. Then $a\cdot g$ is a generic point of $H$ over $K_0$, and so $tp(a\cdot g/K_0)=tp(a/K_0)$. But $\nu(a\cdot g)=\nu(a)=\al$, thus $tp(a\cdot g/K)=tp(a/K)$ as desired.
\end{proof}

\section{Relative logarithmic differential equations and their Galois extensions}\label{galex}

In ordinary differential Galois theory, Picard-Vessiot extensions are Galois extensions corresponding to certain differential equations on linear algebraic groups in the constants (i.e., linear ODE's). Pillay's generalized strongly normal extensions correspond to certain differential equations on algebraic D-groups (under the assumption that the base field is algebraically closed). In this section we introduce the logarithmic derivative on a relative D-group and the Galois extensions associated to (relative) logarithmic differential equations. We then show, under appropriate assumptions, that these are precisely the generalized strongly normal extensions of Section \ref{genstrong}. The theory we develop here extends Pillay's theory from \cite{Pi2}.

We still continue to work in our universal domain $(\U,\Pi)\models DCF_{0,m}$ and a base $\Pi$-field $K < \U$. We assume a partition $\Pi=\DD\cup \D$ with $\DD=\{D_1,\dots,D_r\}\neq \emptyset$.  

\begin{definition}
Let $(G,s)$ be a relative D-group w.r.t. $\DD/\D$ defined over $K$ (c.f. Section~\ref{relDgroup}) and $e$ be its identity. We say that $\al\in \ta G_e$ is an \emph{integrable point of $(G,s)$} if $(G,\al s)$ is a relative D-variety (but not necessarily a relative D-group). Here $\al s:G\to \ta G$ is the $\D$-section given by $(\al s)(x)=\al\cdot s(x)$, where the latter product occurs in $\ta G$. Clearly the identity of $\ta G$ is integrable.
\end{definition}

\begin{example}\label{exa}
Suppose $G$ is a linear $\D$-algebraic group defined over $K^{\DD}$ (that is, a $\D$-algebraic subgroup of GL$_n$ for some $n$). Then, by Lemma~\ref{overc}, $\ta G$ is equal to $(T_\D G)^r$, the $r$-th fibred product of the $\D$-tangent bundle of $G$, and so there is a zero section $s_0:G\to (T_\D G)^r$. Then $(G,s_0)$ is a relative D-group. In this case a point $(\operatorname{Id},A_1,\dots,A_r)\in \ta G_{\operatorname{Id}}=(\mathcal{L}_\D(G))^{r}$, where $\mathcal{L}_\D(G)=T_\D G_e$ is the $\D$-Lie algebra of $G$, is integrable if and only if it satisfies
\begin{equation}\label{lert}
D_iA_j-D_jA_i=[A_i,A_j], \quad \text{ for } i,j=1,\dots,r. 
\end{equation}
Indeed, by definition, the point $(\operatorname{Id},A_1,\dots,A_r)$ is integrable if and only if the section $$(\operatorname{Id},A_1\dots,A_r)\cdot(x,0,\dots,0)=(x,A_1 x,\dots,A_r x)$$ satisfies the integrability condition (\ref{rel1}). This means that for each $i,j$
$$d_{D_i/\D}(A_j x)(x,A_i x)=d_{D_j/\D}(A_i x)(x,A_j x),$$ which is equivalent to
$$A_jA_i x+D_i(A_j)x=A_iA_j x+D_j(A_i)x.$$
From this we get the desired equations. These equations (\ref{lert}) are the integrability conditions on $A_1,\dots,A_r$ that one finds in \cite{PM} or \cite{VS}.
\end{example}

\begin{lemma}\label{intpt}
A point $\al\in \ta G_e$ is integrable if and only if the system of $\Pi$-equations
\begin{displaymath}
\nabla_{\DD}x=\al s(x)
\end{displaymath}
has a solution in $G$.
\end{lemma}
\begin{proof}
If $\al$ is integrable then, by  Proposition \ref{lema1} (1), $\{g\in G: \nabla_{\DD}g=\al s(g)\}$ is $\D$-dense in $G$. Conversely, suppose there is $g\in G$ such that $\nabla_{\DD}g=\al s(g)$. To prove $\al$ is integrable it suffices to show that $(G,\al s)^{\sharp}$ is $\D$-dense in $G$. This follows from $(G,\al s)^{\sharp}=g(G,s)^{\sharp}$ and the fact that $(G,s)^{\sharp}$ is $\D$-dense in $G$.
\end{proof}

Let $(G,s)$ be a relative D-group defined over $K$. The \emph{logarithmic derivative associated to} $(G,s)$ is defined by
\begin{center}
$\ld:\, G \to \ta G_e$

\qquad \qquad \qquad $g\mapsto (\nabla_{\DD}g)\cdot \, (s(g))^{-1}$
\end{center}
where the product and inverse occur in $\ta G$. Note that since $\nabla$ and $s$ are sections of $\pi:\ta G\to G$, the product $(\nabla_{\DD}g)\cdot \, (s(g))^{-1}$ lies indeed in $\ta G_e$ (see the explicit formulas for the group law of $\ta G$ in Section~\ref{relDgroup}).

We now list some properties of $\ld$.

\begin{lemma}\label{propder}
\
\begin{enumerate}
\item $\ld$ is a crossed-homomorphism, i.e., $\ld(gh)=(\ld g)(g*\ld h$). Here $*$ is the adjoint action of $G$ on $\ta G_e$, that is, $g*\al:=\ta C_g (\al)$ for each $\al\in \ta G_e$ where $C_g$ denotes conjugation by $g$.
\item The kernel of $\ld$ is $(G,s)^\#$.
\item The image of $\ld$ is exactly the set of integrable points of $(G,s)$.
\item For all $a\in G$, $tp(a/K\, \ld a)$ is $(G,s)^\#$-internal and has $a$ as a fundamental system of solutions.
\end{enumerate}
\end{lemma}
\begin{proof} \

\noindent(1) This can be shown as in Pillay \cite{Pi2}. We include a sketch of the proof. An easy computation shows that $g*\al=u\al u^{-1}$ for any $u\in\ta G_g$. Thus, the adjoint action of $G$ on $\ta G_e$ is by automorphisms and
\begin{eqnarray*}
\ld (gh)
&=& (\nabla_{\DD}g)(\nabla_{\DD}h)(s(h))^{-1}(s(g))^{-1} \\
&=& (\ld g) (s(g))(\ld h)(s(g))^{-1} =(\ld g)(g*\ld h).
\end{eqnarray*}
\noindent(2) By definition of $\ld$.

\noindent(3) Follows from Lemma \ref{intpt}, since if $\al=\ld(g)$ for some $g\in G$ then $\nabla_\DD g=\al\, s(g)$.

\noindent(4) By basic properties of crossed-homomorphisms $\ld^{-1}(\ld a)=a\, ker(\ld)=a(G,s)^\#$ for all $a\in G$. Thus, if $b\models tp(a/K\,\ld a)$ then $\ld b=\ld a$, and so $b\in \ld^{-1}(\ld a)= a(G,s)^\#$.
\end{proof}

Extending the work of Pillay in \cite{Pi2}, we point out that these relative logarithmic differential equations give rise to generalized strongly normal extensions.

Let $(G,s)$ be a relative D-group defined over $K$ with $(G,s)^\#(K)=(G,s)^\#(\bar K)$, for some (equivalently any) $\Pi$-closure $\bar K$ of $K$, and $\al$ be an integrable $K$-point of $(G,s)$. By Lemma \ref{propder}~(3), the set of solutions in $G$ to $\ld x=\al$ is nonempty. Hence there is a maximal $\Pi$-ideal $\mathcal M\subset K\{x\}_\Pi$ containing $\{\ld x-\al\}\cup \I(G/K)$. It is a well known fact that every maximal $\Pi$-ideal of $K\{x\}_\Pi$ is a prime ideal (see for example \cite{Ka}). Let $a$ be a tuple of $\U$ such that $\mathcal{M}=\I(a/K)_{\Pi}$. Note that $tp(a/K)$ is therefore isolated (by the formula which sets the radical differential generators of $\I(a/K)_{\Pi}$ to zero) and so $K\l a\r_{\Pi}$ is contained in a $\Pi$-closure of $K$. Moreover, Lemma \ref{propder}~(4) tells us that, $tp(a/K)$ is $(G,s)^\#$-internal and $a$ is a fundamental system of solutions. Hence, by Proposition \ref{inter}, $K\l a\r_{\Pi}$ is a $(G,s)^\#$-strongly normal extension of $K$. 

\begin{definition}
We call the above $K\l a\r_\Pi$ the \emph{Galois extension associated to $\ld x=\al$}. The Galois group associated to this generalized strongly normal extension is called the \emph{Galois group associated to $\ld x=\al$}.
\end{definition}

Let us point out that the above construction does not depend on the choice of $a$ (up to isomorphism over $K$). Indeed, if $b$ is another solution such that $\I(b/K)_{\Pi}$ is a maximal $\Pi$-ideal, then both $tp(a/K)$ and $tp(b/K)$ are isolated and so we can find a $\Pi$-closure $\bar K$ of $K$ containing $b$ and an embedding $\phi:K\{a\}_{\Pi}\to \bar K$ over $K$. Since $\al$ is a $K$-point, $\ld \phi(a)=\al$, thus, by Lemma \ref{propder}~(1), $b^{-1}\phi(a)\in (G,s)^{\sharp}(\bar K)=(G,s)^{\sharp}(K)$. Hence, $\phi(a)$ and $b$ are interdefinable over $K$ and so $K\l a\r_{\Pi}$ is isomorphic to $K\l b\r_{\Pi}$ over $K$. This argument actually shows that there is exactly one such extension in each $\Pi$-closure of $K$.

\

\begin{remark}[On the condition $(G,s)^\#(K)=(G,s)^\#(\bar K)$] \
\begin{itemize}
\item [(i)] Let $G$ be a $\D$-group defined over $K^\DD$ and $s_0:G\to \ta G=(T_\D G)^r$ its zero section. If $K^\DD$ is $\D$-closed then $(G,s_0)^\#(K)=(G,s_0)^\#(\bar K)$. On the other hand, if $(G,s_0)^\#(K)=(G,s_0)^\#(\bar K)$ and $\D$-type$(G)=|\D|$ then $K^\DD$ is $\D$-closed.
\item [(ii)] Let $(G,s)$ be a relative D-group defined over $K$. If $(G,s)^\#$ is the Galois group of a generalized strongly normal extension of $K$ then $(G,s)^\#(K)=(G,s)^\#(\bar K)$ (see (5) of Theorem~\ref{group}).
\end{itemize}
\end{remark}

\begin{example}[{\it The linear case}] Suppose $G=GL_n$ and $s_0:G\to \ta G=(T_\D G)^r$ is the zero section. By Proposition \ref{prog}, $T_\D G=TG$ the (algebraic) tangent bundle of $G$, and so $\ta G_{\operatorname {Id}}=\{\operatorname{Id}\}\times (Mat_n)^r$. If $\al=(\operatorname{Id},A_1,\dots,A_r)\in \ta G_{\operatorname {Id}}$, then the logarithmic differential equation $\ell_{s_0} x=\al$ reduces to the system of linear differential equations $$D_1 x=A_1 x,\;\dots\;, D_r x=A_r x.$$ As we already pointed out in Example \ref{exa}, $\al$ will an integrable point if and only if $$D_iA_j-D_jA_i=[A_i,A_j]\; \text{ for } i,j=1,\dots,r.$$ Also, in this case, $(G,s_0)^\#(K)=(G,s_0)^\#(\bar K)$ if and only if $K^\DD$ is $\D$-closed. Thus, the Galois extensions of $K$ associated to logarithmic differential equations of $(GL_n,s_0)$ are precisely the parametrized Picard-Vessiot extensions considered by Cassidy and Singer in \cite{PM}.
\end{example}

\begin{proposition}\label{char}
Let $(G,s)$ be a relative D-group defined over $K$ with $(G,s)^{\sharp}(K)=(G,s)^{\sharp}(\bar K)$ and $\al$ be an integrable $K$-point of $(G,s)$. Let $L$ be a $\Pi$-field extension of $K$ generated by a solution to $\ld x=\al$. Then, $L$ is the Galois extension associated to $\ld x=\al$ if and only if $(G,s)^{\sharp}(K)=(G,s)^{\sharp}(\bar L)$ for some (any) $\Pi$-closure $\bar L$ of $L$.
\end{proposition}
\begin{proof}
By the above discussion, the Galois extension associated to $\ld x=\al$ is a $(G,s)^\#$-strongly normal. Thus, $(G,s)^{\sharp}(K)=(G,s)^{\sharp}(\bar L)$. For the converse, suppose $L=K\l b\r_{\Pi}$ where $b$ is a solution to $\ld x=\al$. Then, since $tp(b/K)$ is $(G,s)^\#$-internal and $b$ is a fundamental system of solutions, $L$ is a $(G,s)^\#$-strongly normal extension and so $L$ is contained in a $\Pi$-closure $\bar K$ of $K$. Let $a$ be a tuple from $\bar K$ such that $\I(a/K)_{\Pi}$ is a maximal $\Pi$-ideal. Then $K\l a\r_{\Pi}$ is the Galois extension associated to $\ld x=\al$. But, by Lemma \ref{propder}, $b^{-1}a\in (G,s)^{\sharp}(\bar K)=(G,s)^\#(K)$, and hence $L=K\l a\r_{\Pi}$. 
\end{proof}

The proof of Lemma 3.9 of \cite{Pi2} extends directly to the partial case and yields the following proposition.

\begin{proposition}\label{gal}
Let $(G,s)$ and $\al$ be as in Proposition \ref{char}. Then the Galois group associated to $\ld x =\al$ is of the form $(H,s_H)^{\sharp}$ for some relative D-subgroup $(H,s_H)$ of $(G,s)$ defined over $K$. Moreover, if $a$ is a solution to $\ld x=\al$ such that $\I(a/K)_{\Pi}$ is a maximal $\Pi$-ideal, then the action of $(H,s_H)^\#$ on $tp(a/K)^{\U}$ is given by $h.b=(aha^{-1})b$.
\end{proposition}
\begin{proof}
We give a slightly more direct argument than what is found in \cite{Pi2}. Let $L=K\l a\r_{\Pi}$ be the Galois extension associated to $\ld x=\al$, where $a$ is a solution to $\ld x=\al$ and $\I(a/K)_{\Pi}$ is a maximal $\Pi$-ideal. Let $Z$ be the $\Pi$-locus of $a$ over $K$, note that $Z=tp(a/K)^{\U}$ and that $Z$ is a $\Pi$-algebraic subvariety of $(G,\al s)^{\sharp}$. Let $f$ be the multiplication on $G$ and $Y=\{g\in (G,s)^{\sharp}:Zg=Z\}$. Following the construction of Theorem \ref{group} (1) with the data $((G,s)^\#,a,Y,f)$, we get a bijection $\mu:\,$Gal$(L/K)_{\Pi}\to Y$ defined by $\mu(\sigma)=a^{-1}\sigma(a)$. It follows that $\mu$ is in fact a group isomorphism, where $Y$ is viewed as a subgroup of $(G,s)^{\sharp}$. Let $H$ be the $\D$-closure of $Y$ over $K$. Since $Y$ is a $\Pi$-algebraic subgroup of $(G,s)^\#$ defined over $K$, $H$ equipped with $s_H:=s|_H$ is a $\DD/\D$-subgroup of $(G,s)$ defined over $K$ and $(H,s_H)^\#=Y$ (see Lemma \ref{corre}).

The moreover clause follows by (\ref{action1}) in the proof of Theorem \ref{group}~(1).
\end{proof}

The Galois correspondence given by Theorem \ref{normal} specializes to this context and, composed with the bijective correspondence between relative D-subgroups of a given relative D-group and the $\Pi$-algebraic subgroups of the sharp points given in Lemma \ref{corre}, yields the following correspondence.

\begin{corollary}\label{corresp}
Let $(G,s)$ and $\al$ be as in Proposition \ref{char}. Let $L$ be the Galois extension associated to $\ld x=\al$ with Galois group $(H,s_H)^\#$. Then there is a Galois correspondence between the intermediate $\Pi$-fields (of $K$ and $L$) and the relative D-subgroups of $(H,s_H)$ defined over $K$.
\end{corollary}

Now, let $(G,s)$ and $\al$ be as in Proposition \ref{char}, and $L$ be the Galois extension associated to $\ld x=\al$ with Galois group $(H,s_H)^\#$. Suppose that there is $\D'\subseteq \D$ such that $G$ is a $\D'$-algebraic group and $s$ is a $\D'$-section (recall that in this case $\ta G=\tau_{\DD/\D'}G$). Let $\Pi'=\DD\cup\D'$. We can consider the Galois extension $L'$ and Galois group $(H',s_{H'})^\#$ associated to $\ld x=\al$ when the latter is viewed as a logarithmic $\Pi'$-equation (note that $\al$ is also an integrable point when $(G,s)$ is viewed as a relative D-group w.r.t. $\DD/\D'$). In other words, $L'$ is a $\Pi'$-field extension of $K$ of the form $K\l a\r_{\Pi'}$ where $\ld a=\al$ and $\I(a/K)_{\Pi'}$ is a maximal $\Pi'$-ideal of $K\{ x\}_{\Pi'}$, and $(H',s_{H'})$ is a relative D-subgroup of $(G,s)$ w.r.t. $\DD/\D'$ such that $(H',s_{H'})^\#$ is (abstractly) isomorphic to the group of $\Pi'$-automorphisms Gal$(L'/K)_{\Pi'}$. We have the following relation between the Galois extensions $L$ and $L'$, and the groups $H$ and $H'$:

\begin{proposition}\label{as}
Let $L$, $L'$, $H$ and $H'$ be as above. If $L=K\l a\r_\Pi$ then $L'=K\l a\r_{\Pi'}$, and $H'$ equals the $\D'$-closure (in the $\D'$-Zariski topology) of $H$ over $K$.
\end{proposition}
\begin{proof}
Since $(G,s)^\#(K)=(G,s)^\#(\bar L)$ for some $\Pi$-closure $\bar L$ of $L$, then $(G,s)^\#(K)=(G,s)^\#(\overline{K\l a\r}_{\Pi'})$ for some $\Pi'$-closure $\overline{K\l a\r}_{\Pi'}$ of $K\l a\r_{\Pi'}$. Now Proposition \ref{char} implies that $K\l a\r_{\Pi'}$ is the Galois extension associated to $\ld x=\al$ when viewed as a logarithmic $\Pi'$-equation, and so $L'=K\l a\r_{\Pi'}$. Now, to show that $H'$ is the $\D'$-closure of $H$ over $K$ it suffices to show that $(H',s_{H'})^\#$ is the $\Pi'$-closure of $(H,s_H)^\#$ over $K$. First we check that $(H,s_H)^\#\subseteq (H',s_{H'})^\#$. Let $h\in (H,s_H)^\#$, then there is $\sigma\in$ Gal$(L/K)_\Pi$ such that $h=a^{-1}\sigma(a)$. Since $\s$ restricts to an element of Gal$(L'/K)_{\Pi'}$, $a^{-1}\sigma(a)\in (H',s_{H'})^\#$, showing the desired containment. Let $Y$ be the $\Pi'$-closure of $(H,s_H)^\#$ over $K$, then adapting the proof of Lemma \ref{corre} one can see that $Y$ is a $\Pi'$-algebraic subgroup of $(H',s_{H'})^\#$. The fixed field of $(H,s_H)^\#$ is $K$, then the fixed field of $Y$ (as a $\Pi'$-subgroup of $(H',s_{H'})^\#$) is also $K$. Hence, by the Galois correspondence, $Y=(H',s_{H'})^\#$.
\end{proof}

In the case when $(G,s)=(GL,s_0)$ and $\D'=\emptyset$, Proposition \ref{as} specializes to Proposition 3.6 of \cite{PM}.

\

We finish this section by showing that, under some natural assumptions on the differential field $K$, every generalized strongly normal extension of $K$ is the Galois extension of a (relative) logarithmic differential equation. This is in analogy with Remark 3.8 of \cite{Pi2}.

\begin{theorem}\label{main}
Let $X$ be a $K$-definable set and $L$ an $X$-strongly normal extension of $K$. Suppose $\D$ is a set of fewer than $m$ linearly independent elements of $\operatorname{span}_{K^\Pi}\Pi$ that bounds the $\Pi$-type of $X$, and such that $K$ is $\D$-closed. If we extend $\D$ to a basis $\DD\cup\D$ of $\operatorname{span}_{K^\Pi}\Pi$, then there is a connected relative D-group $(H,s)$ w.r.t. $\DD/\D$ defined over $K$ and $\al$ an integrable $K$-point of $(H,s)$ such that $L$ is the Galois extension associated to $\ld x=\al$.
\end{theorem}
\begin{proof}
We just need to check that the argument given in Proposition 3.4~(ii) of \cite{Pi} for the finite-dimensional case extends to this setting. Let $\G$ be the Galois group of $L$ over $K$. Note that, by Theorem \ref{group}~(5), $\G$ is connected and that, by Lemma~\ref{onint}~(2) and Theorem \ref{group}~(2), $\D$ also bounds the $\Pi$-type of $\G$. Thus, Theorem \ref{co} implies that $\G$ is of the form $(H,s)^\#$ for some relative D-group $(H,s)$ w.r.t. $\DD/\D$ defined over $K$.

Now, let $b$ be a tuple such that $L=K\l b\r_{\Pi}$ and let $\bar K$ be a $\Pi$-closure of $K$ that contains $L$. Let $\mu$ be the canonical isomorphism from Gal$(L/K)$ to $(H,s)^\#$. We know there is some $K$-definable function $h$ such that $\mu(\sigma)=h(b,\sigma(b))$ for all $\sigma \in$ Gal$(L/K)$. Consider the map $\nu:Aut_{\Pi}(\bar K/K)\to H(\bar K)$ defined by $\nu(\sigma)=h(b,\sigma(b))$. Let $\sigma_i\in Aut_{\Pi}(\bar K/K)$ for $i=1,2$, and denote by $\sigma_i'$ the unique elements of Gal$(L/K)$ such that $\sigma_i'(b)=\sigma_i(b)$. We have 
\begin{eqnarray*}
\nu(\sigma_1\comp \sigma_2)
&=& h(b,\sigma_1\comp\sigma_2(b))=h(b,\sigma_1'\comp\sigma_2'(b)) \\
&=& \mu(\sigma_1'\comp\sigma_2')=\mu(\sigma_1')\, \mu(\sigma_2') \\
&=& h(b,\sigma_1'(b))\, h(b,\sigma_2'(b))=h(b,\sigma_1(b))\, h(b,\sigma_2(b)) \\
&=& \nu(\sigma_1)\, \nu(\sigma_2)=\nu(\sigma_1)\, \sigma_1(\nu(\sigma_2)) \\
\end{eqnarray*}
where the last equality follows from $(H,s)^\#(\bar K)=(H,s)^\#(K)$. In the terminology of (\cite{Ko2}, Chap. 7) or \cite{Pi3} we say that $\nu$ is a definable cocycle from $Aut_{\Pi}(\bar K/K)$ to $H$. Using the fact that $K$ is $\D$-closed, the argument from Proposition 3.2 of \cite{Pi} extends to show that every definable cocycle is cohomologous to the trivial cocycle. In other words, we get a tuple $a\in H(\bar K)$ such that $\nu(\sigma)=a^{-1}\,\sigma(a)$ for all $\sigma\in Aut_{\Pi}(\bar K/K)$. 

\noindent {\bf Claim.} $K\l a\r_{\Pi}=K\l b \r_{\Pi}$. \\
Towards a contradiction suppose $a\notin K\l b\r_{\Pi}$. Since $a\in \bar K$ and $\bar K$ is also a $\Pi$-closure of $K\l b\r_{\Pi}$ (see \cite{Pi6}, Chap. 8), we get that $tp(a/K,b)$ is isolated. Thus we can find $c\in \bar K$ realising $tp(a/K,b)$ such that $c\neq a$. Then there is $\sigma\in Aut_{\Pi}(\bar K/K\l b\r_{\Pi})$ such that $\sigma(a)=c$ (see \cite{Pi6}, Chap. 8), but this is impossible since $\sigma$ fixes $b$ iff $\nu(\sigma)=e$ (where $e$ is the identity of $H$) iff $a^{-1}\sigma(a)=e$ iff $\sigma$ fixes $a$. The other containment is analogous. This proves the claim.

By Proposition \ref{char}, all that is left to show is that $\al:=l_s(a)$ is a $K$-point. Let $\sigma\in Aut_{\Pi}(\U/K)$, then $a^{-1}\sigma(a)\in (H,s)^\#$. Thus, $\nabla_{\DD}(a^{-1}\sigma(a))=s(a^{-1}\sigma(a))$ and so
\begin{eqnarray*}
\sigma(\al)
&=& \sigma((\nabla_{\DD}a)(s(a))^{-1})\\
&=& (\nabla_{\DD}\sigma(a))(s(\sigma(a)))^{-1}\\
&=& (\nabla_{\DD}a)(s(a))^{-1} \\
&=& \al,
\end{eqnarray*}
as desired.
\end{proof}

\

\section{Two examples}

In this section we give two non-linear examples of Galois groups associated to logarithmic differential equations. Our examples are modeled after Pillay's non-linear example given in \cite{Pi2}.

First we exhibit a finite-dimensional non-linear Galois group in two derivations:

\begin{example}
Let $\Pi=\{\d_t,\d_w\}$. Let $G={\mathbb G}_m\times {\mathbb G}_a$ and $s:G\to \tau_{\Pi/\emptyset}G=(TG)^2$ be the (algebraic) section defined by $s(x,y)=(x,y,xy,0,xy,0)$. Here $\mathbb{G}_m$ and $\mathbb{G}_a$ denote the multiplicative and additive groups, respectively. Then $(G,s)$ is a relative D-group w.r.t. $\Pi/\emptyset$, and the logarithmic derivative $\ld:G\to (TG)^2_{(1,0)}$ is given by $$\ld(x,y)=\left(1,0,\frac{\d_t x}{x}-y,\d_t y,\frac{\d_w x}{x}-y,\d_w y\right).$$ Thus, $$(G,s)^\#=\{(x,y)\in G:\, \d_t y=\d_w y=0 \text{ and } \d_t x=\d_w x=yx\}.$$ We take the ground $\Pi$-field to be $K:=\mathbb{C}(t,e^{ct},e^{cw}:c\in \mathbb{C})$, where we regard $t$ and $w$ as two complex variables, and the $\Pi$-field extension $L:=K(w,e^{2wt+w^2})$. Then $L$ is contained in a $\Pi$-closure $\bar K= \bar{\mathbb{C}}$ of $K$ and $\mathbb C$. 

We now show that $(G,s)^\#(\bar K)=(G,s)^\#(K)$. Let $(a,b)\in (G,s)^\#(\bar K)$, then $b\in \bar{K}^\Pi=\mathbb C$ and $$\d_t\left(\frac{a}{e^{b(t+w)}}\right)=\d_w\left(\frac{a}{e^{b(t+w)}}\right)=0.$$ Thus $a=ce^{b(t+w)}$ for some $c\in\mathbb C$, and so $a\in K$.

Now, as $L$ is generated by $(e^{2wt+w^2},2w)$ and this pair is a solution to
\begin{displaymath}
\left\{
\begin{array}{c}
\d_t x=yx \\
\d_t y=0 \\
\d_w x=(y+2t)x \\
\d_w y=2
\end{array}
\right. ,
\end{displaymath}
$L$ is the Galois extension associated to $\ld (x,y)=(1,0,0,0,2t,2)$. Also, since the transcendence degree of $L$ over $K$ is $2$ and $(G,s)^\#$ is a connected $\Pi$-algebraic group whose Kolchin polynomial is constant equal to $2$, then the Galois group associated to $\ld (x,y)=(1,0,0,0,2t,2)$ is $(G,s)^\#$.
\end{example}

Suppose $\DD\cup\D$ is a partition of $\Pi$. Note that while Proposition \ref{bab} shows that for every connected relative D-group $(G,s)$ the subgroup $(G,s)^\#$ is the Galois group of a generalized strongly normal extension, it is not known if $(G,s)^\#$ is the Galois group of a logarithmic differential equation. The following proposition gives a sufficient condition on $G$ that allows a construction of a Galois extension of a logarithmic differential equation on $(G,s)$ with Galois group $(G,s)^\#$.

\begin{proposition}\label{abo}
Let $(G,s)$ be a relative D-group and suppose $G$ is a connected algebraic group. Then there is $\Pi$-field $K$ and an integrable $K$-point $\al$ of $(G,s)$ such that the Galois group associated to $\ld x=\al$ is  $(G,s)^\#$.
\end{proposition}
\begin{proof}
We follow the construction given in the proof of Proposition \ref{bab}. Let $K_0$ be a $\Pi$-closed field over which the $\DD/\D$-group $(G,s)$ is defined and let $a$ be a $\Pi$-generic point of $G$ over $K_0$. Let $\al=\ld a$, $K=K_0\l \al\r_\Pi$ and $L=K\l a\r_\Pi$. Using now the same arguments as in the proof of Proposition~\ref{bab}, we get that $L$ is a $(G,s)^\#$-strongly normal extension of $K$ and $(G,s)^\#$ is the Galois group of $L$ over $K$. Hence, by Proposition \ref{char}, $L$ is the Galois extension of $K$ associated to $\ld x=\al$ and $(G,s)^\#$ is the associated Galois group.
\end{proof}

We finish with an example of an infinite-dimensional Galois group associated to a non-linear logarithmic differential equation.

\begin{example}
Let $\Pi=\{\d_1,\d_2\}$. Let $G={\mathbb{G}}_m \times {\mathbb{G}}_a$ and $s:G\to \tau_{\d_2/\d_1}G=TG$ be the $\d_1$-section given by $$s(x,y)=(x,y,xy,\d_1 y).$$ Then $(G,s)$ is a relative D-group w.r.t. $\d_2/\d_1$. The logarithmic derivative  $\ld:G\to TG_e$ is given by $$\ld (x,y)=(1,0,\frac{\d_2 x}{x}-y, \d_2 y-\d_1 y).$$ Thus the sharp points are given by $$(G,s)^\#=\{(x,y)\in G: \, \d_2 x=xy \text{ and } \d_2y=\d_1 y\}.$$ Note that $\Pi$-type$(G,s)^\#=1$ and $\Pi$-dim$(G,s)^\#=2$. By Proposition \ref{abo}, there is $(\al_1,\al_2)$ such that $(G,s)^\#$ is the Galois group associated to the non-linear logarithmic differential equation 
\begin{displaymath}
\left\{
\begin{array}{c}
\d_2 x=x(y+\al_1) \\
\d_2 y=\d_1 y+\al_2
\end{array}
\right. .
\end{displaymath}
\end{example}

\chapter{Geometric axioms for $DCF_{0,m}$}\label{chapaxioms}\let\thefootnote\relax\footnotetext{The results in this chapter form the basis for a paper entitled ``Geometric axioms for differentially closed fields with several commuting derivations'' and its corrigendum that appear in the \emph{Journal of Algebra}, Vol. 362 and Vol. 382, respectively}

In this chapter we establish three characterizations of the existentially closed differential fields with $\ell+1$ commuting derivations in terms of the geometry of $DCF_{0,\ell}$. We treat the $(\ell+1)^{th}$-derivation as a differential derivation with respect to the first $\ell$, and use relative prolongations to produce solutions to certain systems of differential equations. We observe that one of these characterizations is first order.

\section{A review of the Pierce-Pillay axioms}\label{piax}

In this section we briefly review the case of a single derivation and describe the geometric axiomatization given by Pierce and Pillay. Then we present an example, constructed by Hrushovski, showing that the naive extension of their axioms to the partial case is false.

Recall that in ordinary differential fields the prolongation functor is simply written as $\tau$ as opposed to $\tau_{\d/\emptyset}$, and it goes from the category of algebraic varieties to itself.

\begin{fact}[Pierce-Pillay axioms \cite{PiPi}]
$(K,\d)\models DCF_0$ if and only if
\begin{enumerate}
\item $K\models ACF_0$
\item Let $V$ and $W\subseteq \tau V$ be irreducible affine algebraic varieties defined over $K$ such that $W$ projects dominantly onto $V$. Let $O_V$ and $O_W$ be nonempty Zariski-open subsets of $V$ and $W$, respectively, defined over $K$. Then there is a $K$-point $a\in O_V$ such that $(a,\d a)\in O_W$.
\end{enumerate}
\end{fact}

This geometric characterization of $DCF_0$ is indeed expressible in a first order way. First, one needs to check that irreducibility of a Zariski-closed set is a definable condition on the parameters. This means that if $f_1(u,x),\dots, f_s(u,x)$ is a set of polynomials over $\QQ$ in the variables $u=(u_1,\dots,u_r)$ and $x=(x_1,\dots,x_n)$, then there exists a formula $\phi(u)$ (in the language of rings) such that for any algebraically closed field $K$ of characteristic zero and $a\in K^r$ we have that $K\models \phi(a)$ if and only if the Zariski-closed set $$\{b\in K^n: f_1(a,b)=\cdots=f_s(a,b)=0\}$$ is irreducible. The latter is equivalent to asking for the ideal $(f_1(a,x), \dots,f_s(a,x))$ to be a prime ideal of $K[x]$. Hence, the existence of such a $\phi$ is a consequence of the well known bounds to check primality of ideals in polynomial rings in finitely many variables (see for example \cite{Van}). 

Second, one needs to check that dominant projections of Zariski-closed sets onto irreducible Zariski-closed sets is a definable condition on the parameters. More precisely, if $f_{u,v}:W_v\to V_u$ is an algebraic family of polynomial maps between Zariski-closed sets over $\QQ$, then there is a formula (in the language of rings) such that for any algebraically closed field of characteristic zero, $K$, and tuple $(a,b)$ from $K$ we have that $K\models \phi(a,b)$ if and only if $V_a$ is irreducible and $f_{a,b}(W_b)$ is a Zariski-dense in $V_a$. Once we know that $V_a$ is irreducible, the latter condition is equivalent to $f_{a,b}(W_b)$ having the same algebraic-geometric dimension as $V_a$. Hence, the existence of such a $\phi$ is a consequence of the fact that dimension is definable in $ACF$.

Third, one needs to know that the family of prolongations of a definable family of irreducible algebraic varietes is again a definable family. This means that if $V_u$ is an algebraic family of Zariski-closed sets over $\QQ$, then there is a $\d$-algebraic family $W_u$ of Zariski closed sets over $\QQ$ such that for any algebraically closed differential field $K$ (of characteristic zero) and tuple $a$ from $K$ we have that if $V_a$ is irreducible then $W_a=\tau(V_a)$. To see this, suppose $V_u$ is given by $f_1(u,x)=0,\dots,f_s(u,x)=0$, and let $W_u$ be given by $$f_i(u,x)=0 \text{ and } \sum_{j=1}^n\frac{\partial f_i}{\partial x_j}(u,x)y_j+f_i(\d u,x)=0, \quad i=1,\dots,s.$$ For arbitrary $a$, $W_a\neq \tau(V_a)$; however, if we knew that $(f_1(a,x),\dots,f(a,x))=\I(V/K)$ then we would get equality. Hence, for those $a$ such that $(f_1(a,x),\dots,f_s(a,x))$ is prime we get $W_a=\tau(V_a)$. But we know primality is a definable condition on the parameters.

Now, if one tries to extend these axioms to the partial case $\D=\{\d_1,\dots,\d_m\}$ by simply replacing in condition (2) the prolongation functor $\tau$ for $\tau_\D:=\tau_{\D/\emptyset}$ and $(a,\d a)$ for $(a,\d_1a,\dots,\d_m a)$, then the corresponding statement will not be true for the existentially closed models as the following example shows:

\begin{example}
We will use a construction similar to the one of Example~\ref{lastex}. Let $\D=\{\d_1,\d_2\}$. Let $V=\mathbb{A}^1$ and $W\subseteq \tau_\D V=\mathbb{A}^3$ be defined by 
$$W=\{(x,y,z):\, y=x \text{ and } z=x+1\}.$$
Note that $W$ is irreducible and projects onto $V$. We claim that there is no $\D$-field $K$ such that there exists a $K$-point $a\in V$ with $(a,\d_1 a,\d_2 a)\in W$. Towards a contradiction suppose there is such a $\D$-field $K$ and $K$-point $a$. Then, from the equations defining $W$, we get that $\d_1 a=a$ and $\d_2 a=a+1$. This yields $\d_2\d_1 a=a+1$ and $\d_1\d_2 a=a$. Since $\d_2\d_1 a=\d_1\d_2 a$, we obtain $1=0$, and so we have a contradiction.
\end{example}

In the following section we take a different approach to the problem of geometric axioms for partial differentially closed fields. We establish a characterization (and ultimately an axiomatization) of $DCF_{0,\ell+1}$ which is geometric relative to the theory $DCF_{0,\ell}$. The characterization we present is very much in the spirit of the Pierce-Pillay axioms, and can be used essentially in the same way. It is worth pointing out that first order axiomatizability is a complication that arises in our setting (as explained above this issue does not appear in the ordinary case). Characteristic sets of prime differential ideals are behind our solution to this problem.

It is worth mentioning that Pierce-Pillay type axiomatizations appear in various other contexts: difference fields \cite{CH}, ordinary differential-difference fields \cite{Bu}, fields with commuting Hasse-Schmidt derivations in positive characteristic \cite{Kow}, fields with free operators \cite{MS}, and theories having a ``geometric notion of genericity'' \cite{Hi}. However, none of the techniques used in these works seem to extend to our context.

\section{Geometric characterizations and an axiomatization}\label{thegeo}

In this section we work in the language $\mathcal{L}_{\ell+1}$ with derivations $\{\d_1,\dots,\d_{\ell+1}\}$. Fix a sufficiently saturated $(\U,\Pi)\models DCF_{0,\ell+1}$ and a small $\Pi$-subfield $K$ of $\U$. 

In Chapter \ref{chapro} we have already presented/developed all the notions and machinery needed to state and prove the following characterization.

\begin{theorem}\label{charactheo}
Let $\D=\{\d_1,\dots,\d_\ell\}$ and $D=\d_{\ell+1}$. Then, $(K,\Pi)\models DCF_{0,\ell+1}$ if and only if 
\begin{enumerate}
\item $(K,\D)\models DCF_{0,\ell}$
\item Suppose $V$ and  $W$ are irreducible affine $\D$-algebraic varieties defined over $K$ such that $W\subseteq \tau_{D/\D} V$ and $W$ projects $\D$-dominantly onto $V$. If $O_V$ and $O_W$ are nonempty $\D$-open subsets of $V$ and $W$ respectively, defined over $K$, then there exists a $K$-point $a\in O_V$ such that $\nabla a\in O_W$. 
\end{enumerate}
\end{theorem}

\begin{proof}
Suppose $(K,\Pi)\models DCF_{0,\ell+1}$, and $V$, $W$, $O_V$ and $O_W$ are as in condition (2). Let $(a,b)$ be a $\D$-generic point of $W$ over $K$; that is, $\I(a,b/K)_\D=\I(W/K)_\D$. Then $(a,b)\in O_W$. Since $(a,b)\in \tau_{D/\D} V$ we have that $\et f_{a}(b)=0$ for all $f\in \I(V/K)_\D$. The fact that $W$ projects $\D$-dominantly onto $V$ implies that $a$ is a $\D$-generic point of $V$ over $K$, so $a\in O_V$ and $\I(a/K)_\D=\I(V/K)_\D$. Hence, $\et f_{a}(b)=0$ for all $f\in \I(a/K)_\D$. By Fact \ref{exx}, there is a unique $\D$-derivation $D': K\{a\}\to \U$ extending $D$ such that $D' a=b$. By Fact \ref{exx2}, we can extend $D'$ to all of $\U$, call it $D''$. Hence, $(\U,\D\cup\{D''\})$ is a differential field extending $(K,\Pi)$. Since $a\in O_V$, $(a,b)\in  O_W$ and $D'' a=b$, we get a point $(a',b')$ in $K$ such that $a'\in O_V$, $(a',b')\in O_W$ and $Da'=b'$.

The converse is essentially as in \cite{PiPi}. For the sake of completeness we give the details. Let $\phi(x)$ be a conjunction of atomic $\mathcal{L}_{\ell+1}$-formulas over $K$ with a realisation $a$ in $\U$. Let
\begin{displaymath}
\phi(x)=\psi(x,\d_{\ell+1}x,\dots,\d_{\ell+1}^r x)
\end{displaymath}
where $\psi$ is a conjunction of atomic $\mathcal{L}_\ell$-formulas over $K$ and $r>0$. Let $c=(a,D a,\dots,D^{r-1}a)$ and $V$ be the $\D$-locus of $c$ over $K$. Let $W$ be the $\D$-locus of $\nabla c$ over $K$. Let
\begin{displaymath}
\chi(x_0,\dots,x_{r-1},y_0,\dots,y_{r-1}) :=\psi(x_0,\dots,x_{r-1},y_{r-1}) \land \left(\land_{i=1}^{r-1}x_i = y_{i-1}\right)
\end{displaymath}
then $\chi$ is realised by $(c,Dc)=\nabla c$. Since $(K,\D)\models DCF_{0,\ell}$, and $c$ and $\nabla c$ are $\D$-generic points of $V$ and $W$ respectively, over $K$, we have that $W$ projects $\D$-dominantly onto $V$. Also, since $\nabla c\in\tau_{D/\D} X$, we have $W\subseteq\tau_{D/\D} V$. Applying condition (2) with $V=O_V$ and $W=O_W$, there is a $K$-point $d$ in $V$ such that $\nabla d\in W$. Let $d=(d_0,\dots,d_{r-1})$ then $(d_0,\dots, d_{r-1},D d_0,\dots,D d_{r-1})$ realises $\chi$. Thus, $(d_0, D d_0,\dots,D^r d_0)$ realises $\psi$. Hence, $ d_0$ is a tuple of $K$ realising $\phi$. This proves that $(K,\Pi)\models DCF_{0,\ell+1}$.
\end{proof}

\begin{remark}\label{without}
As can be seen from the proof, it would have been equivalent in condition (2) to take $O_V=V$ and $O_W=W$. Also note that, under the convention that $DCF_{0,0}$ is the theory of algebraically closed fields of characteristic zero $ACF_0$, when $\ell=0$ this is exactly the Pierce-Pillay axioms.
\end{remark}

It remains open as to whether condition (2) of Theorem~\ref{charactheo} is expressible in a first order way for $m>0$. One of the problems lies in determining the definability of differential irreducibility and differential dominance in definable families of differential algebraic varieties (this problem seems to be related to the generalized Ritt problem, see \cite{Ov} and \cite{HKM}). Another problem is to determine whether the family of prolongations of a definable family of differential algebraic varieties is again a definable family.

We get around these problems by injecting some differential algebra into the above characterisation, modifying it to make it first order. That is, we use characteristic sets of prime differential ideals (see Section~\ref{ondiff} for the definition of characteristic sets). We will use the following notation.

\begin{definition}\label{sepin}
Given a characeristic set $\L$ of a prime $\D$-ideal,  by $\V^*(\L)$ we mean $\V(\L)\setminus \V(H_\L)$ where $H_\L$ is the product of the initials and separants of the elements of $\L$ (see Section~\ref{ondiff}). 
\end{definition}

Recall that the notation $\V^*(\L)$ has already been used in Proposition~\ref{basic}.

\begin{theorem}\label{characprime}
Let $\D=\{\d_1,\dots,\d_\ell\}$ and $D=\d_{\ell+1}$. Then, $(K,\Pi)\models DCF_{0,\ell+1}$ if and only if
\begin{enumerate}
\item [(i)] $(K,\D)\models DCF_{0,\ell}$
\item [(ii)] Suppose $\L$ and $\Ga$ are characteristic sets of prime $\D$-ideals of $K\{x\}_{\D}$ and $K\{x,y\}_{\D}$ respectively, such that $$\V^*(\Ga)\subseteq \V(f,\, \et f:\, f\in \L).$$ Suppose $O$ is a nonempty $\D$-open subset of $\V^*(\L)$ defined over $K$ such that the projection of $\V^*(\Ga)$ to $\V(\L)$ contains $O$. Then there is a $K$-point $a \in \V^*(\L)$ such that $\nabla_{D}a \in \V^*(\Ga)$.
\end{enumerate}
\end{theorem}

\begin{proof}
Suppose $(K,\Pi)\models DCF_{0,\ell+1}$, and we are given  $\L$, $\Ga$ and $O$ satisfying the hypotheses of~(ii). Then $\L$ is a characteristic set of the prime differential ideal $\P=[\L]:H_\L^\infty$ and, by Proposition~\ref{basic}, $\V^*(\L)=\V(\P)\setminus \V(H_\L)$. So $O_V:=\V^*(\L)$ is a nonempty $\D$-open subset of the irreducible affine $\D$-algebraic variety $V:=\V(\P)$. Similarly, $O_W:=\V^*(\Ga)$ is a nonempty $\D$-open subset of the irreducible affine $\D$-algebraic variety $W:=\V\left([\Ga]:H_\Ga^\infty\right)$.

Next we show that $\tau_{D/\D} V|_O=\V(f,\et f: f\in\L)|_O$. Recall that, by definition, $\tau_{D/\D} V$ is $\V(f,\et f:f\in \I(V/K)_\D)$. It is easy to see that $\V(f,\et f: f\in\L)=\V(f,\et f:f\in[\L])$. So, supposing that $(a,b)$ is a root of $f$ and $\et f$ for all $f\in[\L]$, and $a\in O$, we need to show that $(a,b)$ is a root of $\et g$ for all $g\in\mathcal \I(V/K)_\D$. But $\I(V/K)_\D=[\L]:H_\L^\infty$, so $H_\L^\ell g\in[\Lambda]$ for some $\ell$. We get
\begin{eqnarray*}
0&=& \et \left(H_\L^\ell g\right)(a,b) \ \ \ \ \ \ \ \ \ \text{ as $H_\Lambda^\ell g\in[\L]$}\\
&=& H_\L^\ell(a)\et g(a,b)+g(a)\et  H_\L^\ell (a,b)\\
&=& H_\L^\ell(a)\tau g(a,b).
\end{eqnarray*}
Since $O$ is disjoint from $\V(H_\L)$ we have that $\tau g(a,b)=0$, as desired. 

It follows that a nonempty $\D$-open subset of $W$ is contained in $\tau_{D/\D} V$, and hence, by irreducibility, $W\subseteq \tau_{D/\D} V$. Since $O$ is contained in the projection of $W$, $W$ projects $\D$-dominantly onto $V$. Applying (2) of Theorem~\ref{charactheo} to $V$, $W$, $O_V$, $O_W$, we get a $K$-point $a\in O_V$ such that $\nabla a\in O_W$, as desired.

For the converse we assume that (ii) holds and aim to prove condition (2) of Theorem~\ref{charactheo}. In fact, it suffices to prove this statement in the case when $O_V=V$ and $O_W=W$ (see Remark~\ref{without}). We thus have irreducible affine $\D$-algebraic varieties $V$ and $W\subseteq \tau_{D/\D} V$ such that $W$ projects $\D$-dominantly onto $V$, and we show that there is $a\in V$ such that $\nabla a \in W$. Let $\L$ and $\Ga$ be characteristic sets of $\I(V/K)_\D$ and $\I(W/K)_\D$, respectively. Then, by Proposition~\ref{basic} $$\V^*(\Ga)=W\setminus\V(H_\Ga)\subseteq \tau_{D/\D} V\subseteq \V(f,\, \et f:\, f\in \L).$$ Since $W$ projects $\D$-dominantly onto $V$, $\V^*(\Ga)$ projects $\D$-dominantly onto $V$. Thus, by quantifier elimination for $DCF_{0,\ell}$ and the assumption $(K,\D)\models DCF_{0,\ell}$, there is $O$ a nonempty $\D$-open subset of $\V^*(\L)$ defined over $K$ such that the projection of $\V^*(\Ga)$ to $\V(\L)$ contains $O$. We are in the situation of condition~(ii), and there is $a\in \V^*(\L)\subseteq V$ such that $\nabla a\in \V^*(\Ga)\subseteq W$.
\end{proof}

\begin{remark}
What the above proof shows is that, in a $\D$-closed field, each instance of condition (2) of Theorem~\ref{charactheo} is equivalent to an instance of condition (ii) of Theorem \ref{characprime}. This is accomplished by passing from prime differential ideals to their characteristic sets, and from $\D$-dominant projections to containment of a nonempty $\D$-open set.
\end{remark}

By Fact~\ref{defchar}, the condition that ``$\L=\{f_1,\dots,f_s\}$ is a characteristic set of a prime $\D$-ideal of $K\{\x\}_\D$" is a definable property on the coefficients of $f_1,\dots,f_s$. That is, there is an $\LL_m$-formula which specifies for wich parameters of the $f_i$'s, over any model of $DF_{0,m}$, the set $\L$ is characteristic set of a prime $\D$-ideal. Thus, condition~(ii) of Theorem \ref{characprime} is first order expressible in the language of differential rings. Indeed, the scheme of axioms will consist of $DCF_{0,\ell}$ and an axiom for each choice of the shape of $\L$, $\Ga$, $O$ and $Q$ together with a formula specifying for which parameters these finite sets of polynomials satisfy condition~(ii). We therefore obtain a \emph{geometric first order axiomatization} for the theory $DCF_{0,\ell+1}$ that refers to the geometry of $DCF_{0,\ell}$.

We conclude by giving another geometric characterization in terms of relative D-varieties that is not on the face of it first order, but may be of independent interest. We show that the models of $DCF_{0,\ell+1}$ are precisely those differential fields that have sharp points on every relative D-variety (where we allow $K^{\Pi}$-linearly independent transformations of $\Pi$). 

\begin{theorem}
$(K,\Pi)\models DCF_{0,\ell+1}$ if and only if
\begin{itemize}
\item [(\dag)] For every basis $\D\cup\{D\}$ of $\operatorname{span}_{K^\Pi}\Pi$, if $(V,s)$ is an affine relative D-variety w.r.t. $D/\D$ defined over $K$ then $(V,s)^\#$ has a $K$-point.
\end{itemize}
\end{theorem}
\begin{proof}
Suppose $(K,\Pi)\models DCF_{0,\ell+1}$. By (1) of Theorem~\ref{lema1}, $(V,s)^\#$ is $\Pi$-dense in $V$ and hence nonempty. As $(K,\Pi)$ is existentially closed, $(V,s)^\#$ has a $K$-point.

For the converse, we assume condition (\dag) and prove that $(K,\Pi)\models DCF_{0,\ell+1}$. Suppose $\phi(x)$ is a conjunction of atomic $\mathcal{L}_{\ell+1}$-formulas over $K$, with a realisation $a=(a_1,\dots,a_n)$ in $\U$. We need to find a realisation of $\phi$ in $K$. We may assume that $a$ is in a $\Pi$-closure of $K$ and hence each $a_i$ is $\Pi$-algebraic over $K$. 

By Fact~\ref{diffal}, we can find a natural number $k>0$ and a basis $\D\cup\{D\}$ of $\operatorname{span}_{K^\Pi}\Pi$ such that $K\l a\r_\Pi=K\l a,\dots,D^k a\r_\D$. Hence, 
\begin{equation}\label{large}
D^{k+1}a=\frac{f(a\dots,D^k a)}{g(a,\dots,D^k a)}
\end{equation}
for some sequence $\frac{f}{g}$ of $\D$-rational functions over $K$. Let
\begin{displaymath}
c=\left(a,D a,\dots,D^k a,\frac{1}{g( a,D a\dots,D^{k} a)}\right). 
\end{displaymath}
Let $V$ be the $\D$-locus of $c$ over $K$ and $W$ the $\D$-locus of $\nabla c$ over $K$. A standard trick yields a sequence $s=(\operatorname{Id}, s')$ of $\D$-polynomials over $K$ such that $s(c)=\nabla c$. Since $c$ is a generic point of $V$ over $K$, we get that $(V,s)$ is a relative D-variety w.r.t. $D/\D$ defined over $K$. 

\

\noindent {\bf Claim.} $W=s(V)$.

\noindent Let $b\in V$. If $h$ is a $\D$-polynomial over $K$ vanishing at $(c,Dc)$, then $h(\cdot,s(\cdot))$ vanishes at $c$ and hence on all of $V$. So $(b,s(b))$ is in the $\D$-locus of $(c,D c)$ over $K$. That is, $(b,s(b))\in Y$. The other containment is clear since $\nabla c=s(c)\in s(V)$.

\

By condition (\dag), $(V,s)^\#$ has a $K$-point say $d$. Hence, $d\in V$ and $\nabla d=s(d)$. The previous claim yields that $(d,D d)\in W$.

Now, let $\rho(x)$ be the $\mathcal{L}_{\ell+1}$-formula over $K$ obtained from $\phi$ by replacing each $\d_i$, $i=1,\dots,\ell+1$, for $e_{i,1}\d_1+\cdots+e_{i,\ell+1}\d_{\ell+1}$, where $(e_{i,j})\in$ GL$_{\ell+1}(K^\Pi)$ is the transition matrix taking the basis $\D\cup\{D\}$ to the basis $\Pi$. By construction, $\phi^{(K,\Pi)}=\rho^{(K,\D\cup\{D\})}$. Thus it suffices to find a realisation of $\rho$ in $(K,\D\cup\{D\})$. We may assume that the $k$ of (\ref{large}) is large enough so that we can write
\begin{displaymath}
\rho(x)=\psi(x,\d_{\ell+1}x,\dots,\d_{\ell+1}^kx)
\end{displaymath}
where $\psi$ is a conjunction of atomic $\mathcal{L}_\ell$-formulas over $K$. Let
\begin{displaymath}
\chi(x_0,\dots,x_{k+1},y_0,\dots,y_{k+1}):=\psi(x_0,\dots,x_k)\land\left(\land_{i=1}^k x_i=y_{i-1} \right).
\end{displaymath}
Then $(\U,\D)\models \chi(c,D c)$, and so, as $(d,D d)\in W$, we have that $(\U,\D)\models \chi(d,D d)$. But since $d$ is a $K$-point, we get $(K,\D)\models \chi( d,D d)$. Writing the tuple $d$ as $(d_0,\dots,d_{k+1})$, we see that $d_0$ is a realisation of $\rho$ in $(K,\D\cup\{D\})$. This completes the proof.
\end{proof}

Notice that, in contrast with condition (2) of Theorem \ref{charactheo}, condition ($\dag$) does not involve differential irreducibility or differential dominance; however, it does involve the condition $s:V\to \tau_{D/\D}V$. Hence, since it is not known if the family of prolongations of a definable family of differential algebraic varieties is definable, it also remains open as to whether condition (\dag) is expressible in a first order way for $m>0$.

\

\chapter{The model theory of partial differential fields with an automorphism}\label{chapaut}

It has been known for over fifteen years that the theory of ordinary differential fields of characteristic zero with an automorphism admits a model companion. In other words, the class of existentially closed ordinary differential fields of characteristic zero equipped with an automorphism is an axiomatizable class. This theory has been studied extensively by Bustamante \cite{Bu}, \cite{Bu2}, \cite{Bu3}. However, the techniques used there to prove the existence of the model companion do not extend to partial differential fields with an automorphism. In this chapter we prove that the latter theory indeed has a model companion. After establishing some of its basic model theoretic properties, we will prove the canonical base property and the Zilber dichotomy for finite dimensional types.

Our approach is similar to the one from Section~\ref{thegeo}. We observe that the existentially closed partial differential fields with an automorphism are characterized by a certain geometric condition in terms of irreducible $\D$-varieties and $\D$-dominant projections (this already appears in \cite{GR}), very much in the spirit of the geometric axioms for $ACFA$. Of course, we run into the same difficulties as we did for Theorem~\ref{charactheo} when trying to express this geometric condition in a first order way: it is not known if differential irreducibility and differential dominance are definable conditions in definable families of differential algebraic varieties. As we did in Section~\ref{thegeo}, we bypass these issues by applying the differential algebraic machinery of characteristic sets of prime differential ideals.

\section{A review of $ACFA$}\label{resim}

Fields equipped with an automorphism (i.e., difference fields) enjoy some formal similarities to differential fields. While our eventual purpose is to consider a combination of differential and difference structure, we begin by reviewing in this section the model theory of difference fields. For a complete exposition of difference algebra the reader may consult \cite{Co}, and for the model theoretic results \cite{CH}.

We work in the language ${\mathcal L}_\s=\{0,1,+,-,\times,\s\}$ of rings expanded by a unary function symbol $\s$. 

\begin{definition}
By a \emph{difference field}, or \emph{$\s$-field}, we mean a field $K$ equipped with an automorphism $\s:K\to K$ (often one only asks for $\s$ to be injective, what we call here difference field is sometimes referred as inversive difference field). The \emph{fixed field} of a difference field $(K,\s)$ is by definition $K^\s=\{a\in K:\, \s a=a\}$. 
\end{definition}

Let $(K,\s)$ be a difference field. If $V$ is an affine algebraic variety defined over $K$ then $V^\s$ is the affine algebraic variety defined by $\{f^\s=0:f\in \I(V/K)\}$, where $f^\s$ is the polynomial obtained by applying $\s$ to the coefficients of $f$. 

\begin{fact}[\cite{CH}, \S 1]\label{axacfa}
The class of existentially closed difference fields is axiomatizable. This theory, denoted by $ACFA$, is axiomatized by the scheme of axioms expressing:
\begin{enumerate}
\item [(i)] $K\models ACF$
\item [(ii)] Suppose $V$ and $W$ are irreducible affine algebraic varieties defined over $K$ such that $W\subseteq V\times V^\s$ and $W$ projects dominantly onto both $V$ and $V^\s$. Then there is a $K$-point $a\in V$ such that $(a,\s a)\in W$.
\end{enumerate}
\end{fact}

In \cite{CH}, Chatzidakis and Hrushovski carry out a thorough analysis of the theory $ACFA$. In particular, they show that:

\begin{itemize}
\item The fixed field of $(K\s)\models ACFA$ is a pseudofinite field; that is, $K^\s$ is perfect, $K^\s$ has a unique extension of degree $n$ for all $n<\w$, and $K^\s$ is pseudo-algebraically closed (i.e., every absolutely irreducible variety defined over $K$ has a $K$-point).
\item If $A\subseteq K$, then $\operatorname{acl}(A)$ is the field theoretic algebraic closure of the $\s$-field generated by $A$.
\item The completions of $ACFA$ are determined by specifying the characteristic and the action of $\s$ on the algebraic closure of the prime field. Each completion of $ACFA$ admits elimination of imaginaries.
\end{itemize}

Let us restrict our attention to characteristic zero. Fix a sufficiently saturated model $(\U,\s)\models ACFA_0$. Given subsets $A,B,C$ of $\U$ define $A$ to be \emph{independent from $B$ over $C$}, denoted by $A\ind_C B$, if $\QQ\l A\cup C\r_\s$ is algebraically disjoint from $\QQ\l B\cup C\r_\s$ over $\QQ\l C\r_\s$. Here the notation $\QQ\l A\r_\s$ means the $\s$-field generated by $A$.

Then, as in the case of differentially closed fields, this captures Shelah's nonforking: if $C\subseteq B$ and $a$ is a tuple, $tp(a/B)$ is a nonforking extension of $tp(a/C)$ if and only if $a\ind_C B$. 

In \cite{CH}, Chatzidakis and Hrushovski proved that $ACFA_0$ is \emph{supersimple}; that is, they prove the following properties of independence:

\begin{enumerate}
\item \emph{(Invariance)} If $\phi$ is an automorphism of $\U$ and $A\ind_C B$, then $\phi(A)\ind_{\phi(C)}\phi(B)$.
\item \emph{(Local character)} There is a finite subset $B_0\subseteq B$ such that $A\ind_{B_0} B$.
\item \emph{(Extension)} There is a tuple $b$ such that $tp(a/C)=tp(b/C)$ and $b\ind_C B$.
\item \emph{(Symmetry)} $A\ind_C B$ if and only if $B\ind_C A$.
\item \emph{(Transitivity)} Suppose $C\subseteq B\subseteq D$, then $A\ind_C D$ if and only if $A\ind_C B$ and $A\ind_B D$.
\item \emph{(Independence theorem)} Suppose $C$ is an algebraically closed $\s$-field, and 
\begin{enumerate}
\item [(i)] $A$ and $B$ are supersets of $C$ with $A\ind_C B$, and
\item [(ii)] $a$ and $b$ are tuples such that $tp(a/C)=tp(b/C)$ and $a\ind_C A$ and $b\ind_C B$.
\end{enumerate}
Then there is a is a common nonforking extension $p\in S_n(A\cup B)$ of $tp(a/A)$ and $tp(b/B)$.
\end{enumerate}

Note that unlike the case of differentially closed fields we do not have stationarity (see Fact~\ref{prur}). The independence theorem takes its place. This corresponds to the fact that $ACFA_0$ is not a \emph{stable} theory.

Now, exactly as in the case of differentially closed field, we make use of the foundation rank associated to the partial ordering on types induced by forking. For historical reasons, this is called $SU$-rank in this case, rather than $U$-rank, but the definition is exactly as in Definition~\ref{Las}, and it also enjoys the appropriate analogues of the properties listed in Fact~\ref{yut}. 

A very consequential property of $ACFA_0$ proved in \cite{CH} is

\begin{fact}[Zilber dichotomy]
Suppose $K$ is an algebraically closed $\s$-field and $p=tp(a/K)$ is of $SU$-rank $1$. Then $p$ is either one-based or almost $\U^\s$-internal.
\end{fact}

Here $p=tp(a/K)$ is \emph{almost $\U^\s$-internal} means that there is a $\s$-field $L$ extension of $K$ with $a\ind_K L$ such that $a\in \operatorname{acl}(L,c)=L\l a\r_\s^{alg}$ for some tuple $c$ from $\U^\s$, and \emph{one-based} means that for every algebraically closed $\s$-field $L$ extension of $K$ we have that $Cb(a/L)\subseteq \operatorname{acl}(K,a)$ where $Cb(a/L)$ is the \emph{canonical base} of $tp(a/L)$ which we explain below.

The Zilber dichotomy can be seen as a consequence of the following result proved by Pillay and Ziegler in \cite{PZ}. First, given a complete type $tp(a/K)$, where $K$ is an algebraically $\s$-field, the \emph{canonical base} of $p$, denoted by $Cb(p)$, is a (finite) tuple from $K$ that comes from the study of supersimple theories in general and was introduced in $\cite{HKP}$. We will not give the precise definition here, but content ourselves with pointing out that in $ACFA_0$, up to interalgebraicity, it is the same thing as the minimal $\s$-field of definition of the $\s$-locus of $a$ over $K$. That is, if $V$ is the intersection of all solution sets of systems of difference polynomials over $K$ that contain $a$, which is again given by a (finite) system of such polynomials \cite{Co},  and $F$ is the minimal $\s$-field of definition of $V$, which exists by model theoretic considerations (by elimination of imaginaries, see for example Chapter 1.1 of \cite{Pi8}), then $\operatorname{acl}(Cb(p))=F^{alg}$.

\begin{fact}[Canonical base property, \cite{PZ}]
Suppose $a$ is a tuple such that $SU(a/K)<\w$ and $L$ is an algebraically closed $\s$-field extension of $K$. Then $tp(Cb(a/L)/a,K)$ is almost $\U^\s$-internal.
\end{fact}

Our goal in this chapter will be to extend the results discussed here to the case of (partial) differential-difference fields. In the case of a single derivation, this was done by Bustamante in \cite{Bu} and \cite{Bu3}. That is, consider the class of ordinary differential-difference fields $(K,\d,\s)$ where $K$ is a field of characteristic zero, $\d$ is a derivation, and $\s$ is an automorphism of $(K,\d)$ (i.e., an automorphism of $K$ such that $\d\s=\s\d$). Then the class of existentially closed ordinary differential-difference fields is axiomatizable, and the theory is denoted by $DCF_0A$. The axioms given in \cite{Bu} do not precisely generalize those of Fact~\ref{axacfa} above, and do not extend to give an axiomatization of existentially closed partial differential-difference fields $(K,\D,\s)$, where $\D$ is a set of $m>1$ commuting derivations on $K$ and $\s$ is an automorphism of $(K,\D)$. This will be accomplished in the next section, thereby introducing $DCF_{0,m}A$. Nevertheless, Bustamante does show that $DCF_0A$ exists, that its completions are described by the $\s$-field structure on $\QQ^{alg}$, that each completion eliminates imaginaries and is supersimple (i.e., for the appropriate definition of $\ind$ properties (1)-(6) stated above for $ACFA$, hold for $DCF_0A$), and satifies the canonical base property (and hence also the Zilber dichotomy). We will prove all these properties for $DCF_{0,m}A$, except that for the canonical base property (and Zilber dichotomy) we will restrict ourselves to \emph{finite dimensional} types.

\section{The model companion $DCF_{0,m}A$}

We work in the language $\LL_{m,\s}$ of differential rings with $m$ derivations expanded by a unary function symbol. We denote by $DF_{0,m,\s}$ the $\mathcal L_{m,\s}$-theory of differential fields of characteristic zero with $m$ commuting derivations and an automorphism that commutes with the derivations. We say that $(K,\D,\s)$ is a differential-difference field if $(K,\D,\s)\models DF_{0,m,\s}$. We are interested in axiomatizing the existentially closed differential-difference fields, and we will begin with a known geometric characterization of them.

First some notation. If $V$ is an affine $\D$-variety defined over $K$ then $V^\s$ is the affine $\D$-variety defined by $\{f^\s=0:f\in \I(V/K)_\D\}$, where $f^\s$ is the differential polynomial obtained by applying $\s$ to the coefficients of $f$. Here, when speaking of $\D$-varieties, we are implicitly working in a sufficiently large saturated $(\U,\D)\models DCF_{0,m}$, where our affine $\D$-varieties live in.

\begin{fact}[\cite{GR}, \S 2]\label{GR}
Let $(K,\D,\s)$ be a differential-difference field. Then, $(K,\D,\s)$ is existentially closed if and only if
\begin{enumerate}
\item [(i)] $(K,\D)\models DCF_{0,m}$
\item [(ii)] Suppose $V$ and $W$ are irreducible affine $\D$-varieties defined over $K$ such that $W\subseteq V\times V^\s$ and $W$ projects $\D$-dominantly onto both $V$ and $V^\s$. If $O_V$ and $O_W$ are nonempty $\D$-open sets of $V$ and $W$ respectively, defined over $K$, then there is a $K$-point $a\in O_V$ such that $(a,\s a)\in O_W$.
\end{enumerate} 
\end{fact}

\begin{proof}
Formally this is done in \cite{GR} for the case when $O_V=V$ and $O_W=W$. The argument there can easily be adapted to the more general statement presented above. Nevertheless, we give some details for the sake of completeness.

Suppose $(K,\D,\s)$ is existentially closed and $V$, $W$, $O_V$ and $O_W$ are as in condition~(ii). Let $(c,d)$ be a $\D$-generic point of $W$ over $K$; that is, $I(c,d/K)_\D=I(W/K)_\D$. Clearly $(c,d)\in O_W$ and, since $W$ projects $\D$-dominantly onto $V$ and $V^\s$, $c$ and $d$ are generic points of $V$ and $V^\s$, respectively, over $K$. Thus, $c\in O_V$. Because $DCF_{0,m}$ has quantifier elimination we have that $tp_{\D}(d/K)=\sigma(tp_\D(c/K))$, where $tp_\D(a/K)$ denotes the type of $a$ over $K$ in the language of differential rings. Hence, there is an automorphism $\s'$ of $\U$ extending $\s$ such that $\s'(c)=d$. Since $(K,\D,\s)$ is existentially closed we can find a point in $K$ with the desired properties.

Now suppose conditions (i) and (ii) are satisfied. Let $\phi(x)$ be a conjunction of atomic $\mathcal L_{m,\s}$-formulas over $K$. Suppose $\phi$ has a realisation $a$ in some differential-difference field $(F,\D,\s)$ extending $(K,\D,s)$. We may assume, and we do, that $(F,\D)$ is a $\D$-subfield of our ambient differentially closed field $(\U,\D)$. Let $$\phi(x)=\psi(x,\s x,\dots,\s^r x)$$ where $\psi$ is a conjunction of atomic $\mathcal L_{m}$-formulas over $K$ and $r>0$. Let $b=(a,\s a,\dots,\s^{r-1}a)$ and $V$ be the $\D$-locus of $b$ over $K$. Let $W$ be the $\D$-locus of $(b,\s b)$ over $K$. Let $$\chi(x_0,\dots,x_{r-1},y_0,\dots, y_{r-1}):= \psi(x_0,\dots,x_{r-1},y_{r-1})\land \left(\land_{i=1}^{r-1} x_i=y_{i-1}\right)$$ then $\chi$ is realised by $(b,\s b)$. Using the assumption that $(K,\D)\models DCF_{0,m}$, and since $(b,\s b)$ is a $\D$-generic point of $W$, $b$ is a $\D$-generic points of $V$ and $\s b$ is a $\D$-generic point of $V^\s$, over $K$, we have that $W$ projects $\D$-dominantly onto both $V$ and $V^\s$. Applying (2) with $V=O_V$ and $W=O_W$, there is $c$ in $V$ such that $(c,\s c)\in W$. Let $c=(c_0,\dots,c_{r-1})$, then $(c_0,\dots,c_{r-1},\s c_0,\dots, \s c_{r-1})$ realises $\chi$. Thus, $(c_0,\s c_0,\dots,\s^r c_0)$ realises $\psi$. Hence, $c_0$ is a tuple from $K$ realising $\phi$. This proves $(K,\D,\s)$ is existentially closed.
\end{proof}

\begin{remark}\label{othere}
As can be seen from the proof, it would have been equivalent in condition~(2) to take $O_V=V$ and $O_W=W$. Also note that, under the convention that  $DCF_{0,0}$ is $ACF_0$, when $m=0$ this is exactly the geometric axiomatization of $ACFA_0$.
\end{remark}

Condition (ii) is not, on the face of it, first order as it is not known if differential irreducibility and differential dominance are definable properties. We have seen this difficulty before, in Chapter 4, and we remedy the situation similarly; by working with characteristic sets of prime differential ideals. Recall that for $\L$ a characteristic set of a prime $\D$-ideal, $\V^*(\L):=\V(\L)\setminus\V(H_\L)$ (see definition \ref{sepin}).

\begin{theorem}\label{mainow}
Let $(K,\D,\s)$ be a differential-difference field. Then, $(K,\D,\s)$ is existentially closed if and only if
\begin{enumerate}
\item $(K,\D)\models DCF_{0,m}$
\item Suppose $\L$ and $\Ga$ are characteristic sets of prime differential ideals of $K\{x\}$ and $K\{x,y\}$, respectively, such that $$\V^*(\Ga)\subseteq \V(\L)\times \V(\L^\s).$$ Suppose $O$ and $Q$ are nonempty $\D$-open subsets of $\V^*(\L)$ and $\V^*(\L^\s)$ respectively, defined over $K$, such that $O\subseteq \pi_x(\V^*(\Ga))$ and $Q\subseteq \pi_y(\V^*(\Ga))$. Then there is a $K$-point $a\in \V^*(\L)$ such that $(a,\s a)\in \V^*(\Ga)$.
\end{enumerate}
\end{theorem}

Here, if $\L=\{f_1,\dots,f_s\}$, then $\L^\s=\{f_1^\s,\dots,f_s^\s\}$, and $\pi_x$ and $\pi_y$ denote the canonical projections from $\V(\L)\times\V(\L^\s)$ to $\V(\L)$ and $\V(\L^\s)$, respectively.

\begin{proof}
Suppose $(K,\D,\s)$ is existentially closed, and let $\L$, $\Ga$, $O$ and $Q$ be as in condition~(2). Then $\L$ is a characteristic set of the prime differential ideal $\P=[\L]:H_\L^\infty$ and, by Proposition~\ref{basic}, $\V^*(\L)=\V(\P)\setminus \V(H_\L)$. So $O_V:=\V^*(\L)$ is a nonempty $\D$-open subset of the irreducible affine $\D$-variety $V:=\V(\P)$. Similarly, $O_W:=\V^*(\Ga)$ is a nonempty $\D$-open subset of the irreducible affine $\D$-variety $W:=\V\left([\Ga]:H_\Ga^\infty\right)$. Note that both $O_V$ and $O_W$ are defined over $K$.

Next we show that $W\subseteq V\times V^\s$. By Proposition~\ref{basic} $$W\setminus \V(H_\Ga)=\V^*(\Ga)\subseteq \V(\L)\times \V(\L^\s),$$ so that by taking $\D$-closures $W\subseteq \V(\L)\times\V(\L^\s)$. Since $W$ is irreducible it must be contained in some $X\times Y$ where $X$ and $Y$ are irreducible components of $\V(\L)$ and $\V(\L^\s)$, respectively. Hence $O\subseteq\pi_x(W)\subseteq X$. On the other hand, $O\subseteq O_V\subset V$, and so $V=X$. Similarly, since $V^\s=\V([\L^\s]:H_{\L^\s}^{\infty})$ and working with $Q$ rather than $O$, we get $Y=V^\s$. Therefore, $W\subseteq V\times V^\s$. Since $O\subseteq \pi_x(W)$ and $Q\subseteq \pi_y(W)$, $W$ projects $\D$-dominantly onto both $V$ and $V^\s$. Applying (ii) of Fact~\ref{GR} to $V$, $W$, $O_V$, $O_W$, we get a $K$-point $a\in O_V$ such that $(a,\s a)\in O_W$, as desired.

For the converse we assume that (2) holds and aim to prove condition (ii) of Fact~\ref{GR}. In fact, it suffices to prove this statement in the case when $O_V=V$ and $O_W=W$ (see Remark~\ref{othere}). We thus have irreducible affine $\D$-varieties $V$ and $W\subseteq V\times V^\s$ such that $W$ projects $\D$-dominantly onto both $V$ and $V^\s$, and we show that there is a $K$-point $a\in V$ such that $(a,\s a)\in W$. Let $\L$ and $\Ga$ be characteristic sets of $\I(V/K)_\D$ and $\I(W/K)_\D$, respectively. Then, by Proposition~\ref{basic} $$\V^*(\Ga)=W\setminus\V(H_\Ga)\subseteq V\times V^\s\subseteq \V(\L)\times\V(\L^\s).$$ Since $W$ projects $\D$-dominantly onto $V$ and $V^\s$, $\V^*(\Ga)$ projects $\D$-dominantly onto both $V$ and $V^\s$. Thus, by quantifier elimination for $DCF_{0,m}$ and the assumption $(K,\D)\models DCF_{0,m}$, there are nonempty $\D$-open sets $O$ and $Q$ of $\V^*(\L)$ and $\V^*(\L^\s)$ respectively, defined over $K$, such that $O\subseteq \pi_x(\V^*(\Ga))$ and $Q\subseteq \pi_y(\V^*(\Ga))$. We are in the situation of condition~(2), and there is a $K$-point $a\in \V^*(\L)\subseteq V$ such that $(a,\s a)\in \V^*(\Ga)\subseteq W$.
\end{proof}

Because being a characteristic set of a prime $\D$-ideal is a definable property (see Fact\ref{defchar}), statement (2) of Theorem~\ref{mainow} is first-order expressible. We thus have a first order axiomatization of the existentially closed differential-difference fields. That is,

\begin{corollary}\label{mc}
The theory $DF_{0,m,\s}$ has a model companion. 
\end{corollary}

Henceforth we denote this model companion by $DCF_{0,m}A$.

\section{Basic model theory of $DCF_{0,m}A$}

In this section we present some of the model theoretic properties of the theory $DCF_{0,m}A$. Many of these results are consequences of the work of Chatzidakis and Pillay in \cite{CP} or more or less immediate adaptations of the arguments from \cite{Bu} or \S5 of \cite{MS}.

Let $(K,\D,\s)$ be a differential-difference field and $A\subseteq K$. The differential-difference field generated by $A$, denoted by $\QQ\l A\r_{\D,\s}$, is the smallest differential-difference subfield of $K$ containing $A$. Note that $\QQ\l A\r_{\D,\s}$ is simply the subfield of $K$ generated by $$\{\d_m^{e_m}\cdots \d_1^{e_1}\s^k a : \, a\in A, e_i< \omega, k\in \mathbb{Z}\}.$$ If $k$ is a differential-difference subfield of $K$, we write $k\l B\r_{\D,\s}$ instead of $\QQ\l k\cup B\r_{\D,\s}$.

\begin{proposition}\label{propert}
Let $(K,\D,\s)$ and $(L,\D',\s')$ be models of $DCF_{0,m}A$.
\begin{itemize}
\item [(i)] Suppose $K$ and $L$ have a common algebraically closed differential-difference subfield $F$, then $(K,\D,\s)\equiv_F(L,\D',\s')$. In particular, the completions of $DCF_{0,m}A$ are determined by the difference field structure on $\QQ^{alg}$.
\item [(ii)] If $A\subseteq K$, then $\operatorname{acl}(A)=\QQ\l A\r_{\D,\s}^{alg}$.
\item [(iii)] Suppose $F$ is a differential-difference subfield of $K$ and $a$, $b$ are tuples from $K$. Then $tp(a/F) = tp(b/F)$ if and only if there is an $F$-isomorphism from $(F\l a\r_{\D,\s}^{alg},\D,\s)$ to $(F\l b\r_{\D,\s}^{alg},\D,\s)$ sending $a$ to $b$.
\end{itemize}
\end{proposition}
\begin{proof}
In \cite{CP}, Chatzidakis and Pillay prove the following general results. Suppose $T$ is a stable theory with elimination of imaginaries and of quantifiers, and $T_0$ is the theory whose models are structures $(\M,\s)$ where $\M\models T$ and $\s\in Aut(\M)$. Assuming that $T_0$ has a model companion $TA$, they show that:
\begin{itemize}
\item If $(\M_1,\s_1)$ and $(\M_2,\s_2)$ are models of $TA$, containing a common substructure $(A,\s)$ such that $A=\operatorname{acl}_T(A)$ (i.e., $A$ is algebraically closed in the sense of $T$), then $(\M_1,\s_1)\equiv_A (\M_2,\s_2)$.
\item If $(\M,\s)\models TA$ and $A\subseteq M$, then $$\operatorname{acl}_{TA}(A)=\operatorname{acl}_T(\s^k(A):k\in \ZZ).$$
\end{itemize}
Hence, (i) and (ii) follow from this general results by specializing $T$ to $DCF_{0,m}$ and the fact that we have already shown that the model companion $TA=DCF_{0,m}A$ exists (Corollary~\ref{mc}).

We now prove (iii). Suppose $tp(a/F)=tp(b/F)$, working in a sufficiently saturated $(\U,\D,\s)\models DCF_{0,m}A$ extension of $(K,\D,\s)$, we get $\phi\in Aut(\U/F)$ such that $\phi(a)=b$. Then, the restriction $\displaystyle \phi|_{F\l a\r_{\D,\s}^{alg}}$ yields the desired isomorphism. Conversely, suppose $f$ is an $F$-isomorphism from $(F\l a\r_{\D,s}^{alg},\D,\s)$ to $(F\l b\r_{\D,\s}^{alg},\D,\s)$ such that $f(a)=b$. Let $(E,\D)\models DCF_{0,m}$ be a sufficiently large saturated extension of $(K,\D)$. Then there is a $\D$-automorphism $\rho$ of $E$ extending $f$. Let $L=\rho(K)$. Then $\rho$ induces a difference structure $\s'$ on $L$ such that $\rho$ is now an isomorphism between $(K,\D,\s)$ and $(L,\D,\s')$. Note that $F\l b\r_{\D,\s}^{alg}$ is a common substructure of $(K,\D,\s)$ and $(L,\D,\s')$. If $\phi(x)$ is an $\LL_{m,\s}$-formula over $F$, we get
$$K\models \phi(a) \iff L\models \phi(b) \iff K\models \phi(b)$$
where the second equivalence follows from part (i). We have shown that $tp(a/F)=tp(b/F)$, as desired.
\end{proof}

Let $(K, \D, \s)$ be a differential-difference field. We denote by $K^\s$ the \emph{fixed field} of $K$, that is $K^\s =\{a\in K:\, \s a=a\}$, and by $K^\D$ we denote the \emph{field of (total) constants} of $K$, that is $K^\D=\{a\in K: \, \d a=0 \text{ for all } \d\in \D\}$. We let $\C_K$ be the field $K^\s\cap K^\D$.

More generally, if $\D'$ is a set of linearly independent elements of the $\C_K$-vector space $\operatorname{span}_{\C_K}\D$, we let $K^{\D'}$ be the field of \emph{$\D'$-constants} of $K$, that is $K^{\D'} =\{a\in K:\, \d a=0\text{ for all } \d\in \D'\}$. In particular, $K^\emptyset = K$ and if $\D'$ is a basis of $\operatorname{span}_{\C_K}\D$ then $K^{\D'}=K^\D$. Note that both $K^\s$ and $K^{\D'}$ are differential-difference subfields of $(K,\D,\s)$. Also, $(K,\D',\s)$ is itself a differential-difference field.

In the following proposition $DCF_{0,0}A$ stands for $ACFA_0$.

\begin{proposition}
Let $(K, \D, \s)\models DCF_{0,m}A$, $\D_1$ and $\D_2$ disjoint sets such that $\D':=\D_1\cup\D_2$ forms a basis of $\operatorname{span}_{\C_K}\D$, and $r = |\D_1|$. Then 
\begin{enumerate}
\item $(K,\D',\s)\models DCF_{0,m}A$
\item $(K,\D_1,\s)\models DCF_{0,r}A$
\item $(K^{\D_2},\D_1,\s)\models DCF_{0,r}A$
\item $ K^{\D_2} \cap K^\s$ is a pseudofinite field. 
\end{enumerate}
\end{proposition}

\begin{proof} 
\

\noindent (1) It is easy to see that a set $V\subseteq K^n$ is $\D$-closed if and only if it is $\D'$-closed. Hence, irreducibility in the $\D$-topology is equivalent to irreducibility in the $\D'$-topology. Similarly a projection (onto any set of coordinates) is $\D$-dominant if and only if it is $\D'$-dominant. Therefore, each instance of the axioms of $DCF_{0,m}A$ (or rather of the characterization given by Fact~\ref{GR}) that needs to be checked for $\D'$ is true, as it is true for $\D$. 

\noindent (2) By (1) we may assume that $\D_1\subseteq \D$. One needs to show that every instance of the axioms of $DCF_{0,r}A$ that needs to checked for $\D_1$ is an instance of the axioms of $DCF_{0,m}A$ which we know is true for $\D$. For this, one needs to know that every $\D_1$-closed set irreducible in the $\D_1$-topology is also irreducible in the $\D$-topology. This is not obvious, rather it is a general version of Kolchin's Irreducibility Theorem. A proof of this can be found in Chap. 7 of \cite{Fre}.

\noindent (3) By (1) we may assume that $\D_1\cup\D_2= \D$. We show that $(K^{\D_2},\D_1,\s)$ is existentially closed. Let $\phi(x)$ be a quantifier free $\mathcal{L}_{\D_1,\s}$-formula over $K^{\D_2}$ with a realisation $a$ in some differential-difference field $(F, \Omega_1, \gamma)$ extending $(K^{\D_2} , \D_1, \s)$. Let $\Omega := \Omega_1\cup \{\rho_1,\dots, \rho_{m-r}\}$ where each $\rho_i$ is the trivial derivation on $F$. Hence, $(F, \Omega, \gamma)$ is a differential-difference field extending $(K^{\D_2} , \D, \s)$. Let $(L, \Omega, \gamma)$ be a model of $DCF_{0,m}A$ extending $(F, \Omega, \gamma)$. Since $K^{\D_2}$ is a common algebraically closed differential-difference subfield of $K$ and $L$, and $L \models \phi(a)$ and $\rho_i a = 0$ for $i=1,\dots,m-r$, by (i) of Proposition~\ref{propert} $K$ has a realisation $b$ of $\phi$ such that $b\in K^{\D_2}$. Thus, since $\phi$ is quantifier free, $K^{\D_2}\models \phi(b)$.

\noindent (4) By (3), we have that $(K^{\D_2}, \D_1, \s)\models DCF_{0,r}A$, and so, by (2), $(K^{\D_2},\s)\models ACFA_0$. Hence, the fixed field of $(K^{\D_2},\s)$, which is $K^{\D_2}\cap K^{\s}$, is pseudofinite.
\end{proof}

We now proceed to show that $DCF_{0,m}A$ is supersimple. Fix a sufficiently saturated $(U,\D,\s)\models DCF_{0,m}A$ and a differential-difference subfield $K$ of $\U$. Given subsets $A,B,C$ of $\U$ define $A$ to be \emph{independent from $B$ over $C$}, denoted by $A\ind_C B$, if $\QQ\l A\cup C\r_{\D,\s}$ is algebraically disjoint from $\QQ\l B\cup C\r_{\D,\s}$ from $\QQ\l C\r_{\D,\s}$. 

\begin{theorem}\label{imtop}
Every completion of $DCF_{0,m}A$ is supersimple; that is, independence satisfies the following properties:
\begin{enumerate}
\item (Invariance) If $\phi$ is an automorphism of $\U$ and $A\ind_C B$, then $\phi(A)\ind_{\phi(C)}\phi(B)$.
\item (Local character) There is a finite subset $B_0\subseteq B$ such that $A\ind_{B_0} B$.
\item (Extension) There is a tuple $b$ such that $tp(a/C)=tp(b/C)$ and $b\ind_C B$.
\item (Symmetry) $A\ind_C B$ if and only if $B\ind_C A$.
\item (Transitivity) Suppose $C\subseteq B\subseteq D$, then $A\ind_C D$ if and only if $A\ind_C B$ and $A\ind_B D$.
\item (Independence theorem) Suppose $C$ is an algebraically closed differential-difference field, and 
\begin{enumerate}
\item [(i)] $A$ and $B$ are supersets of $C$ with $A\ind_C B$, and
\item [(ii)] $a$ and $b$ are tuples such that $tp(a/C)=tp(b/C)$ and $a\ind_C A$ and $b\ind_C B$.
\end{enumerate}
Then there is $d\ind_C A\cup B$ with $tp(d/A)=tp(a/A)$ and $tp(d/B)=tp(b/B)$.
\end{enumerate}
\end{theorem}

In particular, by Theorem 4.2 of \cite{KP}, our notion of independence captures Shelah's nonforking. That is, if $C\subseteq B$ and $a$ is a tuple, $tp(a/B)$ is a nonforking extension of $tp(a/C)$ if and only if $a\ind_C B$.

\begin{proof}
As in the proof of Proposition~\ref{propert}, let us recall some general results of Chatzidakis and Pillay from \cite{CP}. Suppose $T$ is a stable theory with elimination of imaginaries and of quantifiers, and $T_0$ is the theory whose models are structures $(\M,\s)$ where $\M\models T$ and $\s\in Aut(\M)$. Assuming that $T_0$ has a model companion $TA$ and letting $(\M,\s)$ be a sufficiently saturated model of $TA$, they define a notion of independence in $TA$ as follows: if $A,B,C$ are subsets of $M$, then $A\ind_C B$ iff $\operatorname{acl}_T(\s^k(A\cup C):k\in \ZZ)$ is independent from $\operatorname{acl}_T(\s^k(B\cup C):k\in \ZZ)$ over $\operatorname{acl}_T(\s^k(C):k\in \ZZ)$ in the sense of $T$, here $\operatorname{acl}_T$ means algebraic closure in the sense of $T$. Then they show that this notion of independence in $TA$ satisfies (1)-(5) above and, under the assumption that $(C,\s)\models T_0$, also (6).

Hence, by specializing $T$ to $DCF_{0,m}$ and using the fact that in this case the model companion $TA=DCF_{0,m}A$ exists (Corollary~\ref{mc}), we get (1)-(5) for our notion of independence in $DCF_{0,m}A$. Note that the results of Chatzidakis and Pillay imply (6) under the additional assumption that $(C,\D)\models DCF_{0,m}$. However, as we now show, in the case of $DCF_{0,m}A$ this assumption is unnecessary. 

We follow very closely the strategy used by Moosa and Scanlon in Theorem 5.9 of \cite{MS} where they prove the analogous result for fields with free operators (in fact their strategy is an extension of the one used by Chatzidakis and Hrushovski in \cite{CH}). 

Let $c\models tp(a/C)$. It suffices to find $A',B'$ such that
\begin{enumerate}
\item [(i)] $c\ind_C A'\cup B'$
\item [(ii)] $A',c \models tp(A,a/C)$
\item [(iii)] $B',c \models tp(B,b/C)$
\item [(iv)] $A',B' \models tp(A,B/C).$
\end{enumerate}
Indeed, by (iv) there would be $\phi\in  Aut(\U/C)$ such that $\phi(A'\cup B') = A\cup B$, then if we set $d=\phi(c)$ we would get: from (i) that $d\ind_C A\cup B$, from (ii) that $tp(d/A)=tp(a/A)$, and from (iii) that $tp(d/B)=tp(b/B)$.

Since $tp(c/C)=tp(a/C)=tp(b/C)$, there are $A',B'$ satisfying (ii) and (iii). We may assume, by extensionality, that $A' \ind_{C,c}B'$, and, by transitivity, this implies (i). We are only missing (iv).

Let $K_0= \operatorname{acl}(A')\cdot \operatorname{acl}(B')$ and $K_1= \operatorname{acl}(A',c) \cdot \operatorname{acl}(B',c)$ where $\cdot$ means compositum in the sense of fields. Also, let $K_2= \operatorname{acl}(A',B')$. Then, $K_1$ and $K_2$ are field extensions of $K_0$. We will induce on $K_2$ a differential-difference field structure $(\D'=\{\d_1',\dots,\d_m'\},\s')$ such that there is an isomorphism over $C$ between $(K_2,\D',\s')$ and $(\operatorname{acl}(A,B),\D,\s)$. Let $\al,\beta\in Aut(\U,C)$ be such that $\al$ takes $A$ to $A'$ and $\beta$ takes $B$ to $B'$. Since $A\ind_C B$ and $A'\ind_C B'$ and $C$ is algebraically closed, we get that $\operatorname{acl}(A)$ and $\operatorname{acl}(A')$ are linearly disjoint from $\operatorname{acl}(B)$ and $\operatorname{acl}(B')$ over $C$, respectively. Hence, $\operatorname{acl}(A)\cdot\operatorname{acl}(B)$ is the field of fractions of $\operatorname{acl}(A)\otimes_C\operatorname{acl}(B)$, and similarly for $A'$ and $B'$. This implies that $\al\otimes\beta$ induces a field isomorphism over $C$ between $(\operatorname{acl}(A)\cdot\operatorname{acl}(B))^{alg}=\operatorname{acl}(A, B)$ and $(\operatorname{acl}(A')\cdot\operatorname{acl}(B'))^{alg}=\operatorname{acl}(A', B')=K_2$. This field isomorphism induces on $K_2$ the desired differential-difference field structure $(\D',\s')$, and it follows that $(K_0,\D,\s)$ is a substructure of $(K_2,\D',\s')$. 

We now use the following result proved by Chatzidakis and Hrushovski (see the remark following the generalized independence theorem in \cite{CH}) to obtain that $K_1$ and $K_2$ are linearly disjoint over $K_0$: If $F_1,F_2,F_3$ are algebraically closed fields extending an algebraically closed field $F$, with $F_3$ algebraically independent from $F_1\cdot F_2$ over $F$, then $(F_1\cdot F_3)^{alg}\cdot(F_2\cdot F_3)^{alg}$ is linearly disjoint from $(F_1\cdot F_2)^{alg}$ over $F_1\cdot F_2$. 

Now, since $K_1$ and $K_2$ are linearly disjoint over $K_0$, their compositum is the field of fractions of $K_1\otimes_{K_0}K_2$ and hence we can induce a differential-difference field structure $(\D''=\{\d_1'',\dots,\d_m''\},\s'')$ on $K_1\cdot K_2$ by letting:
$$\s''(a\otimes b)=\s(a)\otimes \s'(b)\; \text{ and }\; \d_i''(a\otimes b)=\d_i(a)\otimes b+a\otimes \d_i'(b)\; \text{ for } i=1,\dots,m.$$
It follows that $(K_1,\D,\s)$ and $(K_2,\D',\s')$ are substructures of $(K_1\cdot K_2,\D'',\s'')$. We can embed the latter in some $(L,\D^*,\s^*)\models DCF_{0,m}A$, and note that it shares with $(\U,\D,\s)$ the algebraically closed substructure $C$. Therefore, by Proposition~\ref{propert}~(i), $L\equiv_C \U$ and so, by saturation of $U$, $(L,\D^*,\s^*)$ is embeddable in $(\U,\D,\s)$ over $C$. Replace $A'$ and $B'$ with their images under such an embedding (this images still satisfy conditions (i)-(iii)). Then it follows, from Proposition~\ref{propert}~(iii) and the isomorphism between $(K_2,\D',\s')$ and $(\operatorname{acl}(A,B),\D,\s)$, that $A'$ and $B'$ satisfy condition (iv).
\end{proof}

We now aim towards proving elimination of imaginaries for $DCF_{0,m}A$. We will follow very closely the strategy used by Moosa and Scanlon in \cite{MS} to prove elimination of imaginaries for fields with free operators in characteristic zero\footnote{It is worth pointing out here that our differential-difference fields do not fit into the formalism of Moosa and Scanlon from \cite{MS} so that the results of this chapter are not a consequence of \cite{MS}.} (in fact their strategy is an extension of the one used by Chatzidakis and Hrushovski in \cite{CH} to prove elimination of imaginaries for $ACFA$). 

We will make use of the following notion of \emph{dimension}:

\begin{definition}
Let $a$ be a tuple from $\U$. We define $$\operatorname{dim}_K a= (trdeg_K (\T_r a):r<\w)$$ where $\T_r a=\{\d_m^{e_m}\cdots\d_1^{e_1}\s^k a: e_i, k\in\w \text{ and } e_1+\cdots+e_m+k\leq r\}$. We view $\operatorname{dim}_K a$ as an element of $\w^\w$ equipped with the lexicographical order.
\end{definition}

This dimension can be considered as an analogue of the Kolchin polynomial (see Section~\ref{polype}). It is not preserved under interdefinability; however, it is a good measure for forking:

\begin{lemma}\label{dimpo}
Suppose $a$ is a tuple from $\U$ and $L>K$ is a differential-difference field. Then $a\ind_K L$ if and only if $\operatorname{dim}_K a=\operatorname{dim}_L a$. 
\end{lemma}
\begin{proof}
It follows from (v) of Proposition~\ref{propert} that $a\ind_K L$ if and only if $K\l \s^k a:k\geq 0\r_\D$ is algebraically disjoint from $L$ over $K$. Hence, we have the following
\begin{eqnarray*}
a\ind_K L
&\iff& K\l \s^k a: k\geq 0 \r_\D \text{ is algebraically disjoint from } L \text{ over } K \\
&\iff& K(\T_r a) \text{ is algebraically disjoint from } L \text{ over } K \text{ for all } r<\w \\
&\iff&  trdeg_K (\T_r a)=trdeg_L (\T_r a) \text{ for all } r<\w \\
&\iff& \operatorname{dim}_K a=\operatorname{dim}_L a.
\end{eqnarray*}
\end{proof}

\begin{proposition}
Every completion of $DCF_{0,m}A$ eliminates imaginaries.
\end{proposition}
\begin{proof}
We will work in the saturated multi-sorted structure $\U^{eq}$. The unfamiliar reader is referred to Chapter 1.1 of \cite{Pi8} for definitions and basic results. Let $e\in \U^{eq}$. It suffices to show that there is a tuple $d$ from $\U$ such that $\operatorname{dcl}^{eq}(e)=\operatorname{dcl}^{eq}(d)$. Let $E=\operatorname{acl}^{eq}(e)\cap \U$.
\vspace{0.05in}

\noindent {\bf Claim.} $e\in \operatorname{dcl}^{eq}(E)$ 

\noindent {\it Proof of Claim.}
Let $f$ be a definable function (without parameters) and $a$ a tuple from $\U$ such that $e=f(a)$. We first show that there exists a tuple $c$ from $\U$ such that $f(c)=f(a)$ and $c\ind_E a$. By Lemma 1.4 of \cite{EH}, there exists $b\models tp(a/E,e)$ such that $$\operatorname{acl}^{eq}(E,a)\cap \operatorname{acl}^{eq}(E,b)\cap \U=E.$$ Let us prove that $b$ can be chosen such that $\operatorname{dim}_{E\l a\r_{\D,\s}}b$ is maximal (in the lexicographic order). For each $r$, choose $b_r$ so that the above holds and $trdeg_{E\l a\r_{\D,\s}} (\T_i b_r: i\leq r􏰃)$ is maximal possible. Let $n_r = trdeg_{E\l a\r_{\D,\s}}(\T_r b_r)$ and let $\Phi$ be the set of formulas (over $E\l a\r_{\D,\s}$) saying that
$$x \models tp(a/E,e),\qquad \operatorname{acl}^{eq}(E,a) \cap \operatorname{acl}^{eq}(E,x) \cap \U = E,$$ and 
$$ trdeg_{E\l a\r_{\D,\s}}(\T_r x) \geq n_r \text{ for each } r < \w.$$
The $b_r$'s witness that $\Phi$ is finitely satisfiable, and so, by saturation, $\Phi$ is satisfiable. Any realisation $b$ of $\Phi$ has the desired properties, i.e., $b\models tp(a/E,e)$, $\operatorname{acl}^{eq}(E,a)\cap \operatorname{acl}^{eq}(E,b)\cap \U = E$ and $\operatorname{dim}_{E\l a\r_{\D,\s}} b$ is maximal. 

Let $c \models tp(b/E,a)$ with $c\ind_{E,a} b$. Then $c\models tp(b/E,e)=tp(a/E,e)$ and so $f(c)=f(a)$. It remains to show that $c\ind_E a$. Since $c\ind_{E,a}b$, we have that $$\operatorname{acl}^{eq}(E,c) \cap \operatorname{acl}^{eq}(E,b) \cap \U \subseteq \operatorname{acl}^{eq}(E,a) \cap \operatorname{acl}^{eq}(E,b) \cap \U = E.$$ 
Letting $c'$ be such that $tp(b,c/E,e) = tp(a,c'/E,e)$ we have that $c' \models tp(a/E,e)$ and $\operatorname{acl}^{eq}(E,c') \cap \operatorname{acl}^{eq}(E,a) \cap \U = E$. Thus, by maximality, $\operatorname{dim}_{E\l a\r_{\D,\s}} c' \leq \operatorname{dim}_{E\l a\r_{\D,\s}} b$. Hence, as $\operatorname{dim}$ is automorphism invariant,	$\operatorname{dim}_{E\l b\r_{\D,\s}} c	\leq \operatorname{dim}_{E\l a\r_{\D,\s}}  b$. On the other hand, $$\operatorname{dim}_{E\l b\r_{\D,\s}}c \geq \operatorname{dim}_{E\l a, b\r_{\D,\s}}c= \operatorname{dim}_{E\l a\r_{\D,\s}}c= \operatorname{dim}_{E\l a\r_{\D,\s}}b$$ where the first equality is by Lemma~\ref{dimpo}. Hence, $\operatorname{dim}_{E\l a, b\r_{\D,\s}} c= \operatorname{dim}_{E\l b\r_{\D,\s}}$, and so, by Lemma~\ref{dimpo}, $c\ind_{E,b} a$. Since $c\ind_{E,a}b$ and $\operatorname{acl}^{eq}(E,a)\cap \operatorname{acl}^{eq}(E,b)\cap \U=E$, we get $c\ind_E a,b$. In particular, $c\ind_E a$, as desired.

Now, towards a contradiction, suppose $e\notin \operatorname{dcl}^{eq}(E)$. Then there is $a'\models tp(a/E)$ such that $f(a)\neq f(a')$. Since $tp(a'/E)=tp(a/E)$, there is $c'\models tp(a'/E)$ such that $f(c')=f(a')$ and $c'\ind_E a'$.  We may assume that $c'\ind_E c$. Now observe that we have the conditions of the independence theorem (see Fact~\ref{imtop}~(6)); that is, $E=\operatorname{acl}(E)$, $c\ind_E c'$, $tp(a/E)=tp(a'/E)$, $a\ind_E c$ and $a'\ind_E c'$. Hence, by the independence theorem, there is a $u\ind_E c,c'$ with $tp(u/E,c)=tp(a/E,c)$ and $tp(u/E,c')=tp(a'/E,c')$. But then $f(u)=f(c)=f(a)$ and $f(u)=f(c')=f(a')$, implying that $f(a)=f(a')$ which yields the contradiction. Hence, it must be the case that $e\in \operatorname{dcl}^{eq}(E)$.
This proves the claim.

Finally, let $b$ be a tuple from $E$ such that $e\in\operatorname{dcl}^{eq}(b)$. Since $b\in \operatorname{acl}^{eq}(e)$, $b$ has only finitely many $Aut(\U^{eq}/e)$-conjugates. Say $B=\{b_1,\dots,b_n\}$. Since $ACF_0$ admits elimination of imaginaries, there is a tuple $d$ from $\U$ such that any field automorphism of $\U$ will fix $d$ pointwise iff it fixes $B$ setwise .  We claim that $\operatorname{dcl}^{eq}(e)=\operatorname{dcl}^{eq}(d)$. Let $\phi\in Aut(\U^{eq}/e)$. Then, by definition of $B$, $\phi$ fixes $B$ setwise and hence $d$ pointwise. On the other hand, let $\phi \in Aut(\U^{eq}/d)$ and let $\psi(x,y)$ be an $\LL_{m,\s}^{eq}$-formula such that $e$ is the only realisation of $\psi(x,b)$ (this formula exists since $e\in \operatorname{dcl}^{eq}(b)$). Then $\phi(e)$ is the only realisation of $\psi(x,\phi(b))$; however, since $\phi(b)\in B$, $e$ also realises $\psi(x,\phi(b))$. This implies that $\phi(e)=e$, as desired.
\end{proof}

We end this section with the following description (up to $\operatorname{acl}$) of the canonical base in $DCF_{0,m}A$:

\begin{lemma}\label{minf}
Suppose $a$ is a tuple such that $\s a, \d a \in K(a)^{alg}$ for all $\d\in \D$. Let $L>K$ be an algebraically closed differential-difference field, $V$ the Zariski locus of $a$ over $L$, and $F$ the minimal field of definition of $V$. Then, $$Cb(a/L)\subseteq\operatorname{acl}(F,K)\quad \text{ and } \quad F\subseteq Cb(a/L).$$ In particular, $\operatorname{acl}(Cb(a/L), K)=\operatorname{acl}(F,K)$. 
\end{lemma}
\begin{proof}
We have not defined in general the canonical base in this thesis (we refer the reader to \cite{Ca} for the formal definition), but it suffices to say that by definiton $a\ind_{Cb(a/L)} L$ and that to prove this lemma we only need to prove that:
\begin{enumerate}
\item [(a)] $a\ind_{\operatorname{acl}(F,K)} L$, and
\item [(b)] $V$ is defined over $Cb(a/L)$.
\end{enumerate}
\noindent (a) As $F$ is the minimal field of definition of $V=loc(a/L)$, $a$ is algebraically disjoint from $L$ over $F$. Then, as $\s a, \d a\in K(a)^{alg}$ for all $\d\in \D$, $K\l a\r_{\D,\s}\subseteq K(a)^{alg}$ is algebraically disjoint from $L$ over $\operatorname{acl}(F, K)$. Hence, $a\ind_{\operatorname{acl}(F, K)} L$. 

\noindent (b) Since $a\ind_{Cb(a/L)} L$, then $\operatorname{dim}_{L} a=\operatorname{dim}_{Cb(a/L)}a$ (see Lemma~\ref{dimpo}). It follows that $a$ is algebraically disjoint from $L$ over $Cb(a/L)$. Hence, the Zariski locus of $a$ over $Cb(a/L)$ must also be $V$, and so $V$ is defined over $Cb(a/L)$.
\end{proof}

\section{Zilber dichotomy for finite dimensional types}\label{zidi}

In this section we prove the canonical base property, and consequently Zilber's dichotomy, for finite dimensional types in $DCF_{0,m}A$. Our proof follows the arguments given by Pillay and Ziegler in \cite{PZ}, where they prove the dichotomy for $DCF_0$ and $ACFA_0$ using suitable jet spaces (this is also the strategy of Bustamante in \cite{Bu3} to prove the dichotomy for finite dimensional types in $DCF_{0,1} A$).

Fix a sufficiently large saturated $(\U,\D,\s)\models DCF_{0,m}A$, and an algebraically closed differential-difference subfield $K$ of $\U$. We first recall the theory of jet spaces from algebraic geometry. We refer the reader to \S 5 of \cite{Moo2} for a more detailed treatment of this classical material. The reason we are interested in jet spaces, which are higher order analogues of tangent spaces, is that they effect a linearisation of algebraic varieties.

Let $V$ be an irreducible affine algebraic variety defined over $K$. Let $\displaystyle \U[V]=\U[x]/\I(V/\U)$ denote the coordinate ring of $V$ over $\U$. For each $a\in V$, let $$\M_{V,a}=\{f\in \U[V]: f(a)=0\}.$$ Let $r>0$, the \emph{$r$-th jet space of $V$ at $a$}, denoted by $j_r V_a$, is defined as the dual space of the finite dimensional $\U$-vector space $\M_{V,a}/\M_{V,a}^{r+1}$.

The following gives explicit equations for the $r$-th jet space and allows us to consider it as an affine algebraic variety.

\begin{fact}
Let $V\subseteq \U^n$ be an irreducible affine algebraic variety defined over $K$. Fix $r>0$. Let $\DD$ be the set of operator of the form $$\frac{\partial^s}{\partial x_{i_1}^{s_1}\cdots\partial x_{i_k}^{s_k}}$$ where $0<s\leq r$, $1\leq i_1<\cdots<i_k\leq n$, $s_1+\cdots+s_k=s$, and $0<s_i$. Let $a\in V$ and $d=|\DD|$. Then $j_r V_a$ can be identified with the $\U$-vector subspace $$\{(u_D)_{D\in \DD}\in \U^d: \sum_{D\in \DD}Df(a)u_D=0, \text{ for all } f\in \I(V/K)\}.$$ 
\end{fact}

Let $X\subseteq V$ be an irreducible algebraic subvariety and $a\in X$. The containment of $X$ in $V$ yields a canonical linear embedding of $j_r X_a$ into $j_r V_a$ for all $r$. We therefore identify $j_r X_a$ with its image in $j_r V_a$.

The following fact makes precise what we said earlier about linearisation of algebraic varieties.

\begin{fact}\label{thethe}
Let $X$ and $Y$ be irreducible algebraic subvarieties of $V$ and $a\in X\cap Y$. Then, $X=Y$ if and only if $j_r X_a=j_r Y_a$, as subspaces of $j_m V_a$, for all $r$.
\end{fact}

We want to develop a $(\D,\s)$ analogue of the notions of differential and difference modules from \cite{PZ}. 

\begin{definition}
By a \emph{$(\D,\s)$-module over $(\U,\D,\s)$} we mean a triple $(E,\Omega,\Sigma)$ such that $E$ is a finite dimensional $\U$-vector space, $\Omega=\{\dd_1,\dots,\dd_m\}$ is a family of additive endomorphisms of $E$ and $\Sigma$ is an additive automorphism of $E$ such that $$\dd_i(\al e)=\d_i (\al) e+\al \dd_i(e)$$ and $$\Sigma(\al e)=\s(\al)\Sigma(e)$$ for all $\al\in \U$ and $e\in E$, and the operators in $\Omega\cup\{\Sigma\}$ commute. If we omit $\s$ and $\Sigma$ we obtain Pillay and Ziegler's definition of a \emph{$\D$-module over $(\U,\D)$}. Similarly, if we omit $\D$ and $\Omega$ we obtain the definition of a \emph{$\s$-module over $(\U,\s)$}. 

\end{definition}

The following is for us the key property of $(\D,\s)$-modules.

\begin{lemma}
Let $(E,\Omega,\Sigma)$ be a $(\D,\s)$-module over $\U$. Let $$E^\#=\{e\in E: \Sigma(e)=e \text{ and } \dd(e)=0 \text{ for all } \dd\in \Omega\}.$$ Then $E^\#$ is a $\C_\U$-vector space (recall that $\C_\U=\U^\s\cap\U^\D$) and there is a $\C_\U$-basis of $E^\#$ which is also a $\U$-basis of $E$.
\end{lemma}
\begin{proof}
Since $\dd_1,\dots,\dd_m$ and $\Sigma$ are $\C_\U$-linear, $E^\#$ is a $\C_\U$-vector space. Let $\{e_1,\dots,e_d\}$ be a $\U$-basis of $E$. With respect to this basis, let $A_i$ be the matrix of $\dd_i$, $i=1,\dots, m$, and $B$ the matrix of $\Sigma$. By this we mean that $A_i$ is the matrix whose $j$-th column consists of the coefficients of the linear combination of $\dd_i(e_j)$ in terms of the basis, and similarly for $B$. Under the linear transformation that takes the basis $\{e_1,\dots,e_d\}$ to the standard basis of $\U^d$, the $(\D,\s)$-module $(E,\Omega,\Sigma)$ is transformed into the $(\D,\s)$-module $$(\U^d,\{\d_1+A_1,\dots,\d_m+A_m\},B\s).$$ It suffices to prove the result for this $(\D,\s)$-module. 
As $\Sigma$ is an additive automorphism of $E$, the matrix $B$ is invertible. Also, the commutativity of $\Omega\cup\{\Sigma\}$ yields:
\begin{equation}\label{uset}
\d_j A_i -\d_i A_j=[A_i,A_j], \quad i=1,\dots,m
\end{equation}
and 
\begin{equation}\label{useg}
B\s(A_i)=\d_i(B)+A_i B, \quad i=1,\dots,m.
\end{equation}
Since $\d_i(B^{-1})=-B^{-1}\d_i(B)B^{-1}$, the previous equation yields
\begin{equation}\label{usec}
B^{-1}A_i=\d_i(B^{-1})+\s(A_i)B^{-1}, \quad i=1,\dots,m.
\end{equation}
Now, note that $(\U^d)^\#=\{u\in \U^d: B \s u=u \text{ and } \d_i u+A_i u=0, i=1,\dots,m\}$, and thus it suffices to find a nonsingular $n\times n$ matrix $M$ over $\U$ such that $B\s(M)=M$ and $\d_i M+A_i M=0$. Let $X$ and $Y$ be $n\times n$ matrices of variables. It follows from (\ref{uset}), and Rosenfeld's criterion (see Fact~\ref{rosen}), that the set $\L=\{\d_i X+A_i X:i=1,\dots,m\}$ is a characteristic set of a prime $\D$-ideal of $\U\{X\}_\D$. Since in this case $H_\L=1$, the corresponding prime $\D$-ideal is simply
$$\P=[\d_i X+A_i X: i=1,\dots, m]\subset \U\{X\}_\D.$$ Then $V=\V(\P)$ is an irreducible $\D$-variety. Let $W$ be the irreducible $\D$-variety defined by $$X=BY \quad \text{ and } \quad Y\in V^\s.$$ Then, by (\ref{useg}), $W\subseteq V\times V^\s$. Clearly, $W$ projects onto $V^\s$ and, by (\ref{usec}), it also projects onto $V$. Hence, by Fact \ref{GR}, there is a nonsingular matrix $M$ over $\U$ such that $M\in V$ and $(M,\s M)\in W$. This $M$ satifies the desired properties. 
\end{proof}

Notice that if $\{e_1,\dots,e_d\}\subset E^\#$ is a $\U$-basis of $E$, which exists by the previous lemma, then, under the linear transformation that takes this basis to the standard basis of $\U^d$, the $(\D,\s)$-module $(E,\Omega,\Sigma)$ is transformed into the $(\D,\s)$-module $(\U^d, \D,\s)$.

\begin{remark}\label{coli}
Let $(E,\Omega)$ be a $\D$-module over $(\U,\D)$ and $E^*$ be the dual space of $E$. If we define the \emph{dual operators} $\Omega^*=\{\dd_1^*,\dots,\dd_m^*\}$ on $E^*$ by $$\dd_i^*(\lambda)(e)=\d_i(\lambda(e))-\lambda(\dd_i(e))$$ for all $\lambda \in E^*$ and $e\in E$, then $(E^*,\Omega^*)$ becomes a $\D$-module over $\U$. Indeed, by Remark~3.3 of \cite{PZ}, $(E^*,\dd_i^*)$ is a $\{\d_i\}$-module. Hence, all we need to verify is that the $\dd_i^*$'s commute: 
\begin{eqnarray*}
\dd^*_j(\dd^*_i(\lambda))(e)
&=& \d_j(\d_i(\lambda(e)))-\d_j(\lambda(\dd_i(e)))-\d_i(\lambda(\dd_j(e)))+\lambda(\dd_i(\dd_j(e))) \\
&=& \d_i(\d_j(\lambda(e)))-\d_i(\lambda(\dd_j(e)))-\d_j(\lambda(\dd_i(e)))+\lambda(\dd_j(\dd_i(e))) \\
&=& \dd^*_i(\dd^*_j(\lambda))(e).
\end{eqnarray*}
\end{remark}

We now describe a natural $(\D,\s)$-module structure on the jet spaces of an algebraic D-variety equipped with a finite-to-finite correspondence with its $\s$-transform.

By a D-variety we will mean a relative D-variety w.r.t. $\D/\emptyset$ in the sense of Section~\ref{relDvar}. That is, the classical non relative notion studied by Buium \cite{Buium}. Let $(V,s)$ be an irreducible affine D-variety defined over $K$. Then we can extend the derivations $\D$ to the coordinate ring $\U[V]$ by defining $\d_i(x)=s_i(x)$ where $s=(\operatorname{Id},s_1,\dots,s_m)$ and $x=(x_1,\dots,x_n)$ are the coordinate functions of $\U[V]$. The integrability condition (\ref{rel1}) shows that these extensions commute with each other. Indeed, we have
$$\d_j\d_i(x)=\d_js_i(x)=d_{\d_j/\emptyset} s_i(x,s_j(x))=d_{\d_i/\emptyset}s_j(x,s_i(x))=\d_is_j(x)=\d_i\d_j(x).$$
Hence, $(\U[V],\D)$ becomes a $\D$-ring (having a $\D$-ring structure on the coordinate ring is the approach of Buium to D-varieties \cite{Buium}).

Let $a\in (V,s)^\#$. Then, $\M_{V,a}^r$ is a $\d_i$-ideal of the $\d_i$-ring $\U[V]$ for all $r>0$. This is shown explicitly in Lemma 3.7 of \cite{PZ}. Hence, $(\M_{V,a}/\M_{V,a}^r,\D)$ becomes a $\D$-module over $(\U,\D)$. By Remark \ref{coli}, $(j_r V_a,\D^*)$ is a $\D$-module over $(\U,\D)$. 

Suppose now $a\in (V,s)^\#$ is a generic point of $V$ over $K$ and $\s a\in K(a)^{alg}$. Let $W$ be the Zariski locus of $(a,\s a)$ over $K$. Then $W\subseteq V\times V^\s$ projects dominantly and generically finite-to-one onto both $V$ and $V^\s$. Moreover, for each $r>0$, $j_r W_{(a,\s a)}\subseteq j_r V_a\times j_rV^\s_{\s a}$ is the graph of an isomorphism $f:j_r V_a\to j_r V^\s_{\s a}$ and the map $\s^*=f^{-1}\comp \s$ equips $j_r V_a$ with the structure of a $\s$-module over $(\U,\s)$ (see Lemma 4.3 of \cite{PZ} for details). Furthermore, Lemma 4.4 of \cite{Bu3} shows that $(j_r V_a,\d_i^*,\s^*)$ is a $(\d_i,\s)$-module over $(\U,\d_i,\s)$ for all $i=1,\dots,m$. Thus, since we have already seen that the dual operators $\D^*$ commute, $(j_r V_a,\D^*,\s^*)$ is a $(\D,\s)$-module over $(\U,\D,\s)$.

\begin{remark}
Let $V$ be an irreducible affine D-variety defined over $K$ and suppose that $a\in (V,s)^\#$ is a generic point of $V$ over $K$ such that $\s a\in K(a)^{alg}$. Suppose $L>K$ is an algebraically closed differential-difference field and $W$ is the Zariski locus of $a$ over $L$. Then $(W,s|_W)$ is a $D$-subvariety of $(V,s)$ and, under the identification of $j_r W_a$ as a subspace of $j_r V_a$, we have that $j_r W_a$ is a $(\D,\s)$-submodule of $(j_r V_a,\D^*,\s^*)$. Indeed, by Lemma 4.7 of \cite{Bu3}, $j_r W_a$ is a $(\d_i,\s)$-submodule of $(j_r V_a,\d_i^*,\s^*)$ for all $i=1,\dots,r$.
\end{remark}

\begin{definition}
A type $p=tp(a/K)$ is said to be \emph{finite dimensional} if the transcendence degree of $K\l a\r_{\D,\s}$ over $K$ is finite. 
\end{definition}

\begin{lemma}\label{trey}
Suppose $p=tp(a/K)$ is finite dimensional. Then there is an irreducible affine D-variety $(V,s)$ over $K$ and a generic point $c\in (V,s)^\#$ of $V$ over $K$ such that $K\l a\r_{\D,\s}=K\l c\r_{\D,\s}$ and $\s c\in K(c)^{alg}$.
\end{lemma}

\begin{proof}
Since $p$ is finite dimensional then there is $s< \w$ such that $K\l a\r_{\D,\s}\subseteq K(\T_s a)^{alg}$ where (as in Definition~\ref{dimpo})
$$\T_s a=\{\d_m^{e_m}\cdots \d_1^{e_1}\s^k a: e_i,k\in \w \text{ and } e_1+\cdots+e_m+k\leq s\}.$$
In particular, $\s^{s+1} a\in K(\T_s a)^{alg}$. Hence, if we let $b=(a,\s a,\dots,\s^s a)$, then $K\l a\r_{\D,\s}=K\l b\r_{\D,\s}$ and $\s b\in K\l b\r_{\D}^{alg}$. Also, since $K\l b\r_\D\subseteq K(\T_s b)$, we have that $\D$-type$(b/K)=0$. So, by Proposition~\ref{finteo}, there is an irreducible affine D-variety $(V,s)$ over $K$ and a generic point  $c\in (V,s)^\#$ of $V$ over $K$ such that $K\l b\r_{\D}=K(c)$. Hence, $K\l a\r_{\D,\s}=K\l c\r_{\D,\s}$ and $\s c\in K(c)^{alg}$.
\end{proof}

Let us make explicit what the notion of almost internality means in $DCF_{0,m}A$. We say that $p=tp(a/K)$ is \emph{almost $\C_\U$-internal} if there is a differential-difference field extension $L$ of $K$ with $a\ind_K L$ such that $a\in \operatorname{acl}(L,c)=L\l c\r_{\D,\s}^{alg}$ for some tuple $c$ from $\C_\U$.

We are now in position to prove the canonical base property for finite dimensional types.

\begin{theorem}\label{cbp}
Suppose $tp(a/K)$ is finite dimensional and $L>K$ is an algebraically closed differential-difference field. Let $c$ be a tuple such that $\operatorname{acl}(c)=\operatorname{acl}(Cb(a/L))$. Then $tp(c/K\l a\r_{\D,\s})$ is almost $\C_\U$-internal.
\end{theorem}
\begin{proof}
We may replace $a$ by anything interdefinable with it over $K$. Hence, by Lemma~\ref{trey}, we may assume that $\s a\in K(a)^{alg}$ and that there is an irreducible affine D-variety $(V,s)$ defined over $K$ such that $a\in (V,s)^\#$ is a generic point of $V$ over $K$. Let $W$ be the locus of $a$ over $L$. Then $(W,s|_W)$ is a $D$-subvariety of $(V,s)$ and $a\in (W,s|_W)^\#$. By Lemma~\ref{minf}, if $d$ is a tuple generating the minimal field of definition of $W$ then $\operatorname{acl}(d,K)=\operatorname{acl}(c,K)$. Thus, it suffices to prove that $tp(d/K\l a\r_{\D,\s})$ is almost $\C_\U$-internal.

Consider the $(\D,\s)$-module $(j_r V_a, \D^*,\s^*)$ and recall that $j_r W_a$ is a $(\D,\s)$-submodule. For each $r$, let $b_r$ a $\C_\U$-basis of $j_r V_a^\#$ which is also a $\U$-basis of $j_r V_a$. Let $\displaystyle B=\cup_{r=1}^{\infty}b_r$, we may choose the $b_r$'s such that $d\ind_{K\l a\r_{\D,\s}} B$. The basis $b_r$ yields a $(\D,\s)$-module isomorphism between $(j_r V_a,\D^*,\s^*)$ and $(\U^{d_r},\D,\s)$ which therefore takes $j_r W_a$ into a $(\D,\s)$-submodule $S_r$ of $(\U^{d_r},\D,\s)$. We can find a $\C_\U$-basis $e_r$ of $S_r^\#\subseteq \C_\U^{d_r}$ which is also a $\U$-basis of $S_r$, so $S_r$ is defined over $e_r\subset\C_\U^{d_r}$. Let $\displaystyle E=\cup_{r=1}^{\infty}e_r$. 

It suffices to show that $d\in \operatorname{dcl}(a,K,B,E)$. To see this, let $\phi$ be an automorphism of $(\U,\D,\s)$ fixing $a,K,B,E$ pointwise. Since $j_r V_a$ is defined over $K\l a \r_{\D,\s}$, then $\phi(j_r V_a)=j_rV_a$. Also, as each $S_r$ is defined over $E$ and the isomorphism between $S_r$ and $j_r W_a$ is defined over $B$, $\phi(j_rW_a)=j_r W_a$ for all $r>0$. Since $V$ is defined over $K$ and $\phi$ fixes $a$ pointwise, Fact~\ref{thethe} implies that $\phi(W)=W$. But $d$ is in the minimal field of definition of $W$, thus $\phi$ fixes $d$ pointwise, as desired.
\end{proof}

\begin{remark}
Even though we did not mentioned it explicitly in the proof of Theorem~\ref{cbp}, the key construction in it is that of
\begin{eqnarray*}
(j_r W_a,\D^*,\s^*)^\#
&=& \{\lambda \in j_r W_a: \s^*(\lambda)=\lambda \text{ and }\d_i^*(\lambda)=0, i=1,\dots,m\} \\
&=& \{\lambda \in j_r W_a: \s(\lambda)=f(\lambda) \text{ and } \d_i\comp \lambda=\lambda\comp \d_i, i=1\dots,m\},
\end{eqnarray*}
which is a finite dimensional $\C_\U$-vector space. We could define this vector space to be the $r$-th $(\D,\s)$-jet space at $a$ of the $(\D,\s)$-locus of $a$ over $L$. But $(\D,\s)$-jet spaces for arbitrary $(\D,\s)$-varieties have already been defined, they are a special case of the generalized Hasse-Schmidt jet spaces of Moosa and Scanlon from \cite{Moo3}. In this context, namely $a$ is a $\D$-generic point of the sharp points of an algebraic D-variety over $L$ and $\s a \in L(a)^{alg}$, the two constructions agree. A proof of this in the differential case appears at the end of \S4.3 of \cite{Moo3}. We do not include a proof of the general case because it is not necessary for our results.
\end{remark}

As a consequence of the canonical base property, we can prove Zilber dichotomy for finite dimensional types. But first let us give the definition of a one-based type in the context of $DCF_{0,m}A$. A type $p=tp(a/K)$ is said to be \emph{one-based} if for every algebraically closed differential-difference field extension $L$ of $K$ we have that $Cb(a/L)\subseteq \operatorname{acl}(K,a)=K\l a\r_{\D,\s}^{alg}$.

\begin{corollary}\label{zildi}
Let $p$ be a finite dimensional complete type over $K$ with $SU(p)=1$. Then $p$ is either one-based or almost $\C_\U$-internal.
\end{corollary}
\begin{proof}
Suppose $p=tp(a/K)$ is not one-based. Then there is an algebraically closed differential-difference field $L>K$ such that $Cb(a/L)$ is not contained in $\operatorname{acl}(K,a)$. Let $c$ be a tuple such that $\operatorname{acl}(c)=\operatorname{acl}(Cb(a/L))$, then $c\notin \operatorname{acl}(K,a)$. By Theorem~\ref{cbp}, $tp(c/K\l a\r_{\D,\s})$ is almost $\C_\U$-internal. Hence, $c\in \operatorname{acl}(K,a,b,e)$ where $c\ind_{K,a} b$ and $e$ is a tuple from $\C_\U$. Therefore we get $c\Notind_{K,a,b} e$. 

Now, by a general fact about canonical bases is supersimple theories we get that $c\in \operatorname{acl}(a_1,\dots,a_s)$ where the $a_i$'s are independent realisations of $p$ (see Chap. 17 of \cite{Ca}). We may assume the $a_i$'s are independent from $a,b$ over $K$. We get that $$a_1,\dots,a_s\Notind_{K,a,b} e,$$ and so, since $SU(p)=1$, $a_{i+1}\in\operatorname{acl}(K,a,b,a_1,\dots,a_i,e)$ for some $i<s$. Since $$a_{i+1}\ind_{K}a,b,a_1,\dots,a_i$$ and $e$ is a tuple from $\C_\U$, $p$ is almost $\C_\U$-internal.
\end{proof}

\begin{remark}
The assumption in Corollary~\ref{zildi} that $p$ is finite dimensional should not be necessary, though a different proof is needed. To prove the general case one could work out the theory of arc spaces of Moosa, Pillay, Scanlon \cite{Moo} in the $(\D,\s)$ setting, and apply their weak dichotomy for regular types. This is done in the $(\d,\s)$ setting (i.e., ordinary differential-difference fields) by Bustamante in \cite{Bu4}, and the arguments there should work here as well.
\end{remark}

\chapter*{Appendix \\
\vspace{0.4in}
Some $U$-rank computations in $DCF_{0,m}$}\label{ucomp}
\addcontentsline{toc}{chapter}{Appendix: Some $U$-rank computations in $DCF_{0,m}$}

In this appendix, we present some results on computing and finding lower bounds for the $U$-rank of certain definable subgroups of the additive group. We work inside a sufficiently saturated $(\U,\D)\models DCF_{0,m}$ and fix a base $\D$-field $K$.

Let $p\in S_n(K)$, $\tau=\D$-type$(p)$ and $d=\D$-dim$(p)$. In \cite{Mc}, McGrail proved that 
\begin{equation}\label{upbo}
U(p)<\w^{\tau}(d+1).
\end{equation}
She also claimed the lower bound $\omega^{\tau}d\leq U(p)$. However, in \cite{So}, Suer shows that this lower bound is in general false. He showed that, in $DCF_{0,2}$, if $p$ is the generic type of $\d_2^k x=\d_1 x$ then $U(p)=\w$ while $\D$-type$(p)=1$ and $\D$-dim$(p)=k$ for all $k>0$. So far there is no known lower bound for the $U$-rank in terms of the $\D$-type; in fact, it is not even known if $U(p)=1$ implies $\D$-type$(p)=0$. 

On the other hand, it is generally known that the generic type of $DCF_{0,m}$ has $U$-rank $\w^m$. However, the only written argument (for $m>1$)  found by the author uses McGrail's incorrect lower bound. We record a correct (and very standard) proof.

\begin{lemma}\label{utheo}
If $p\in S_1(K)$ is the generic type of $\U$ over $K$ then $U(p)=\w^m$.
\end{lemma}
\begin{proof}
For $i=0,1,\dots,m$, let $\D_i=\{\d_1,\dots, \d_{m-i}\}$. Hence, $\U^{\D_0}=\U^\D$ and $\U^{\D_m}=\U$. We show, by induction on $i$, that $U(\U^{\D_i})=\w^i$. The base case holds trivially, as in this case we are dealing with the (total) constants $\U^\D$. Now suppose the result holds for $i$. Let $n>0$ and consider the definable group $$G=\{a\in \U^{\D_{i+1}}:\, \d_{m-i}^n a=0\},$$ which is an $n$ dimensional $\U^{\D_i}$-vector subspace of $\U^{\D_{i+1}}$. By the induction hypothesis $U(G)=\w^i n$. As $n$ was arbitrary and $G\subset \U^{\D_{i+1}}$, we get that $U(\U^{\D_{i+1}})\geq\w^{i+1}$. Now we need to show that $U(\U^{\D_{i+1}})\leq\w^{i+1}$; that is, we have to show that for every forking extension $tp(a/L)$ of the generic type of $\U^{\D_{i+1}}$ over $K$ we have that $U(a/L) < \w^{i+1}$. But, since $a\in \U^{\D_{i+1}}$ is $\D\setminus \D_{i}$-algebraic over $L$ (as $tp(a/L)$ forks over $K$), $\D$-type$(a/L)\leq i$. Hence, by (\ref{upbo}), $U(a/L)<\w^{i+1}$. This completes the induction.

Applying this to the case $i=m$ shows that the generic type of $\U$ over $K$ is $\w^m$.
\end{proof}

Let us observe that the idea here was to find a family of definable subgroups of the additive group $(\U,+)$ whose image under the $U$-rank map was a cofinal subset of $\w^m$. This cofinal subset was the set of ordinals of the form $\w^i n$. In the following example we observe that, through a more detailed argument, one can produce a family of subgroups of the additive group such that for every $\al<\w^m$ there is an element in the family whose $U$-rank is ``close'' to $\al$. 

\begin{example}
We construct a family $(G_r)_{r\in \NN^m}$ of definable subgroups of the additive group with the following property. For every $\al$ of the form $$\al= \w^n a_n+\w^{n-1}a_{n-1}+\cdots+ \w a_1+ a_0$$ with $n<m$, $a_i<\w$ and $a_n\neq 0$, there is $r\in \NN^m$ such that $$\al\leq U(G_r)<\w^n(a_n +1).$$
Let $G_r$ be defined by the homogeneous system of linear differential equations $$\d_1^{r_1+1}x=0, \, \d_2^{r_2+1}\d_1^{r_1}x=0,\, \dots,\, \d_{m-1}^{r_{m-1}+1}\d_{m-2}^{r_{m_2}}\cdots\d_1^{r_1}x=0, \, \d_{m}^{r_m}\d_{m-1}^{r_{m-1}}\cdots\d_1^{r_1}x=0.$$ We first show that $$U(G_r)\geq \w^{m-1}r_1+\w^{m-2}r_2+\cdots+\w r_{m-1}+r_m.$$ We prove this by transfinite induction on $r=(r_1, . . . , r_m)$ in the lexicographical order. The base case holds trivially. Suppose first that $r_m \neq 0$ (i.e., the succesor ordinal case). Consider the definable group homomorphism $f:(G_r,+)\to (G_r,+)$ given by $f(a)=\d_m^{r_m -1}a$. Then the generic type of the generic fibre of $f$ is a forking extension of the generic type of $G_r$. Since $f$ is a definable group homomorphism, the $U$-rank of the generic fibre is the same as the $U$-rank of $\operatorname{Ker}(f)=G_{r'}$ where $r'=(r_1,\dots,r_{m-1},r_m -1)$. By induction, $$U(G_{r'})\geq \w^{m-1}r_1+\w^{m-2}r_2+\cdots+\w r_{m-1}+(r_m -1).$$ Hence, $$U(G_{r})\geq \w^{m-1}r_1+\w^{m-2}r_2+\cdots+\w r_{m-1}+r_m.$$ 

Now suppose $r_m = 0$ (i.e., the limit ordinal case) and that $k$ is the largest such that $r_k \neq 0$. Let $n\in \w$ and $r' = (r_1,\dots,r_k -1,n,0,...,0)$. Then $G_{r'}\subset G_r$ and, by induction, $$U(G_{r'})\geq \w^{m-1}r_1+\w^{m-2}r_2+\cdots+\w^{m-k} (r_{k}-1)+\w^{m-k-1}n.$$ Since n was arbitrary, $$U(G_r)\geq \w^{m-1}r_1 +\w^{m-2}r_2 +\dots +\w^{m-k}r_k.$$ This complete the induction.

Next we show that if $k$ is the smallest such that $r_k>0$ then $\D$-type$(G)=m-k$ and $\D$-dim$(G)=r_k$. Let $tp(a/K)$ be the generic type of $G_r$. As $r_1=\cdots=r_{k-1}=0$, we have $\d_1 a=0,\,\dots,\,\d_{k-1} a=0,\, \d_k^{r_k+1}a=0$ and $\d_k^{r_k}a$ is $\D_{k}'$-algebraic over $K$ where $\D_k'=\{\d_{k+1},\dots,\d_m\}$. It suffices to show that $a,\d_k a,\dots, \d_k^{r_k-1}a$ are $\D_{k}'$-algebraically independent over $K$. Let $f\in K\{x_0,\dots,x_{k-1}\}_{\D_{k}'}$ be nonzero, and let $g(x)=f(x,\d_k x,\dots,\d_k^{r_k-1} x)\in K\{x\}_\D$. Then $g$ is a nonzero $\D$-polynomial over $K$ reduced with respect to all $f\in \I(G_r/K)_\D$. Thus, as $a$ is a generic point of $G_r$ over $K$, $$0\neq g(a)=f(a,\d_k a,\dots,\d_k^{r_k-1} a),$$ as desired.

Applying this, and (\ref{upbo}), in the case when $r_1=0, \dots,r_{m-n-1}=0, r_{m-n}=a_n, r_{m-n+1}=a_{n-1},\dots, r_{m-1}=a_1, r_m=a_0$, we obtain the desired inequalities $$\al \leq U(G_r)< \w^n(a_n+1).$$ Note that, in particular, each $G_r$ satisfies McGrail's lower bound.
\end{example}








\bibliographystyle{plain}
\cleardoublepage 
\phantomsection  
\renewcommand*{\bibname}{References}

\addcontentsline{toc}{chapter}{\textbf{References}}

\bibliography{uw-ethesis}


\end{document}

%% file: uw-ethesis-frontpgs.tex
\pagestyle{empty}
\pagenumbering{roman}

\begin{titlepage}
        \begin{center}
        \vspace*{1.0cm}

        \Huge
        {\bf Contributions to the model theory of partial differential fields }

        \vspace*{1.0cm}

        \normalsize
        by \\

        \vspace*{1.0cm}

        \Large
        Omar Le\'on S\'anchez \\

        \vspace*{3.0cm}

        \normalsize
        A thesis \\
        presented to the University of Waterloo \\ 
        in fulfillment of the \\
        thesis requirement for the degree of \\
        Doctor of Philosophy \\
        in \\
        Pure Mathematics \\

        \vspace*{2.0cm}

        Waterloo, Ontario, Canada, 2013 \\

        \vspace*{1.0cm}

        \copyright\ Omar Le\'on S\'anchez 2013 \\
        \end{center}
\end{titlepage}

\pagestyle{plain}
\setcounter{page}{2}

\cleardoublepage 

  \noindent
I hereby declare that I am the sole author of this thesis. 

\bigskip
  
  \noindent
This is a true copy of the thesis, including any required final revisions, as accepted by my examiners.

  \bigskip
  
  \noindent
I understand that my thesis may be made electronically available to the public.

\cleardoublepage


\begin{center}\textbf{Abstract}\end{center}

In this thesis three topics on the model theory of partial differential fields are considered: the generalized Galois theory for partial differential fields, geometric axioms for the theory of partial differentially closed fields, and the existence and properties of the model companion of the theory of partial differential fields with an automorphism. The approach taken here to these subjects is to relativize the algebro geometric notions of prolongation and D-variety to differential notions with respect to a fixed differential structure.

It is shown that every differential algebraic group which is not of maximal differential type is definably isomorphic to the sharp points of a relative D-group. Pillay's generalized finite dimensional differential Galois theory is extended to the possibly infinite dimensional partial setting.   Logarithmic differential equations on relative D-groups are discussed and the associated differential Galois theory is developed. The notion of generalized strongly normal extension is naturally extended to the partial setting, and a connection betwen these extensions and the Galois extensions associated to logarithmic differential equations is established.

A geometric characterization, in the spirit of Pierce-Pillay, for the theory $DCF_{0,\ell+1}$ (differentially closed fields of characteristic zero with $\ell+1$ commuting derivations) is given in terms of the differential algebraic geometry of $DCF_{0,\ell}$ using relative prolongations. It is shown that this characterization can be rephrased in terms of characteristic sets of prime differential ideals, yielding a first-order geometric axiomatization of $DCF_{0,\ell+1}$. 

Using the machinery of characteristic sets of prime differential ideals it is shown that the theory of partial differential fields with an automorphism has a model companion. Some basic model theoretic properties of this theory are presented: description of its completions, supersimplicity and elimination of imaginaries. Differential-difference modules are introduced and they are used, together with jet spaces, to establish the canonical base property for finite dimensional types, and consequently the Zilber dichotomy for minimal finite dimensional types.

\cleardoublepage


\begin{center}\textbf{Acknowledgements}\end{center}

Throughout my years in Waterloo, I have been fortunate enough to be immersed in a welcoming, friendly and selfless environment. I am profoundly thankful to the faculty, staff, and my fellow graduate students at the Pure Mathematics Department of the University of Waterloo. Foremost, I am deeply indebted to Rahim Moosa, whose knowledge and guidance have been essential to the development of this thesis. I thank him for his patience, wise advise, generosity and frankness, which have been invaluable to me. He has not only been an outstanding supervisor, but also a great mentor and exceptional example of dedication and enthusiasm. 

I am also compelled to thank the Model Theory community, which has been incredibly welcoming and supportive. In particular, I am thankful to Moshe Kamensky, Lynn Scow, Philipp Hieronymi, Lou van den Dries, Deirdre Haskell, Patrick Speissegger and Chris Miller for giving the opportunity to present my research, and for their hospitality. Special thanks go also to Anand Pillay and David Marker, whose own work, and clear exposition in both books and papers, have taught me so much, and inspired a great part of the work presented here. 

This thesis has also benefited from the many discussions and e-mail exchanges with James Freitag and Ronnie Nagloo. I thank James for being a familiar face in many of my talks, for many clear explanations about his ideas and for his infectious enthusiasm. I thank Ronnie for so many exciting conversations about math and future projects, and for his encouraging words.

I can certainly not overlook the generosity and support of the Differential Algebra community. I would like to thank the organizers and participants of the Kolchin seminar at CUNY and the DART V meeting. In particular, I thank Alexey Ovchinnikov, Phyllis Cassidy and William Sit for carefully listening while I talked about my research, and for giving me helpful advise and constructive feedback. 

Finally, I have not enough words of gratitude to express my recognition to Ale, but I especially thank her for embarking us in this academic adventure. I thank her for her patience, for her love and for her friendship. I am also indebted to her for her always smart and honest advise, and for teaching me, with her own example, that one must always put both brain and heart in even the tiniest of tasks. I lastly thank her for sharing and enjoying with me these five years, in spite of them being filled with \emph{derivations}, \emph{model companions}, \emph{characteristic sets} and, in some lucky days, \emph{stability}.

\cleardoublepage


\begin{center}\textbf{Dedication}\end{center}

\vspace{2in}

\begin{center}
To Ale, her family and my family, for their continuous support.
\end{center}

\cleardoublepage

\renewcommand\contentsname{Table of Contents}
\tableofcontents
\cleardoublepage
\phantomsection




\pagenumbering{arabic}